\documentclass[11pt]{article}

\usepackage[a4paper,margin=1in]{geometry}
\usepackage{amsmath,amsthm}
\usepackage{enumitem}

\usepackage{amssymb, amsfonts}
\usepackage{amsmath,comment,algorithm}
\allowdisplaybreaks
\usepackage{subcaption}

\usepackage{tikz}
\usetikzlibrary{shapes.arrows,shadings,shadows,decorations.pathreplacing,decorations.markings}

\usepackage{hyperref}           
\usepackage{url}
\usetikzlibrary{positioning,fit,calc}

\newtheorem{theorem}{Theorem}
\newtheorem{lemma}{Lemma}
\theoremstyle{remark}
\newtheorem{remark}{Remark}
\newcommand{\R}{\mathbb{R}}

\begin{document}

\title{Adjoint-Based Bayesian Uncertainty Quantification for PDE-Constrained Inverse Problems with Application to Semiconductor Imaging\thanks{B.~M.~Afkham is supported by the Research Council of Finland under grant numbers 371523, and H.~Yazdanian is supported by the Finnish Ministry of Education and Culture's Pilot for Doctoral Programmes (Pilot project Mathematics of Sensing, Imaging and Modelling). L.~Taghizadeh acknowledges support from an Elise-Richter grant, funded by the Austrian Science Fund (FWF), grant DOI 10.55776/V1000.}}

\author{
Hassan Yazdanian\thanks{Corresponding author. Email: \texttt{hassan.yazdanian@oulu.fi}. ORCID: \href{https://orcid.org/0000-0002-6372-348X}{0000-0002-6372-348X}.}
\and
Leila Taghizadeh\thanks{Email: \texttt{leila.taghizadeh@tuwien.ac.at}. ORCID: \href{https://orcid.org/0000-0001-6040-1787}{0000-0001-6040-1787}.}
\and
Babak Maboudi Afkham\thanks{Email: \texttt{babak.maboudi@oulu.fi}. ORCID: \href{https://orcid.org/0000-0003-3203-8874}{0000-0003-3203-8874}.}
}

\date{
\small
\textit{Hassan Yazdanian and Babak Maboudi Afkham}\\
Research Unit of Mathematical Sciences, University of Oulu,\\
Pentti Kaiteran katu 1, Linnanmaa, 90014 Oulu, Finland
\\[1ex]
\textit{Leila Taghizadeh}\\
Institute of Analysis and Scientific Computing, TU Wien,\\
Wiedner Hauptstraße 8--10, 1040 Wien, Austria
}

\maketitle

\begin{abstract}
We formulate a Bayesian framework for reconstructing doping profiles in pn-junction semiconductor devices from boundary flux measurements. The unknown doping field is modeled as a nearly piecewise-constant function characterized by an unknown interface and two unknown plateau concentrations, leading to a nonlinear ill-posed inverse problem governed by a Poisson--Boltzmann-type equation.

To represent this structure while enabling efficient gradient-based inference, we introduce a pushforward prior constructed by mapping a smooth latent Gaussian field with Mat\'ern-type covariance through a sigmoid transformation. The latent field is parameterized by a truncated Karhunen--Lo\`eve expansion, while the two piecewise-constant levels are represented by scalar plateau parameters. The resulting prior yields differentiable approximations of piecewise-constant fields with controllable interface sharpness. We establish well-posedness of the Bayesian formulation by proving Lipschitz continuity of the forward map and Hellinger stability of the posterior.

We then sample the posterior using the No-U-Turn Sampler (NUTS), with gradients computed by the adjoint method. Numerical experiments show that the combination of the proposed prior and NUTS provides more efficient posterior exploration than the dimension-robust preconditioned Crank--Nicolson (pCN) sampler, yielding one to two orders of magnitude larger effective sample sizes. In the known-plateau setting, the method reconstructs both planar and curved interfaces and provides spatially resolved uncertainty quantification (UQ). When the interface geometry and plateau concentrations are inferred jointly, posterior correlations reveal structural non-identifiability. These results demonstrate the effectiveness of combining pushforward priors with adjoint-gradient-based sampling for reliable UQ in nonlinear partial differential equation-constrained inverse problems with sharp interfaces.


\end{abstract}

\section{Introduction}
\label{intro}

Nanoscale semiconductor devices form the basis of many modern electronic technologies \cite{esseni2011nanoscale,lundstrom2006nanoscale}. Their operation relies on modifying the electrical properties of semiconducting materials, such as silicon, by introducing donor ($n$-type) or acceptor ($p$-type) impurities, a process known as doping. The resulting spatial distribution of impurities, called the \emph{doping profile}, determines the internal electric field, carrier transport, and boundary response of the device.

In pn-junction diodes, the doping profile is often characterized by two approximately constant plateau concentrations corresponding to the $p$- and $n$-type regions, separated by an interface \cite{shockley1949theory,shockley1951p,peiner1995doping,schubert1996delta}. Consequently, the doping profile can be viewed as a nearly \emph{piecewise-constant} field, characterized by the interface geometry and the plateau concentration in each region. Since the doping profile is not directly measurable after fabrication, it must be inferred from indirect boundary observations. This leads to a nonlinear and highly ill-posed PDE-constrained inverse problem governed by semiconductor drift--diffusion models or their equilibrium Poisson--Boltzmann-type reductions \cite{burger2001identification,cheng2011recovering,leitao2006semiconductors}.

% A substantial body of work has been devoted to the modeling and simulation of semiconductor devices \cite{markowich2012semiconductor,markowich1985stationary,selberherr1984analysis,jungel2009transport}, as well as to the numerical solution and stochastic modeling of charge transport problems \cite{taghizadeh2017optimal,khodadadian2018optimal}. In contrast, the corresponding inverse problems have received comparatively less attention, with only a limited number of works addressing parameter identification and Bayesian reconstruction in semiconductor models \cite{burger2001identification,leitao2006semiconductors,stadlbauer2019bayesian,taghizadeh2025bayesian,cossettini2019determination}. Owing to the nonlinear and ill-posed nature of these inverse problems, uncertainty quantification (UQ) plays a central role in assessing the reliability and stability of inferred parameters, particularly in the presence of measurement noise and incomplete boundary observations. Moreover, the resulting objective functions are typically non-convex, making it essential not only to compute parameter estimates but also to characterize the uncertainty associated with them. However, effective UQ for PDE-constrained inverse problems involving piece-wise constant parameters remains computationally challenging, since standard approaches such as Markov chain Monte Carlo (MCMC) methods often become prohibitively expensive due to the high dimensionality and repeated PDE solves required by the underlying model.

A substantial body of work has addressed the modeling and simulation of semiconductor devices \cite{markowich2012semiconductor,markowich1985stationary,selberherr1984analysis,jungel2009transport}, as well as numerical and stochastic methods for semiconductor transport models \cite{taghizadeh2017optimal,khodadadian2018optimal}. By contrast, inverse problems for semiconductor models have received less attention, with only a few works addressing parameter identification and doping profile reconstruction \cite{burger2001identification,leitao2006semiconductors,stadlbauer2019bayesian,taghizadeh2025bayesian,cossettini2019determination}. In this setting, uncertainty quantification (UQ) is essential for assessing the reliability of doping profiles reconstructed from noisy, incomplete boundary data.

Bayesian inversion provides a principled framework for UQ in ill-posed inverse problems by incorporating prior information through a \emph{prior} distribution and returning a probability distribution over the unknown parameters rather than a single point estimate \cite{stuart2010inverse,kaipio2005statistical,smith2024uncertainty}. This distribution describes the doping profile conditioned on the boundary measurements and is referred to as the \emph{posterior} distribution. Statistical quantities derived from the posterior, such as the posterior mean and variance, provide quantitative measures of uncertainty and enable assessment of the credibility of the reconstructed doping profile. These quantities can be estimated by sampling from the posterior distribution using Markov chain Monte Carlo (MCMC) methods. However, MCMC sampling can be computationally demanding for PDE-constrained inverse problems, since each posterior sample may require the solution of a nonlinear forward PDE. Moreover, successive MCMC samples are generally correlated, so a large number of PDE solves may be needed to obtain a sufficient number of effectively independent samples. This motivates the development of more efficient strategies for posterior exploration in Bayesian PDE-constrained inverse problems~\cite{martin2012stochastic,bui2014solving,beskos2017geometric}.

%Each posterior evaluation requires the solution of a nonlinear forward PDE, and gradient-free Markov chain Monte Carlo (MCMC) methods may require many highly correlated samples, especially in high-dimensional or strongly data-constrained settings~\cite{martin2012stochastic,bui2014solving,beskos2017geometric}.

% Adjoint methods have made gradient-based MCMC practical for many PDE-constrained inverse problems. They allow derivatives of the data-misfit or posterior density to be computed at a cost that is largely independent of the parameter dimension. They have been combined with Hamiltonian Monte Carlo (HMC), Riemannian manifold HMC, stochastic Newton MCMC, and function-space MCMC methods to improve posterior exploration in large-scale problems~\cite{martin2012stochastic,bui2014solving,beskos2017geometric,bui2016fem}. Applications include seismic imaging, wave propagation, and tomographic inverse problems~\cite{fichtner2019hamiltonian,koch2020adjoint,zunino2023hmclab}. In contrast, Bayesian UQ for inverse problems with sharp interfaces and approximately piecewise-constant parameters remains less explored.

Adjoint methods \cite{hinze2008optimization} provide one such strategy. They have made gradient-based optimization practical for many PDE-constrained inverse problems by enabling derivatives of the data-misfit or log-posterior to be computed at a cost that is largely independent of the parameter dimension. More recently, adjoint techniques have also been incorporated into gradient-based MCMC algorithms, including Hamiltonian Monte Carlo (HMC), Riemannian manifold HMC, stochastic Newton MCMC, and function-space MCMC, to improve the efficiency of posterior exploration in Bayesian PDE inverse problems~\cite{martin2012stochastic,bui2014solving,beskos2017geometric,bui2016fem}. While these methods have been successfully applied to problems such as seismic imaging, wave propagation, and tomographic inversion~\cite{fichtner2019hamiltonian,koch2020adjoint,zunino2023hmclab}, the use of adjoint-based MCMC for Bayesian PDE inverse problems remains comparatively limited relative to the widespread adoption of adjoint-based optimization methods. At the same time, the benefit of gradient-based sampling is problem dependent, especially for edge-preserving priors, where the posterior structure may require tailored sampling strategies~\cite{uribe2023horseshoe}.

%Adjoint methods have made gradient-based optimization practical for many PDE-constrained inverse problems, since they allow derivatives of the data-misfit or log-posterior to be computed at a cost that is largely independent of the parameter dimension. \bbk{Beside some recent advances for adjoint in combination with} Hamiltonian Monte Carlo (HMC), Riemannian manifold HMC, stochastic Newton MCMC, and function-space MCMC methods to improve posterior exploration problems~\cite{martin2012stochastic,bui2014solving,beskos2017geometric,bui2016fem}, with some applications in seismic imaging, wave propagation, and tomographic inverse problems~\cite{fichtner2019hamiltonian,koch2020adjoint,zunino2023hmclab}, \bbk{gradient-based MCMC methods for Bayesian PDE inverse problems is comparatively less widespread.}

% Motivated by these developments, we revisit the Bayesian formulation of the pn-junction inverse problem and reformulate it in a manner that enables the efficient use of gradient-based methods. In particular, we introduce a differentiable prior model that remains effective for representing piecewise-constant doping profiles. Furthermore, we demonstrate how adjoint-based gradient computations can be leveraged for posterior exploration and uncertainty quantification, through Hamiltonian Monte Carlo methods \cite{neal2011mcmc}, aspects that have received comparatively limited attention in the PDE-based inverse problem literature.

Motivated by these developments, we revisit the Bayesian formulation of the pn-junction inverse problem and reformulate it to enable efficient use of gradient-based sampling. A key difficulty is the choice of prior. Standard Gaussian priors, widely used in Bayesian inversion, are well-suited to smooth unknown fields but may oversmooth sharp interfaces. Conversely, piecewise-constant priors represent discontinuities more naturally, but their lack of differentiability makes them difficult to combine with gradient-based samplers.

% This motivates differentiable prior constructions that approximate piecewise-constant fields while retaining a parameterization suitable for adjoint-based inference. In particular, we introduce a pushforward prior that maps a smooth Gaussian latent field through a sigmoid transformation, yielding nearly piecewise-constant doping profiles with controllable interface sharpness. We then demonstrate how adjoint-based gradient computations can be leveraged for posterior exploration using the No-U-Turn Sampler (NUTS)~\cite{hoffman2014no}, an adaptive variant of HMC. Compared with other gradient-based samplers, such as Langevin-type methods~\cite{roberts1996exponential,roberts1998optimal}, NUTS can exploit Hamiltonian dynamics and adapt the integration path length automatically, which can improve exploration in high-dimensional and correlated posteriors.

This motivates differentiable prior constructions that approximate piecewise-constant fields while retaining a parameterization suitable for adjoint-based inference. In particular, we introduce a pushforward prior that maps a smooth Gaussian latent field with Mat'{e}rn-type covariance through a sigmoid transformation, yielding nearly piecewise-constant doping profiles with controllable interface sharpness. We then demonstrate how adjoint-based gradient computations can be leveraged for posterior exploration using the No-U-Turn Sampler (NUTS)~\cite{hoffman2014no}, an adaptive variant of HMC. Compared with Langevin-type methods \cite{roberts1996exponential,roberts1998optimal}, NUTS can generate longer-distance proposals and automatically adapt the integration path length, making it well suited to the posterior geometries considered here.

The analytical foundation of the proposed formulation is provided by a well-posedness result for the Bayesian inverse problem. Using the uniform ellipticity of the diffusion operator and the monotonicity of the nonlinear Poisson--Boltzmann term, we establish Lipschitz continuity of the forward map. This implies that the posterior is well-defined, absolutely continuous with respect to the prior, and stable under perturbations of the observed data in the Hellinger metric. Complementing this theoretical result, numerical experiments on synthetic data demonstrate the effectiveness of the proposed framework for posterior exploration in the pn-junction inverse problem. In particular, we show that, for the proposed prior and posterior geometry, the NUTS-based approach substantially reduces sample autocorrelation compared with the dimension-robust preconditioned Crank--Nicolson (pCN) sampler~\cite{cotter2013mcmc}. 

The main contributions of this work are summarized as follows:
\begin{itemize}
\renewcommand\labelitemi{$\bullet$}

\item A Bayesian PDE-constrained formulation for reconstructing piecewise-constant pn-junction doping profiles, including joint inference of interface geometry and plateau concentrations, with the latter entering both the interior source term and the Dirichlet boundary data.

\item A differentiable sigmoid pushforward prior with controllable interface sharpness, suitable for gradient-based posterior sampling.

\item An adjoint-based gradient formulation for the semiconductor inverse problem, enabling NUTS-based posterior sampling.

\item A well-posedness result proving Lipschitz continuity of the forward map and Hellinger stability of the posterior.

\item Numerical evidence that NUTS improves posterior exploration over pCN in effective sample size and sample autocorrelation.

\end{itemize}

The remainder of the paper is organized as follows. Sections~\ref{s:forward problem}--\ref{sec:Bayesian_formulation} introduce the semiconductor forward model, its finite element discretization, and the Bayesian inverse problem formulation. Section~\ref{G-inference} derives the adjoint-based gradient formulation and presents the NUTS algorithm. Numerical experiments comparing NUTS and pCN for several reconstruction scenarios are reported in Section~\ref{sec:numerical}, followed by concluding remarks in Section~\ref{s:conclusions}.

\section{The Mathematical Formulation of pn-Junction Diodes}\label{s:forward problem}

% In this section, we introduce a mathematical model for pn-junction diodes formulated as a nonlinear system of partial differential equations. In the steady-state regime, this system takes the form of a coupled nonlinear elliptic problem describing the electrostatic potential and carrier transport. We then describe the associated inverse problem, in which a voltage is applied on the boundary via electrodes and the resulting current response is measured. The goal is to reconstruct the underlying doping profile from these boundary measurements. 

In this section, we introduce a reduced equilibrium model for pn-junction doping reconstruction. The electrostatic potential is described by a nonlinear Poisson--Boltzmann-type elliptic equation, with the doping profile acting as a source term. The corresponding inverse problem consists of reconstructing the doping profile from boundary flux measurements over prescribed electrodes. Thus, the model is best viewed as a semiconductor-inspired nonlinear elliptic inverse source problem, rather than a full carrier-transport model.

%%%%%%%%%%%%%%%%%%%%%%%%%%%%%%%%%%%%%%%%%%%
\subsection{Doping Profile Reconstruction as an Inverse Problem}\label{ss:doping}
%%%%%%%%%%%%%%%%%%%%%%%%%%%%%%%%%%%%%%
We consider \(\Omega\subset \mathbb R^2\) to be a square spatial domain representing a two-dimensional cross-section of the pn-junction diode. The electrical properties of a semiconductor device are modified by introducing donor and acceptor impurities through a process known as doping. In a pn-junction diode, these impurities give rise to two regions with different effective charge concentrations, referred to as the n-type and p-type regions. We denote these subdomains by \(\Omega_{\mathrm n}\) and \(\Omega_{\mathrm p}\), respectively, with
\(\Omega = \Omega_{\mathrm n}\cup\Omega_{\mathrm p}\), and assume that they are separated by a continuous and differentiable interface \(\Gamma\), as illustrated in Fig.~\ref{fig:pn-junction}. The spatial distribution of ionized dopants is described by the doping profile \(c\in L^\infty(\Omega;\mathbb{R})\), which is modeled here as a piecewise-constant function of the form
\begin{equation}
c(\boldsymbol{x}) =
\begin{cases}
c_{\mathrm n} & \boldsymbol{x}\in \Omega_{\mathrm n},\\
c_{\mathrm p} & \boldsymbol{x}\in \Omega_{\mathrm p},
\end{cases}
\label{eq:doping_profile}
\end{equation}
where \(c_{\mathrm n},c_{\mathrm p}\in \mathbb R\) denote the plateau doping levels in the n- and p-type regions, respectively, with \(c_{\mathrm p}<0<c_{\mathrm n}\) under the present sign convention. The inverse problem considered in this work is to recover this doping profile, including the interface geometry and the plateau values, from indirect boundary flux measurements.

Let \(u \in H^1(\Omega;\mathbb{R})\) denote the dimensionless electrostatic potential. Under thermal equilibrium conditions, \(u\) is modeled by a nonlinear Poisson--Boltzmann-type equation \cite{taghizadeh2025bayesian}, in which the doping profile acts as a source term representing the net fixed charge due to ionized dopants. Assuming the boundary $\partial \Omega$ decomposes as $\partial \Omega = \partial \Omega_1 \cup \partial \Omega_2 \cup \partial \Omega_3$, The potential satisfies prescribed Dirichlet boundary values on \(\partial \Omega_2\) and \(\partial \Omega_3\), determined by the equilibrium potentials associated with the p- and n-type plateau values, while the remaining boundary \(\partial \Omega_1\) is assumed to be insulating. Thus, \(u\) is described by the nonlinear boundary value problem
\begin{equation} \label{eq:semi-PDE-strong}
\begin{aligned}
& -\nabla\cdot(\epsilon\nabla u(\boldsymbol x)) + 2\delta^2 \sinh(u(\boldsymbol{x}) ) = c(\boldsymbol{x}), \qquad &&\boldsymbol{x} \in \Omega, \\
& \nabla u\cdot \boldsymbol{n} = 0, &&\boldsymbol{x} \in \partial \Omega_1, \\
& u(\boldsymbol x) = f_1(\boldsymbol{x}) = \operatorname{arcsinh}(c_{\mathrm p}/2\delta^2), && \boldsymbol{x} \in \partial \Omega_2, \\
& u(\boldsymbol x) = f_2(\boldsymbol{x}) = \operatorname{arcsinh}(c_{\mathrm n}/2\delta^2), && \boldsymbol{x} \in \partial \Omega_3 .
\end{aligned}
\end{equation}
Here, $c$ is defined in \eqref{eq:doping_profile}, \(\epsilon>0\) denotes the permittivity of the semiconductor material, and \(\delta>0\) is a scaling parameter related to the intrinsic carrier density. 

The dependence of the Dirichlet data on the plateau doping levels is an important feature of the semiconductor model. In several earlier formulations of inverse doping problems, the doping profile is assumed to be known on the boundary \cite{burger2001identification,burger2004inverse}. Under such assumptions, the boundary values entering the forward problem are fixed. In contrast, when \(c_{\mathrm p}\) and \(c_{\mathrm n}\) are treated as unknowns, the Dirichlet data in \eqref{eq:semi-PDE-strong} become parameter-dependent and must be inferred consistently with the interior doping profile. This additional dependence will be accounted for in the Bayesian formulation and in the adjoint-based gradient computation below.

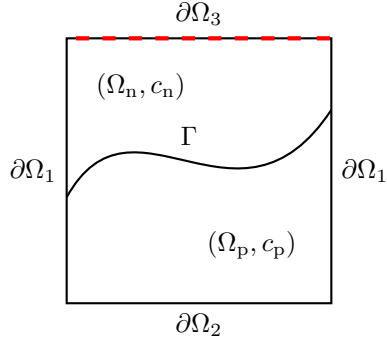
\begin{figure}
    \centering
    \begin{tikzpicture}[scale=3.5, every node/.style={font=\small}]

      % Domain
      \draw[thick] (0,0) rectangle (1,1);
    
      % Electrodes / measurement facets on top boundary
      \foreach \x in {0.06,0.16,...,0.96, 0.97}
      {
        \draw[red, ultra thick] (\x-0.025,1.0) -- (\x+0.025,1.0);
      }
    
      % Label for measurement facets
      % \node[red, above=5pt] at (0.5,1.0) {\scriptsize measurement facets};
    
      % Curved interface Gamma
      \draw[thick]
        (0,0.40)
        .. controls (0.25,0.85) and (0.65,0.20) ..
        (1,0.73)
        node[midway, above=2pt] {$\Gamma$};
    
      % Boundary labels
      \node[left] at (0,0.5) {$\partial\Omega_1$};
      \node[right] at (1,0.5) {$\partial\Omega_1$};
      \node[below] at (0.5,0) {$\partial\Omega_2$};
      \node[above=1pt] at (0.5,1) {$\partial\Omega_3$};
    
      % Region labels
      \node at (0.28,0.82) {$\left(\Omega_{\mathrm n}, c_{\mathrm n}\right)$};
      \node at (0.70,0.22) {$\left(\Omega_{\mathrm p}, c_{\mathrm p}\right)$};
    
    \end{tikzpicture}
    \caption{A two-dimensional pn-junction with a piecewise-constant doping function $c(\boldsymbol{x})$. The red segments on the top boundary indicate the measurement electrodes.}
    \label{fig:pn-junction}
\end{figure}

We now define the boundary observations used in the inverse problem. We consider a mathematical measurement operator defined in terms of the normal flux of the electrostatic potential on a prescribed part of the boundary. Specifically, measurements are taken on a collection of disjoint boundary segments \[
\{I_{\ell}\}_{\ell=1}^{N_{\mathrm{meas}}} \subset \partial\Omega_3,
\]
and the corresponding data vector
\(\boldsymbol{y}\in\mathbb{R}^{N_{\mathrm{meas}}}\) is defined by
\begin{equation} \label{eq:current}
    [\boldsymbol{y}]_\ell
    =
    \int_{I_\ell} \nabla u \cdot \boldsymbol{n} \, ds
    =
    \int_{x_{\ell}^1}^{x_{\ell}^2}
    \partial_y u(x,1)\, dx,
    \qquad
    \ell = 1,\dots, N_{\mathrm{meas}},
\end{equation}
where \(I_{\ell}=(x_{\ell}^1,x_{\ell}^2)\subset\partial\Omega_3\) denotes the
\(\ell\)-th measurement segment and
\(\boldsymbol{n}=(0,1)^T\) is the outward unit normal vector on the top
boundary. Accordingly, we define the observation operator
\begin{equation} \label{eq:observation_operator}
    \mathcal{O}(u)
    =
    \left\{
    \int_{I_1} \nabla u\cdot\boldsymbol{n}\,ds,
    \ldots,
    \int_{I_{N_{\mathrm{meas}}}} \nabla u\cdot\boldsymbol{n}\,ds
    \right\}
    \in \mathbb{R}^{N_{\mathrm{meas}}}.
\end{equation}

Let $\mathcal{W}$ denote the solution operator mapping a doping profile $c$ to the corresponding potential $u$, by solving \eqref{eq:semi-PDE-strong}. We define the forward operator as
\[
\mathcal{G}(c) := \mathcal{O} \circ \mathcal{W},
\]
which maps the doping function $c$ to the associated boundary measurements $\boldsymbol{y}$ of the normal current. These measurements are typically contaminated by noise, and we therefore assume an additive noise model of the form
\begin{equation} \label{eq:forward-model}
    \boldsymbol{y} = \mathcal{G}(c) + \boldsymbol{\varepsilon},
\end{equation}
where $\boldsymbol{\varepsilon} \in \mathbb{R}^{N_{\text{meas}}}$ represents a Gaussian measurement noise.

In the following section, we discretize \eqref{eq:semi-PDE-strong} using the finite element method (FEM) and define the corresponding discrete forward operator. This discrete setting then serves as the basis for formulating the problem within a statistical framework later in Section \ref{sec:Bayesian_formulation}.

\section{Finite Element Discretization of the Forward Problem} \label{sec:FEM}

In this section, we briefly describe the FEM approximation of the solution operator $\mathcal W$, used to construct the discrete forward operator.

Let $V := H^1(\Omega; \mathbb{R})$ and define the subspace $V_0 = \{v \in V : v = 0 \text{ on } \partial\Omega_2 \cup \partial\Omega_3\}$. Multiplying \eqref{eq:semi-PDE-strong} by $v\in V_0$ and integrating over $\Omega$ yields the weak form
\begin{equation} \label{eq:semi-weak}
    \int_{\Omega} \epsilon \nabla u \cdot \nabla v~d\boldsymbol x + 2\delta^2 \int_\Omega \sinh(u)v~d \boldsymbol x = \int_\Omega cv~d \boldsymbol{x}.
\end{equation}
Here, Green's identity has been used, and the boundary contributions vanish because $(\nabla u \cdot \boldsymbol{n})|_{\partial \Omega_1} = 0$ and $v|_{\partial\Omega_2 \cup \partial\Omega_3} = 0$. 

To impose the Dirichlet boundary conditions, we employ the lifting method \cite{quarteroni2006numerical}. Let $g \in V$ satisfy $g|_{\partial \Omega_2} = f_1$ and $g|_{\partial \Omega_3} = f_2$. We then introduce an auxiliary function \(w\in V_0\)
% which has vanishing trace on the Dirichlet part of the boundary, 
% We then introduce an auxiliary function $w \in V$ with vanishing trace, i.e., $w|_{\partial \Omega} = 0$, 
and define the change of variables $u = w + g$. Substituting this relation into \eqref{eq:semi-weak} gives the lifted weak formulation: find \(w\in V_0\) such that
\begin{equation} \label{eq:semi-weak-lifted}
    \int_{\Omega} \epsilon \nabla w \cdot \nabla v~d\boldsymbol x + 2\delta^2 \int_\Omega \sinh(w + g)v~d \boldsymbol x
    =
    -\int_{\Omega} \epsilon \nabla g \cdot \nabla v~d\boldsymbol x + \int_\Omega cv~d \boldsymbol{x}  \qquad \forall v\in V_0.
\end{equation}
To further discretize this weak form, we consider a triangulation of $\Omega$ and let 
$\{\phi_i\}_{i=1}^{N_{\mathrm{FEM}}}$ denote the first-order Lagrange finite element basis functions,  where $N_{\mathrm{FEM}}$ is the number of finite element degrees of freedom. Also define $\boldsymbol w$, $\boldsymbol{g}$, $\boldsymbol{c}$ to be FEM expansion coefficient vectors for $w$, $g$ and $c$ respectively. Now, a discrete version of \eqref{eq:semi-weak-lifted} takes the form
\begin{equation}
    \boldsymbol{K}\boldsymbol w + F(\boldsymbol{w}) = \boldsymbol{b},
\end{equation}
where $\boldsymbol{K}\in \mathbb R^{N_{\mathrm{FEM}} \times N_{\mathrm{FEM}}}$, $\boldsymbol{b} \in \mathbb R^{N_{\mathrm{FEM}}}$, and nonlinear function $F:\mathbb R^{N_{\mathrm{FEM}}} \to \mathbb R^{N_{\mathrm{FEM}}}$ as
\begin{align}\label{eq:FEM-mat-vec}
    [\boldsymbol{K}]_{m,n} &= \int_\Omega \epsilon \nabla \phi_m \cdot \nabla \phi_n ~d\boldsymbol{x}, \qquad m,n=1,\dots,N_{\mathrm{FEM}},\\
    [\boldsymbol b]_n &= \int_\Omega c~ \phi_n ~d\boldsymbol{x}- \int_\Omega \epsilon \nabla g \cdot \nabla \phi_n ~d\boldsymbol{x} ,  \qquad n=1,\dots,N_{\mathrm{FEM}},
\end{align}
and
\begin{equation}
    [F(\boldsymbol{w})]_i = 2\delta^2 \int_{\Omega} \sinh\left( \sum_{j=1}^{N_{\mathrm{FEM}}} [\boldsymbol{w}]_j \phi_j + g \right) \phi_i~d\boldsymbol{x}.
\end{equation}
Here, with a slight abuse of notation, we are referring to finite element approximation $g^h$ and $c^h$ of $g$ and $c$, respectively, with the same symbols.

To solve the nonlinear discrete problem, we introduce the residual $R(\boldsymbol{w}) := \boldsymbol{K}\boldsymbol{w} + F(\boldsymbol{w}) - \boldsymbol{b}$ and apply Newton's method \cite{quarteroni2006numerical} to approximate a root $\boldsymbol{w}^\star$ satisfying $R(\boldsymbol{w}^\star) \approx 0$. The Newton iterations take the form
\begin{equation} \label{eq:newton}
\begin{aligned}
\boldsymbol{w}^{(i+1)} &= \boldsymbol{w}^{(i)} + \Delta \boldsymbol{w}^{(i)}, \\
J_R(\boldsymbol{w}^{(i)}) \Delta \boldsymbol{w}^{(i)} &= -R(\boldsymbol{w}^{(i)}),
\end{aligned}
\end{equation}
where $\boldsymbol{w}^{(i)} \in \mathbb R^{N_{\text{FEM}}}$ denotes the current Newton iterate and $\Delta \boldsymbol{w}^{(i)}$ is the Newton increment. Here $J_R(\boldsymbol{w})$ denotes the Jacobian of the residual evaluated at $\boldsymbol{w}$, given by
\begin{equation}
\begin{aligned}
J_R(\boldsymbol{w}) &= \boldsymbol{K} + J_F(\boldsymbol{w}), \\
[J_F(\boldsymbol{w})]_{i,j} &= 2\delta^2 \int_{\Omega} \cosh(w + g)\phi_i \phi_j ~ d\boldsymbol{x},
\end{aligned}
\end{equation}
where $w = \sum_j [\boldsymbol{w}]_j \phi_j$ is the finite element function corresponding to the coefficient vector $\boldsymbol{w}$. The Newton iterations are terminated once
% the $L^2$-norm of the residual satisfies 
$\|R(\boldsymbol{w}^{(i)})\|_2 < \delta_{\mathrm{Newton}}$.

We denote by $\mathcal W^h : \boldsymbol{c} \mapsto \boldsymbol{u}$ the discrete solution operator that computes $\boldsymbol{w}$ from the discrete nonlinear system and then recovers $\boldsymbol{u}$ using the lifting relation $\boldsymbol{u} = \boldsymbol{w} + \boldsymbol{g}$.

Given the finite element approximation $u^h$, the boundary measurements are approximated by evaluating the flux of $u^h$ across each electrode. For $I_\ell \subset \partial \Omega_3$, we define
\begin{equation} \label{eq:discrete-current}
    [\boldsymbol{y}^h]_\ell 
    := \int_{I_\ell} \nabla u^h \cdot \boldsymbol{n} \, ds, \qquad
    \ell=1,\dots,N_{\mathrm{meas}}.
\end{equation}
That is, the discrete measurements are obtained by integrating the normal component of the numerical flux over each electrode segment. In practice, the integrals in \eqref{eq:discrete-current} are evaluated using boundary quadrature on the mesh. This defines the discrete observation operator $\mathcal{O}^h$. Finally, we define the discrete forward operator by $\mathcal{G}^h := \mathcal{O}^h \circ \mathcal{W}^h$ and write the discrete inverse problem as
\begin{equation} \label{eq:forward-model-discrete}
    \boldsymbol{y} = \mathcal{G}^h(\boldsymbol{c}) + \boldsymbol{\varepsilon},
\end{equation}

\section{Bayesian Formulation of Semiconductor Imaging}\label{sec:Bayesian_formulation}
In this section, we recast the inverse problem in semiconductor imaging \eqref{eq:forward-model-discrete} from a Bayesian perspective. Within this framework, the unknown doping profile parameters, measurement noise, and observed data are treated as random variables. The solution is characterized by the conditional distribution of the unknown parameters given the observed data, commonly known as the \emph{posterior} distribution \cite{kaipio2005statistical}. By Bayes’ theorem, this posterior is proportional to the product of the likelihood, which describes the data distribution for a fixed set of parameters, and the prior, which represents prior knowledge about the parameters before observing the data.

In the following subsections, we introduce the components of this Bayesian formulation, namely the prior, likelihood, and posterior distributions for the semiconductor imaging problem.

\subsection{Gaussian Priors and Pushforward Measures} \label{sec:gaussian-kl}

In this section, we first recall the notion of Gaussian random fields and then review the push-forward construction used to define probability measures on smooth and nearly piecewise-constant fields.

Let $(\mathbb X, \langle \cdot, \cdot \rangle, \| \cdot \|)$ be a Hilbert function space and $(\mathbb X,\mathcal B(\mathbb X), \mathbb P)$, where $\mathcal B(\mathbb X)$ is the Borel $\sigma$-algebra, be a probability space defined on $\mathbb X$. We say that $X$ is an $\mathbb X$-valued Gaussian random function if for any $\mu \in \mathbb X$, the real-valued random variable $\langle X, \mu \rangle$ is a Gaussian, i.e.,  $\langle X, \mu \rangle \sim \mathcal N(m,\tau^2)$, for some $m\in \mathbb R$ and $\tau\in \mathbb R^+$.

% \hsnalt{Let \((\mathbb X, \langle \cdot, \cdot \rangle, \| \cdot \|)\) be a Hilbert function space and let \((\Omega_{\mathrm{prob}},\mathcal F,\mathbb P)\) be a probability space. We say that \(X:\Omega_{\mathrm{prob}}\to \mathbb X\) is an \(\mathbb X\)-valued Gaussian random function if, for any \(\mu \in \mathbb X\), the real-valued random variable \(\langle X, \mu \rangle\) is Gaussian.}

The following lemma characterizes $X$ in terms of a mean function $m\in \mathbb X$ and a symmetric, trace-class, and non-negative linear operator $\mathcal C:\mathbb X\to \mathbb X$. 
\begin{lemma} \cite{ibragimov2012gaussian}
Let $X$ be an $\mathbb X$-valued Gaussian random function.  Then we can find $m \in \mathbb X$ and a trace-class, symmetric, and non-negative linear operator $\mathcal C:\mathbb X \to \mathbb X$, referred to as the \emph{covariance operator}, such that
    \begin{equation}
        \begin{aligned}
            \langle m , \mu \rangle &= \mathbb E \langle  X , \mu \rangle, \qquad &\forall \mu\in \mathbb X, \\
            \langle \mathcal C \mu , \eta \rangle &= \mathbb E \left[\langle X - m , \mu \rangle \langle X - m , \eta \rangle \right], \qquad &\forall \mu,\eta \in \mathbb X,
        \end{aligned}
    \end{equation}
%where $\mathbb E$ denotes expectation,
%\[
%    \mathbb E\big[f(X)\big] := \int_{\mathbb X} f(\mu)  \ d\mathbb P(\mu).
%\]
%We then write $\mathcal N(m, \mathcal C) := \mathbb P \circ X^{-1}$ and $X \sim \mathcal N(m, \mathcal C)$, when referring to a Gaussian random function $X$ defined on the probability space $(\mathbb X, \mathcal B (\mathbb X), N(m, \mathcal C))$.
where \(\mathbb E\) denotes expectation with respect to the probability measure \(\mathbb P\), i.e.,
\[
    \mathbb E[f(X)] := \int_{\mathbb X} f(x)\,d\mathbb P(x).
\]
We write \(X\sim \mathcal N(m,\mathcal C)\) to denote a Gaussian random function with mean \(m\) and covariance operator \(\mathcal C\).
\end{lemma}
The following lemma recalls the \emph{Karhunen--Loève} (KL) expansion, which enables us to express $X$ in terms of the spectral decomposition of the covariance operator $\mathcal C$.

\begin{lemma} \cite{ibragimov2012gaussian} \label{thm:kl}
    Let $m$ and $\mathcal C$ be the mean and covariance operator defined above. Furthermore, let $\{e_j\}_{j=1}^{\infty}$ be the eigenfunctions of \(\mathcal C\), with corresponding eigenvalues \(\{\lambda_j\}_{j=1}^{\infty}\), sorted in decreasing order. Then, $X\sim \mathcal N(m,\mathcal C)$ if and only if $X$ has the infinite expansion
    \begin{equation} \label{eq:kl}
        X= m + \sum_{j=1}^{\infty} \sqrt{\lambda_j} Z_j e_j,
    \end{equation}
    where $Z_j\sim \mathcal N(0,1)$, $j \geq 1$, are independent standard normal real-valued random variables. We interpret the infinite summation as $\mathbb E \| X - m \|^2 < \infty$.% \hsnalt{The infinite series is understood to converge in mean square.}
\end{lemma}
This lemma provides a practical recipe for constructing 
% \textcolor{brown}{(Do you mean "sampling from"?)} 
Gaussian random functions, or fields. First, one chooses a covariance operator $\mathcal{C}$ and computes its eigendecomposition. The expansion in \eqref{eq:kl} is then truncated after \(N_{\mathrm{KL}}\) terms to obtain an approximate random function \(X_N\). The truncation level is chosen so that a desired proportion of the total variance \(\mathbb E\|X-m\|^2\) is retained. For a realization \(\boldsymbol z=(z_1,\ldots,z_{N_{\mathrm{KL}}})\in\mathbb R^{N_{\mathrm{KL}}}\) of the truncated KL coefficient vector, we define the KL reconstruction map
\begin{equation}\label{eq:KL-reconstruction-map}
\mathcal R_N(\boldsymbol z)
:=
m+
\sum_{j=1}^{N_{\mathrm{KL}}}
\sqrt{\lambda_j}\, z_j e_j .
\end{equation}
% Then assemble the expansion in \eqref{eq:kl} and truncate the series after \(N_{\mathrm{KL}}\) terms to obtain the approximate random function $X_N$. The truncation level is chosen so that a desired proportion of the total variance \(\mathbb E\|X-m\|^2\) is retained. For a realization
% \(\boldsymbol z=(z_1,\ldots,z_{N_{\mathrm{KL}}})\in\mathbb R^{N_{\mathrm{KL}}}\)
% of the truncated KL coefficient vector, we define the KL reconstruction map
% \[
% \mathcal R_N(\boldsymbol z)
% :=
% m+
% \sum_{j=1}^{N_{\mathrm{KL}}}
% \sqrt{\lambda_j}\, z_j e_j .
% \]
% We write this truncated field as
% \begin{equation}\label{eq:truncated_KL}
%    X_N(\boldsymbol z)
%    =
%    m+
%    \sum_{j=1}^{N_{\mathrm{KL}}}
%    \sqrt{\lambda_j}\, Z_j e_j,
% \end{equation}
% where \(\boldsymbol z=(z_1,\ldots,z_{N_{\mathrm{KL}}})\) denotes a realization of the KL coefficient vector \(\boldsymbol Z=(Z_1,\ldots,Z_{N_{\mathrm{KL}}})\).
A useful class of two-dimensional covariance operators satisfying the conditions of Lemma~\ref{thm:kl} is given by \cite{dunlop2017hierarchical}
\begin{equation} \label{eq:matern-cov}
\mathcal{C} = (\tau I - \Delta)^{-\alpha},
\end{equation}
for $\tau > 0$ and $\alpha > 1.5$. These operators are motivated by the Whittle--Mat\'ern covariance kernels \cite{whittle1954}, which allow control over both the correlation length through $\tau$ and the regularity of sample functions through $\alpha$. It is known that if $X \sim \mathcal{N}(0,\mathcal{C})$, with $\mathcal{C}$ as in \eqref{eq:matern-cov}, then $X \in W^{s,2}$ for all $s < \alpha - 1$, almost surely \cite{dunlop2017hierarchical}.

To obtain an eigendecomposition of \(\mathcal C\), we first compute the eigenpairs \(\{e_i,\gamma_i\}_{i=1}^{N_{\mathrm{KL}}}\) of \((\tau I - \Delta)\) using FEM. Then we define
\begin{equation} \label{eq:eig_inverse_relation}
    \lambda_i := \gamma_i^{-\alpha}, \qquad i=1,\dots,N_{\mathrm{KL}},
\end{equation}
to assemble the truncated KL expansion in \eqref{eq:kl}.

Informally, one may writes the Gaussian measure in a density-like form as
\[
    d\mathcal N(0,\mathcal C)
    \propto
    \exp\left(
        -\frac{1}{2}
        \langle \boldsymbol{x}, \mathcal C^{-1}\boldsymbol{x} \rangle
    \right)d\boldsymbol{x}.
\]
% \textcolor{brown}{What does $d$ denote? It looks weird to me. Can we include a reference for this?}\bbkc{It doesn't mean anything. The correct way is to use Radon-Nikodym derivative. }
This expression is only symbolic in infinite-dimensional settings, since there is no infinite-dimensional Lebesgue measure with respect to which $\mathcal N(0,\mathcal C)$ admits a density. Nevertheless, this notation is convenient and correctly reflects the quadratic structure of the Gaussian measure.  In a finite-dimensional truncated KL representation, the corresponding negative log-prior for the KL coefficient vector \(\boldsymbol z\) is $
    \Psi_Z(\boldsymbol z)
    :=
    \frac12
    \|\boldsymbol z\|_2^2 $. 
% as $\Psi_X(\boldsymbol{x}) := 1/2 \langle \boldsymbol{x}, \mathcal C^{-1} \boldsymbol{x} \rangle $.

In the next subsection, we combine Gaussian random fields with nonlinear mappings to define probability measures of nearly piecewise constant functions. These constructions are inspired by \cite{dunlop2017hierarchical}.

% \textcolor{brown}{I would suggest to add a pseudocode/algorithm for the prior construction as well as for the gradient-based-inference (section 5). I would even add a comprehensive algorithm for Adjoint-Based Bayesian for source inversion (for semiconductors). This is of course just a suggestion and you can leave it like as it is.}

\subsection{Piecewise-Constant Random Fields and Their Smooth Approximation} \label{sec:prior-piecewise-constant}
In many semiconductor applications, the doping profile $c$ exhibits a nearly piecewise constant structure, consisting of regions with approximately uniform dopant concentrations separated by narrow transition layers \cite{burger2001identification,markowich1985stationary,taghizadeh2025bayesian}. To accommodate such models, we have previously defined two plateau values $c_{\mathrm{p}},~c_{\mathrm{n}}\in \mathbb R$. In this section, we use the Heaviside function, defined by $H(t)=\boldsymbol{1}_{[0,\infty)}(t)$ 
% $H(\boldsymbol x) = \boldsymbol{1}_{[0,\infty)}(\boldsymbol x)$
, where $\boldsymbol{1}$ denotes the characteristic function, to construct a bilevel piecewise constant field:
\begin{equation} \label{eq:heaviside}
    C = F_1(X;c_{\mathrm p},c_{\mathrm n}) := c_{\text{p}} + (c_{\text{n}} - c_{\text{p}}) H(X), \qquad X \sim \mathcal{N}(0,\mathcal{C}),
\end{equation}
where $C$ denotes the random field corresponding to the doping profile $c$, with an associated probability space specified below, and $X$ is a Gaussian random field as defined in Section~\ref{sec:gaussian-kl}.

Although this construction is a useful modeling tool, the Heaviside map is non-differentiable. To address this, we approximate $F_1$ by a function with continuous spatial transitions that remains close to a piecewise constant field. This can be achieved using a sigmoid transformation:
\begin{equation} \label{eq:piecewise-constant}
    C = F_2(X;c_{\mathrm p},c_{\mathrm n}) := c_{\text{p}} + (c_{\text{n}} - c_{\text{p}}) \frac{1}{1 + \exp(-k X)}, \qquad X \sim \mathcal N(0,\mathcal C),
\end{equation}
where $k > 0$ controls the sharpness of the transition. As $k \to \infty$, the transition becomes increasingly sharp and $F_2$ approaches $F_1$.

For fixed plateau values \(c_{\mathrm p}\) and \(c_{\mathrm n}\), we define a probability measure on the target doping concentration space $\mathbb S$, consisting of piecewise-constant or nearly piecewise-constant fields, via the push-forward of $\mathbb{P}$ under $F_{\ell}$, that is, $\mathcal P_\ell :=\mathbb{P} \circ F_{\ell}^{-1}$ for $\ell = 1,2$. This construction defines the probability space $(\mathbb S, \mathcal B(\mathbb S), \mathcal P_\ell)$ for the random doping field, and we write $C \sim \mathcal P_\ell$. Note that the inverse map appears only as a measure-theoretic device for defining probability measures on the desired function spaces; it is not evaluated explicitly.

\subsection{Priors for Bilevel Parameters and the Joint Measure}
\label{sec:joint_prior}
While the prior measure introduced in the previous subsection provides a flexible framework for modeling the geometry of the interface in the doping profile, it is also of interest to infer the doping levels themselves \cite{cheng2011recovering,Ali2006Inverse,burger2004inverse}. To this end, we assign a probabilistic structure to the constants $c_{\text{p}}$ and $c_{\text{n}}$ by introducing real-valued random variables $C_{\text{p}}$ and $C_{\text{n}}$, defined on the probability spaces $(\mathbb R, \mathcal B(\mathbb R), \mu_{\text{p}})$ and $(\mathbb R, \mathcal B(\mathbb R), \mu_{\text{n}})$, respectively. The measures $\mu_{\text{p}}$ and $\mu_{\text{n}}$ encode prior distributions over the admissible values of $c_{\text{p}}$ and $c_{\text{n}}$, for instance, the Lebesgue measure \cite{Rudin1987} on $\mathbb R$ or an interval.% \hsnalt{for instance uniform distributions on bounded intervals or Gaussian distributions on $\mathbb R$.}

We assume that $X$, $C_{\text{p}}$, and $C_{\text{n}}$ are mutually independent. 
% Let us define the joint random variable $\Theta := (X, C_{\text{p}}, C_{\text{n}})$ which allows us to define a joint prior measure on the product space $\mathbb X \times \mathbb R \times \mathbb R$ as the product measure
Let \(\mu_X:=\mathcal N(0,\mathcal C)\) denote the law of the Gaussian random field \(X\). We define the joint random variable $\Theta := (X, C_{\text{p}}, C_{\text{n}})$  which takes values in the product space \(\mathbb X \times \mathbb R \times \mathbb R\). Under the independence assumption, the joint prior measure is the product measure
\begin{equation}
\Pi := \mu_X \otimes \mu_{\text{p}} \otimes \mu_{\text{n}}.
\end{equation}
% The corresponding random element $(X, C_{\text{p}}, C_{\text{n}})$ then takes values in $\mathbb X \times \mathbb R \times \mathbb R$ with law $\Pi$. 

Using this joint prior, the mappings $F_1$ and $F_2$ can be extended to depend on $(X, C_{\text{p}}, C_{\text{n}})$, yielding random fields of the form
\begin{equation} \label{eq:joint-dope}
C = F_{\ell}(X, C_{\text{p}}, C_{\text{n}}), \qquad \ell = 1,2.
\end{equation}
Consequently, the prior measure on the doping space $\mathbb S$ is given by the pushforward
\begin{equation} \label{eq:final-measure}
\mathcal P_\ell := \Pi \circ F_{\ell}^{-1}, \qquad \ell = 1,2,
\end{equation}
which defines a probability measure on $(\mathbb S, \mathcal B(\mathbb S))$ incorporating uncertainty in both the geometry encoded by \(X\) and the plateau values \(C_{\mathrm p}\) and \(C_{\mathrm n}\). In the following lemma, we show exponential integrability of the proposed joint measure.

\begin{lemma} \label{thm:exp-int}
Suppose that $\Omega\subset\mathbb R^2$ is bounded and that the prior measures
$\mu_{\mathrm p}$ and $\mu_{\mathrm n}$ are supported on a bounded interval
$[-M,M]$ for some $M>0$. Let $C$ be the doping profile defined in \eqref{eq:joint-dope} where $(X,C_{\mathrm p},C_{\mathrm n})\sim\Pi$. Then, for every $\alpha>0$,
\[
\int_{\mathbb X\times\mathbb R\times\mathbb R}\exp\!\left(\alpha\|C\|_{L^2(\Omega)}\right)\,d\Pi<\infty.
\]
\end{lemma}

\begin{proof}
Since the logistic function takes values in \((0,1)\), we have, for almost every \(\boldsymbol x\in\Omega\),
\[
C(\boldsymbol x)
\in
\big[
\min\{C_{\mathrm p},C_{\mathrm n}\},
\max\{C_{\mathrm p},C_{\mathrm n}\}
\big].
\]. Hence, $|C(\boldsymbol x)|\le M,$
which implies $\|C\|_{L^2(\Omega)}\le M|\Omega|^{1/2}$.
Therefore,
\[
\int
\exp\!\left(\alpha\|C\|_{L^2(\Omega)}\right)
\,d\Pi
\le
\exp\!\left(\alpha M|\Omega|^{1/2}\right)
<\infty.
\]
\end{proof}

After the finite-dimensional KL truncation, we denote by
\(\boldsymbol\theta=(\boldsymbol z,c_{\mathrm p},c_{\mathrm n})\)
the corresponding parameter vector. Here,
\(\boldsymbol z\in\mathbb R^{N_{\mathrm{KL}}}\) is a realization of the truncated KL coefficient vector, and the field associated with this realization is obtained through the KL reconstruction map \(\mathcal R_N\) defined in \eqref{eq:KL-reconstruction-map}. The parameters \(c_{\mathrm p}\) and \(c_{\mathrm n}\) denote the plateau doping levels. 
% The associated negative log-prior, up to an additive constant, is written as
% After the finite-dimensional KL truncation, we denote by \(\boldsymbol\theta=(\boldsymbol z,c_{\mathrm p},c_{\mathrm n})\) the corresponding parameter vector, where \(\boldsymbol z\) contains the KL coefficient vector used to construct the truncated Gaussian field \(X_N(\boldsymbol z)\) \bbkc{This is wrong. I commented the formula you had above but now I see you used it here. You can write $X(Z)$ or $\boldsymbol{x(\boldsymbol{z})}$ but you can't write $X(z)$. In General it is not a good idea to write X as a function of Z in infinite dimensions, because the infinite vector $Z$ is not well-defined. That is why even when we truncated the expansion we don't usually write it as a dependency. It makes statisticians uncomfortable :-)} in \eqref{eq:truncated_KL}, and \(c_{\mathrm p}\) and \(c_{\mathrm n}\) denote the plateau doping levels. 
The associated negative log-prior, up to an additive constant, is written as
\begin{equation}\label{eq:NLprior}
    \Psi_\Theta(\boldsymbol\theta)
    =
    \Psi_Z(\boldsymbol z)
    +
    \Psi_{C_{\mathrm p}}(c_{\mathrm p})
    +
    \Psi_{C_{\mathrm n}}(c_{\mathrm n}),
\end{equation}
where \(\Psi_Z\) denotes the negative log-density of the Gaussian prior on the truncated KL coefficients.

For notational simplicity, we also write the finite-dimensional representation of the transformation as
\begin{equation}\label{eq:pf_map_finite}
     F_\ell(\boldsymbol\theta) := F_\ell\big(\mathcal R_N(\boldsymbol z),c_{\mathrm p},c_{\mathrm n}\big), \qquad \ell=1,2.
\end{equation}
% \bbkc{The same comment here: you cannot write $X_N(\boldsymbol z)$. One solution is to write a mappting $P:Z\mapsto X$ and then writing $P(\boldsymbol{z})$.
This finite-dimensional parameterization also clarifies how the plateau values enter the forward problem. Since the Dirichlet boundary values in \eqref{eq:semi-PDE-strong} depend on \(c_{\mathrm p}\) and \(c_{\mathrm n}\), the joint prior on \((X,C_{\mathrm p},C_{\mathrm n})\) induces variability not only in the interior doping field but also in the boundary data. Consequently, the parameter vector \(\boldsymbol\theta\) enters the forward map through two distinct channels: the source term \(c=F_\ell(\boldsymbol\theta)\) and the lifting function associated with the parameter-dependent Dirichlet boundary conditions.

\subsection{Likelihood} \label{sec:likelihood-posterior}

In this subsection, we formulate the likelihood function for the semiconductor inverse problem. This likelihood will later be combined with the prior distribution via Bayes' theorem to obtain the posterior distribution.

Recall that the random doping field is represented by $F_\ell(\Theta)$ , $\ell = 1,2$, with prior measure $\mathcal P_\ell= \Pi \circ F^{-1}_\ell$(cf. Section \ref{sec:joint_prior}). We now reformulate the deterministic inverse problem \eqref{eq:forward-model-discrete} in a statistical setting. Let $Y$ denote the random variable corresponding to the observed data $\boldsymbol y$, and let $E$ denote the random variable describing the observational noise $\boldsymbol \varepsilon$. The probabilistic version of the discrete semiconductor inverse problem \eqref{eq:forward-model-discrete} is then
\begin{equation} \label{eq:forward-probabilistic}
    Y = \mathcal G^h\big(F_{\ell}(\Theta) \big) + E, \qquad \ell = 1,2.
\end{equation}
A standard approach to defining the likelihood distribution, i.e., the distribution of the conditional random variable $Y|(F_\ell(\Theta) = c)$, is to observe that $Y - \mathcal G^h(c) \sim E$. Therefore, the likelihood distribution is obtained by shifting the distribution of $E$ by $\mathcal G^h(c)$. 

% After the finite truncation of the KL expansion, we use the parameter vector \(\boldsymbol{\theta}\) introduced in Section~\ref{sec:joint_prior} and the notation \(F_\ell(\boldsymbol\theta)\) introduced in \eqref{eq:pf_map_finite}. \bbkc{Why do we need the truncated system? Everything is well-defined in infinite dimensions. My suggestion is to remove the previous sentence.} 
Assuming the noise is a multivariate Gaussian random variable with a covariance $\sigma^2_{\text{noise}}I_{N_{\text{meas}}}$, the likelihood function considered here, up to a constant of proportionality, is
\begin{equation} \label{eq:likelihood}
    L(\boldsymbol{y};\boldsymbol{\theta} ) \propto \exp\left( -  \frac{1}{2\sigma^2_{\text{noise}}}\| \mathcal G^h\big(F_\ell(\boldsymbol{\theta})\big) - \boldsymbol{y} \|_{2}^2 \right).
\end{equation}
The parameter \(\sigma_{\mathrm{noise}}\) denotes the noise standard deviation.

The corresponding negative log-likelihood, up to an additive constant, is defined as
\begin{equation} \label{eq:negative-log-likelihood}
    \Phi(\boldsymbol y;\boldsymbol\theta)
    :=
    \frac{1}{2\sigma^2_{\mathrm{noise}}}
    \left\|
    \mathcal G^h\big(F_\ell(\boldsymbol\theta)\big)
    -
    \boldsymbol y
    \right\|_2^2 .
\end{equation}

\subsection{Posterior Distribution}
We now apply Bayes' theorem to define the posterior distribution of the doping profile, namely the conditional distribution of \(C\) given the measurements \(\boldsymbol y\). At the measure level, for each \(\ell=1,2\), the posterior is defined on the space \((\mathbb S,\mathcal B(\mathbb S))\) through the prior measure \(\mathcal P_\ell\) introduced in \eqref{eq:final-measure}. The next result gives the corresponding posterior measure formulation for the doping field and states its stability with respect to perturbations in the observed data.

\begin{theorem}
Let $(\mathbb S, \mathcal B(\mathbb S), \mathcal P_\ell )$, with $\ell=1,2$, be the probability space defined above associated with the prior distribution $\mathcal P_\ell$ for random doping profile $C$, introduced in \eqref{eq:final-measure}. Suppose that $\Phi$ is the negative log-likelihood defined in \eqref{eq:negative-log-likelihood}. Then the posterior distribution $\mathcal P_\ell^{\text{post}}$, defined as the conditional probability measure of the doping field given the measurements, is absolutely continuous with respect to the prior measure $\mathcal P_\ell$, i.e., $\mathcal P_\ell^{\text{post}} \ll \mathcal P_\ell$, and is expressed as the Radon-Nikodym derivative
\begin{equation} \label{eq:Bayes}
    \frac{d \mathcal P_\ell^{\text{post}} }{d \mathcal P_\ell} (s) = \frac{1}{Z(\boldsymbol y)} \exp\big(-\Phi(\boldsymbol{y}; s )\big),
\end{equation}
with normalization constant
\begin{equation}
    Z(\boldsymbol y) = \int_{\mathbb S}  \exp(-\Phi(\boldsymbol{y}; s ) )  \mathcal P_\ell(ds ).
\end{equation}
Furthermore, for any two data vectors $\boldsymbol y_1 $ and $\boldsymbol y_2$ satisfying $\max \{ \| \boldsymbol y_1 \|_2 , \| \boldsymbol y_2 \|_{2} \}<r$ for some \(r>0\), there is $\kappa>0$ independent of \(\boldsymbol y_1\) and \(\boldsymbol y_2\), such that
% \begin{equation}
%     d_{\text{Hell}}( \mathcal P^{\text{post}}_{C|(Y=\boldsymbol{y}_1)},\mathcal P^{\text{post}}_{C|(Y=\boldsymbol{y}_2 )} ) \leq \kappa \| \boldsymbol{y}_1 -\boldsymbol{y}_2 \|_2.
% \end{equation}
\begin{equation}
    d_{\mathrm{Hell}}\left(
    \mathcal P_{\ell,\boldsymbol y_1}^{\mathrm{post}},
    \mathcal P_{\ell,\boldsymbol y_2}^{\mathrm{post}}
    \right)
    \leq
    \kappa
    \| \boldsymbol y_1-\boldsymbol y_2 \|_2 .
\end{equation}
where $d_{\text{Hell}}(\cdot,\cdot)$ denotes the Hellinger distance between probability measures \cite{le2000asymptotics}.
\end{theorem}
\begin{proof}
We refer the reader to Appendix~\ref{s:inverse problem} for the proof.
\end{proof}

After a truncation of the KL expansion, we use the parameter vector \(\boldsymbol{\theta}\) introduced in Section~\ref{sec:joint_prior} and the notation \(F_\ell(\boldsymbol\theta)\) introduced in \eqref{eq:pf_map_finite}. The negative log-posterior is then obtained by combining the negative log-likelihood \(\Phi(\boldsymbol y;\boldsymbol\theta)\) with the negative log-prior \(\Psi_\Theta(\boldsymbol\theta)\):
\begin{equation}
\begin{aligned}
    \mathcal J(\boldsymbol{\theta})
    &:=
    \Phi(\boldsymbol y;\boldsymbol{\theta})
    +
    \Psi_\Theta(\boldsymbol{\theta}) \\
    &=
    \frac{1}{2\sigma^2_{\mathrm{noise}}}
    \left\|
    \mathcal G^h \big(F_\ell(\boldsymbol{\theta})\big)
    -
    \boldsymbol y
    \right\|_2^2
    +
    \Psi_\Theta(\boldsymbol{\theta}).
\end{aligned}
\label{eq:negative-log-posterior}
\end{equation}
Equivalently, the finite-dimensional posterior density is characterized, up to a normalizing constant, by
\begin{equation}
    \pi(\boldsymbol{\theta}\mid \boldsymbol y)
    \propto
    \exp\left(
        -
        \mathcal J(\boldsymbol{\theta})
    \right).
    \label{eq:finite-dimensional-posterior}
\end{equation}
The functional \(\mathcal J\) is used in Section~\ref{subsec:adjoint-gradient} for adjoint-based gradient computation, MAP estimation, and gradient-based posterior sampling.

\section{Gradient-Based Inference}\label{G-inference}

In this section, we describe how to compute gradient information for the semiconductor inverse problem in order to enable efficient gradient-based optimization and sampling. We adopt the adjoint method to compute derivatives of a scalar objective functional with respect to the inferred parameters. Rather than explicitly forming the full Jacobian of the PDE solution with respect to all parameter directions, which would require one sensitivity solve per parameter dimension, the adjoint approach introduces an auxiliary (adjoint) variable and evaluates the gradient through a single adjoint solve combined with sensitivity terms from the forward solution. This yields the gradient with respect to all parameters at a computational cost that is essentially independent of the parameter dimension, making it significantly more efficient for high-dimensional inference problems.

\subsection{Adjoint-Based Gradient Computation}
\label{subsec:adjoint-gradient}

In this subsection, we derive the gradient of the finite-dimensional negative log-posterior \(\mathcal J\) defined in \eqref{eq:negative-log-posterior}. This gradient is used both in MAP estimation and in the gradient-based posterior sampling algorithm. Throughout this subsection, we assume that the doping profile is constructed using the differentiable sigmoid map $c = F_2(\boldsymbol\theta):= F_2\big(\mathcal R_N(\boldsymbol z),c_{\mathrm p},c_{\mathrm n}\big),$ as defined in \eqref{eq:pf_map_finite}.  Recall that the finite-dimensional parameter vector is
\(\boldsymbol\theta=(\boldsymbol z,c_{\mathrm p},c_{\mathrm n})\), where
\(\boldsymbol z\in\mathbb R^{N_{\mathrm{KL}}}\) is the truncated KL coefficient vector, \(\mathcal R_N(\boldsymbol z)\) is the reconstructed field defined in \eqref{eq:KL-reconstruction-map}, and \(c_{\mathrm p}\) and \(c_{\mathrm n}\) denote the plateau doping levels.
% \bbkc{here again we have the bad notation!}  defined in \eqref{eq:pf_map_finite}. Recall the fininte dimention vectopr \(\boldsymbol\theta=(\boldsymbol z,c_{\mathrm p},c_{\mathrm n})\), where \(\boldsymbol z\) contains the KL coefficient vector used to construct the truncated Gaussian field \(X_N(\boldsymbol z)\) in \eqref{eq:truncated_KL} \bbkc{same notation issue.}, and \(c_{\mathrm p}\) and \(c_{\mathrm n}\) denote the plateau doping levels. 

From \eqref{eq:negative-log-posterior}, the gradient decomposes into a likelihood contribution and a prior contribution:
\begin{equation}
    \nabla_{\boldsymbol\theta}\mathcal J(\boldsymbol\theta)
    =
    \nabla_{\boldsymbol\theta}
    \Phi(\boldsymbol y;\boldsymbol\theta)
    +
    \nabla_{\boldsymbol\theta}
    \Psi_\Theta(\boldsymbol\theta).
    \label{eq:NLP-gradient-decomposition}
\end{equation}
% We first summarize the parameter dependencies in the forward model. The likelihood depends on the predicted data through the lifted solution \(u=w+g\). The function \(w\) depends on \(\boldsymbol\theta\) implicitly through the weak form, while the lifting \(g\) depends on the plateau values \(c_{\mathrm p}\) and \(c_{\mathrm n}\) through the Dirichlet boundary conditions. The doping field \(c=F_2(\boldsymbol\theta)\) depends on both the KL coefficients \(\boldsymbol z\) and the plateau values. These dependencies are summarized in Fig.~\ref{fig:dependency}.
We first summarize the parameter dependencies in the forward map
\(\mathcal G^h\big(F_2(\boldsymbol{\theta})\big)\). The doping field \(c=F_2(\boldsymbol\theta)\) depends on both the KL coefficients \(\boldsymbol z\) and the plateau values \(c_{\mathrm p}\) and \(c_{\mathrm n}\). The lifting function \(g=g(c_{\mathrm p},c_{\mathrm n})\) depends on the plateau values through the Dirichlet boundary conditions
$ g|_{\partial \Omega_2} = \operatorname{arcsinh}\left(c_{\mathrm p}/({2\delta^2})\right)$ and $g|_{\partial \Omega_3} = \operatorname{arcsinh}\left(c_{\mathrm n}/({2\delta^2})\right).$ The auxiliary solution \(w=w(c,g)\) depends implicitly on both \(c\) and \(g\) through the lifted weak form; see \eqref{eq:semi-weak-lifted}. Consequently, the full solution \(u=w+g\), the predicted data, and hence the likelihood all depend on \(\boldsymbol\theta\).
%This parameter dependence of the lifting is absent in the known-plateau setting, but becomes essential in the joint reconstruction problem.
The adjoint formulation must therefore account for both the perturbation of the interior source term and the perturbation of the Dirichlet lifting. These dependencies are summarized in Fig.~\ref{fig:dependency}.

\begin{figure}[ht]
\centering

\resizebox{\textwidth}{!}{%
\begin{tikzpicture}[
  node distance=1.0cm and 1.2cm,
  box/.style={
    draw,
    rounded corners,
    thick,
    minimum width=1.1cm,
    minimum height=0.6cm,
    align=center,
    font=\normalsize
  },
  param/.style={
    box,
    fill=cyan!15
  },
  jbox/.style={
    box,
    fill=yellow!35
  },
  groupbox/.style={
    draw,
    rounded corners,
    thick,
    inner sep=6pt
  },
  arrow/.style={->, thick}
]

% Parameters grouped as theta
\node[param] (x) {$\boldsymbol{z}$};
\node[param, below=0.35cm of x] (cp) {$c_{\mathrm p}$};
\node[param, below=0.35cm of cp] (cn) {$c_{\mathrm n}$};

\node[
  groupbox,
  fit=(x)(cp)(cn),
  label=left:{$\boldsymbol \theta$}
] (theta) {};

% Intermediate variables
\node[box, right=2.2cm of x] (c) {$c=F_2(\boldsymbol\theta)$};

\node[box, below=0.7cm of c] (g) {$g(c_{\mathrm p},c_{\mathrm n})$};

\node[box, right=2.2cm of c, yshift=0.0cm] (w) {$w(c,g)$};

\node[box, right=2.05cm of g] (u) {$u=w+g$};
\node[box, right=of u] (y) {$\mathcal G^h(c)$};
\node[box, right=of y] (phi) {$\Phi$};

\node[box, above=0.55cm of phi] (phiprior) {$\Psi_\Theta$};

% Objective
% \node[jbox, right=1.5cm of phi] (J) {$\mathcal J$};
\node[jbox, right=1.5cm of $(phi)!0.5!(phiprior)$] (J) {$\mathcal J$};

% Arrows into g
% \draw[arrow] (x) -- (g);
\draw[arrow] (cp) -- (g);
\draw[arrow] (cn) -- (g);

% Arrows into c
% \draw[arrow] (x) -- (c);
% \draw[arrow] (cp) -- (c);
% \draw[arrow] (cn) -- (c);
\draw[arrow] ($(theta.east)+(0,0.96)$) -- (c.west);
% \draw[arrow] (theta.east) -- ++(0,0.95) -- (c.west);
% \draw[arrow] (cp) -- (c);
% \draw[arrow] (cn) -- (c);

% Dependencies into w
% \draw[arrow] (x) -- (w);
% \draw[arrow] (c) -- (w);
\draw[arrow] (c) -- (w);
\draw[arrow] (g) -- (w);

% Forward model dependencies
\draw[arrow] (w) -- (u);
\draw[arrow] (g) -- (u);

\draw[arrow] (u) -- (y);
\draw[arrow] (y) -- (phi);
\draw[arrow] (phi) -- (J);

% Prior term: theta box to Phi_prior from above
\draw[arrow] (theta.north) -- ++(0,0.2) -| (phiprior.north);

% \draw[arrow] (phiprior) -| (J);
\draw[arrow] (phiprior) -- (J);

\end{tikzpicture}
}
\caption{Dependency chart for finite-dimensional parameters in the semiconductor inverse problem.}
\label{fig:dependency}
\end{figure}
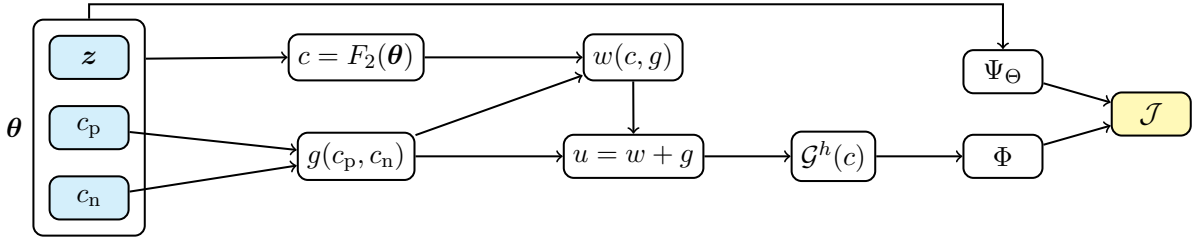
%Let \(u=w+g\) be the lifted state introduced in \eqref{eq:semi-weak-lifted}, where \(w\in V_0\) and \(g\) satisfies the Dirichlet boundary conditions. In the joint reconstruction case, the plateau values \(c_{\mathrm p}\) and \(c_{\mathrm n}\) are components of \(\boldsymbol\theta\). Therefore, the lifting \(g\) depends on \(\boldsymbol\theta\) through the Dirichlet boundary functions \(f_1\) and \(f_2\) in \eqref{eq:semi-PDE-strong}.

We seek the directional derivative of the cost functional $\mathcal J$ with respect to $\boldsymbol{\theta}$ in the direction $\dot{\boldsymbol{\theta}}$, defined by
\begin{equation} \label{eq:directional-derivative}
D_{\boldsymbol{\theta}} \mathcal J [\dot{\boldsymbol{\theta}}]
:=
\left.\frac{d}{d\varepsilon}\mathcal J(\boldsymbol{\theta}+\varepsilon \dot{\boldsymbol{\theta}})\right|_{\varepsilon=0}.
\end{equation}
Once this directional derivative is characterized, the gradient $\nabla_{\boldsymbol{\theta}}\mathcal J$ can be recovered by evaluating the directional derivative along the basis directions of the parameter space. Expanding \eqref{eq:directional-derivative} yields
% \begin{equation}
%     \begin{aligned}
%     D_{\boldsymbol{\theta}} \mathcal J[\dot{\boldsymbol{\theta}}]
%     = &
%     \frac{1}{\sigma_{\mathrm{noise}}^2}
%     \left\langle
%     r,
%     \mathcal O^h(\dot w+\dot g)
%     \right\rangle_{\ell^2}
%     +
%     \left\langle
%     \dot{\boldsymbol{x}},
%     \mathcal C^{-1}\boldsymbol{x}
%     \right\rangle =\\
%     &\frac{1}{\sigma_{\mathrm{noise}}^2}
%     \left\langle
%     r,
%     \mathcal O^h(\dot w)
%     \right\rangle_{\ell^2}
%     +
%     \frac{1}{\sigma_{\mathrm{noise}}^2}
%     \left\langle
%     r,
%     \mathcal O^h(\dot g)
%     \right\rangle_{\ell^2}
%     +
%     \left\langle
%     \dot{\boldsymbol{x}},
%     \mathcal C^{-1}\boldsymbol{x}
%     \right\rangle.
%     \end{aligned}
% \end{equation}
\begin{equation}
    \begin{aligned}
    D_{\boldsymbol{\theta}} \mathcal J[\dot{\boldsymbol{\theta}}]
    =
    &\frac{1}{\sigma_{\mathrm{noise}}^2}
    \left\langle
    r,
    \mathcal O^h(\dot w+\dot g)
    \right\rangle_{\ell^2}
    +
    D_{\boldsymbol{\theta}}\Psi_\Theta(\boldsymbol{\theta})
    [\dot{\boldsymbol{\theta}}] \\
    =
    &\frac{1}{\sigma_{\mathrm{noise}}^2}
    \left\langle
    r,
    \mathcal O^h(\dot w)
    \right\rangle_{\ell^2}
    +
    \frac{1}{\sigma_{\mathrm{noise}}^2}
    \left\langle
    r,
    \mathcal O^h(\dot g)
    \right\rangle_{\ell^2}
    +
    D_{\boldsymbol{\theta}}\Psi_\Theta(\boldsymbol{\theta})
    [\dot{\boldsymbol{\theta}}],
    \end{aligned}
    \label{eq:DJ-expanded}
\end{equation}
Here, we have defined the residual vector $r := \mathcal O^h(u)-\boldsymbol y$ and used the quadratic structure of the $\ell^2$-norm in the likelihood. In the following, we focus on the likelihood contribution. The perturbation \(\dot g\) can be computed directly from the dependence of the lifting on \(c_{\mathrm p}\) and \(c_{\mathrm n}\). The perturbation \(\dot w\in V_0\) is obtained by differentiating the lifted weak form \eqref{eq:semi-weak-lifted}, which gives the sensitivity equation
% Furthermore, $\dot g$ can be directly computed from its immediate relation to $\boldsymbol{\theta}$ and $\dot w$ is obtained by applying \eqref{eq:directional-derivative} to the weak form PDE problem \eqref{eq:semi-weak-lifted} and solving
\begin{equation} \label{eq:sensitivity}
    \begin{aligned}
    \int_{\Omega} \epsilon \nabla \dot w \cdot \nabla v \, d\boldsymbol{x}
    &+
    2\delta^2 \int_{\Omega}
    \cosh(w+g)\dot w v \, d\boldsymbol{x}
    = \int_\Omega \dot c\,v\,d\boldsymbol x
    \\
    &- \int_{\Omega} \epsilon \nabla \dot g \cdot \nabla v \, d\boldsymbol{x}
    -
    2\delta^2 \int_{\Omega}
    \cosh(w+g)\dot g v \, d\boldsymbol{x},
    \qquad \forall v\in V_0.
    \end{aligned}
\end{equation}
% Therefore, we can evaluate the first term of $D_{\boldsymbol{\theta}} \mathcal J[\dot{\boldsymbol{\theta}}]$ by first solving \eqref{eq:semi-weak-lifted} for $(u,w,g)$ and then solving the sensitivity equation \eqref{eq:sensitivity} for $\dot w$. 
Thus, the likelihood contribution involving \(\dot w\) could be evaluated by first solving the nonlinear forward problem \eqref{eq:semi-weak-lifted} for \(w\), recovering \(u=w+g\), and then solving the sensitivity equation \eqref{eq:sensitivity} for \(\dot w\).
% However, this approach is computationally expensive when the dimension of \(\boldsymbol\theta\) is large. Therefore, we next introduce an adjoint equation that eliminates the explicit dependence on \(\dot w\).

% We remark that the immediate relation of $c$ to $\boldsymbol{\theta}$ through $F_2$ (cf. \eqref{eq:piecewise-constant}) yields
% \[
% \begin{aligned}
%     \dot c
%     &=
%     \left(1-\frac{1}{1+\exp(-k\mathcal C^{1/2}\boldsymbol x)}\right)\dot c_{\mathrm p}
%     +
%     \frac{1}{1+\exp(-k\mathcal C^{1/2}x)}\dot c_{\mathrm n} \\
%     &\quad+
%     (c_{\mathrm n}-c_{\mathrm p})
%     k
%     \frac{\exp(-k\mathcal C^{1/2} \boldsymbol x)}
%     {\left(1+\exp(-k\mathcal C^{1/2}\boldsymbol x)\right)^2}
%     \mathcal C^{1/2}\dot {\boldsymbol x}.
% \end{aligned}
% \]

The perturbation $\dot c$ is obtained by differentiating the finite-dimensional push-forward map $F_2(\boldsymbol\theta)=F_2(\mathcal R_N(\boldsymbol z),c_{\mathrm p},c_{\mathrm n})$ introduced in \eqref{eq:pf_map_finite}, i.e.,
$\dot c =D_{\boldsymbol\theta}F_2(\boldsymbol\theta)[\dot{\boldsymbol\theta}]$. We define the corresponding finite-dimensional sigmoid factor by
\[ s_N := \frac{1}{1+\exp\!\left(-k\mathcal R_N(\boldsymbol z)\right)}, \]
where \(\mathcal R_N\) is the KL reconstruction map defined in \eqref{eq:KL-reconstruction-map}.
% where $X_N(\boldsymbol z)$ \bbkc{same notation issue} is defined in \eqref{eq:truncated_KL}. 
Since
$ c = F_2(\boldsymbol\theta) = c_{\mathrm p} + (c_{\mathrm n}-c_{\mathrm p})s_N,$
we obtain
\begin{equation}\label{eq:dc-pushforward}
\dot c = (1-s_N)\dot c_{\mathrm p}
+ s_N\dot c_{\mathrm n} + \left[(c_{\mathrm n}-c_{\mathrm p})
k~ s_N(1-s_N) \right] \sum_{j=1}^{N_{\mathrm{KL}}}
\sqrt{\lambda_j}\dot z_j e_j .
\end{equation}
Here, $\dot c_{\mathrm p}$ and $\dot c_{\mathrm n}$ denote perturbations of the plateau values, while $\dot z_j$ denotes the perturbation of the $j$th KL coefficient.

However, to evaluate the full gradient, we must repeatedly solve the sensitivity equation for each basis direction of the parameter space. To avoid this computational inefficiency, we introduce the adjoint variable $p\in V_0$, defined as the solution of
\begin{equation} \label{eq:adjoint}
    \int_{\Omega} \epsilon \nabla v \cdot \nabla p \, d\boldsymbol{x}
    +
    2\delta^2 \int_{\Omega}
    \cosh(w+g)vp \, d\boldsymbol{x}
    =
    \frac{1}{\sigma_{\mathrm{noise}}^2}
    \left\langle
    r,
    \mathcal O^h(v)
    \right\rangle_{\ell^2},
    \qquad \forall v\in V_0.
\end{equation}
Let $a(\cdot,\cdot)$ denote the bilinear form associated with the left-hand side of \eqref{eq:adjoint}. By the definition of the adjoint problem, evaluating at $v=\dot w$ yields
\begin{equation} \label{eq:adjoint-rhs1}
    a(\dot w,p)
    =
    \frac{1}{\sigma_{\mathrm{noise}}^2}
    \left\langle
    r,
    \mathcal O^h(\dot w)
    \right\rangle_{\ell^2},
\end{equation}
which corresponds to the directional derivative term of interest. On the other hand, using \(v=p\) in the sensitivity equation \eqref{eq:sensitivity}, we obtain
\begin{equation} \label{eq:adjoint-rhs2}
    a(\dot w,p)
    = \int_\Omega \dot c\,p\,d\boldsymbol x
    - \int_{\Omega} \epsilon \nabla \dot g \cdot \nabla p \, d\boldsymbol{x}
    -
    2\delta^2 \int_{\Omega}
    \cosh(w+g)\dot g p \, d\boldsymbol{x}.
\end{equation}

Combining \eqref{eq:DJ-expanded}, \eqref{eq:adjoint-rhs1}, and \eqref{eq:adjoint-rhs2}, we obtain
\begin{equation}
\begin{aligned}
D_{\boldsymbol{\theta}}\mathcal J[\dot{\boldsymbol{\theta}}]
=
&\int_\Omega \dot c\,p\,d\boldsymbol x
-
\int_{\Omega} \epsilon \nabla \dot g \cdot \nabla p \, d\boldsymbol{x}
-
2\delta^2 \int_{\Omega}
\cosh(w+g)\dot g\,p \, d\boldsymbol{x}
\\
&+
\frac{1}{\sigma_{\mathrm{noise}}^2}
\left\langle
r,
\mathcal O^h(\dot g)
\right\rangle_{\ell^2}
+
D_{\boldsymbol{\theta}}\Psi_\Theta(\boldsymbol{\theta})
[\dot{\boldsymbol{\theta}}].
\end{aligned}
\label{eq:adjoint-gradient-directional}
\end{equation}
The right-hand side of \eqref{eq:adjoint-rhs2} no longer depends on the sensitivity variable $\dot w$, but only on the adjoint variable $p$ and the perturbation directions $\dot g$ and $\dot c$. Therefore, for a fixed parameter value $\boldsymbol\theta$, the adjoint equation \eqref{eq:adjoint} needs to be solved only once. Subsequently, for each new direction $\dot{\boldsymbol{\theta}}$, the directional derivative can be evaluated by assembling the right-hand side of \eqref{eq:adjoint-rhs2}, the explicit observation term involving $\mathcal O^h(\dot g)$ in \eqref{eq:adjoint-gradient-directional}, and the prior contribution. The contributions of the negative log priors are obtained analytically from the corresponding prior distributions.

%The prior contribution $D_{\boldsymbol{\theta}}\Psi_\Theta(\boldsymbol\theta)[\dot{\boldsymbol\theta}]$ can be evaluated directly from the chosen prior model. The KL coefficient part follows from the standard Gaussian prior on \(\boldsymbol z\), giving the contribution \(\boldsymbol z\) to the gradient. Terms involving $c_{\mathrm p}$ and $c_{\mathrm n}$ depend on the prior distributions assigned to the plateau values.

\subsection{MAP estimation}
\label{subsec:MAP}

After the finite-dimensional approximation of the Bayesian formulation, the unknown is represented by a parameter vector
\(\boldsymbol\theta\in\mathbb R^d\), where 
% \(d=N_{\mathrm{KL}}\) in the known-plateau setting and
\(d=N_{\mathrm{KL}}+2\).
% in the joint-inference setting.
A point estimate is obtained by maximizing the finite-dimensional posterior density, or equivalently by minimizing the negative log-posterior functional \(\mathcal J\) defined in \eqref{eq:negative-log-posterior}. The MAP estimate is
therefore defined as
\begin{equation}
    \boldsymbol\theta_{\mathrm{MAP}}
    =
    \operatorname*{arg\,min}_{\boldsymbol\theta\in\mathbb R^d}
    \mathcal J(\boldsymbol\theta).
    \label{eq:MAP-estimate}
\end{equation}
% Due to the nonlinear dependence of the forward model on the doping field, together with the nonlinear pushforward parametrization of the prior, the objective functional \(\mathcal J\) is generally nonconvex and may admit multiple local minimizers. Consequently, the MAP estimate is computed using an iterative gradient-based optimization method.
The gradient \(\nabla_{\boldsymbol\theta}\mathcal J(\boldsymbol\theta)\) is computed using the adjoint formulation derived in subection~\ref{subsec:adjoint-gradient}.
%This avoids forming the full parameter-to-observable Jacobian and yields the gradient of the scalar objective with a number of PDE solves that is essentially independent of the parameter dimension. As a result, MAP estimation in the reduced latent parameter space becomes computationally feasible even when the underlying finite-element discretization is high-dimensional.

\subsection{Posterior sampling and uncertainty quantification}
\label{subsec:NUTS_UQ}

While the MAP estimate provides a single representative point of the posterior distribution, it does not characterize the range of parameter configurations that are consistent with the data and the prior. To quantify uncertainty in the reconstructed doping profile, we sample from the finite-dimensional posterior density \(\pi(\boldsymbol\theta\mid \boldsymbol y)\) defined in
\eqref{eq:finite-dimensional-posterior}.

Let \(\{\boldsymbol\theta^{(s)}\}_{s=1}^{N_s}\) denote samples generated by a Markov chain Monte Carlo method targeting \(\pi(\boldsymbol\theta\mid \boldsymbol y)\). The corresponding realizations of the doping field at each spatial point $\boldsymbol x\in\Omega$ are obtained through the pushforward map introduced in \eqref{eq:pf_map_finite} as
\[
c^{(s)}(\boldsymbol x)
=
F_2(\boldsymbol\theta^{(s)})(\boldsymbol x),
\qquad s=1,\ldots,N_s .
\]
Posterior moments of the doping profile are then approximated pointwise by ergodic averages,
\begin{equation*}
    \bar c(\boldsymbol x)
    =
    \frac{1}{N_s}
    \sum_{s=1}^{N_s}
    c^{(s)}(\boldsymbol x),   \qquad  \operatorname{Var}[c(\boldsymbol x)]
    \approx
    \frac{1}{N_s-1}
    \sum_{s=1}^{N_s}
    \left(
    c^{(s)}(\boldsymbol x)-\bar c(\boldsymbol x)
    \right)^2 .
\end{equation*}
These quantities provide spatially resolved measures of posterior uncertainty, highlighting regions where the reconstruction is strongly constrained by the data and regions where it remains mainly influenced by the prior.

To generate posterior samples, we employ the No-U-Turn Sampler (NUTS)
\cite{hoffman2014no}, an adaptive variant of Hamiltonian Monte Carlo. NUTS uses gradient information to construct long-distance proposals in parameter space and automatically adapts the integration path length, thereby reducing the need for manual tuning. The required gradient \(\nabla_{\boldsymbol\theta}\mathcal J(\boldsymbol\theta)\) is computed using the adjoint formulation derived in subection~\ref{subsec:adjoint-gradient}.

\section{Numerical Results }
\label{sec:numerical}

We evaluate the performance of the proposed Bayesian reconstruction framework on synthetic benchmark problems for the inverse problem described in Section \ref{ss:doping}. The experiments are designed to address the following key questions: (i)~Can the framework accurately recover the junction interface and doping plateau levels from noisy boundary flux data? (ii)~Is the method robust with respect to interface shape? (iii)~Does gradient-based sampling via NUTS yield a measurable efficiency gain over the gradient-free pCN sampler?

All simulations were performed in Python 3.11.13 on a MacBook Pro equipped with an Apple M4 Pro chip with 14 cores and 24\,GB unified memory, running macOS 26.4.1. The implementation was developed using FEniCS 2019.1.0, dolfin-adjoint 2018.1.0, PyTorch 2.7.1, and Pyro 1.9.1. The MAP estimation and pCN sampling algorithms were implemented using custom-developed Python code, while adjoint-based gradient computations were carried out using the dolfin-adjoint framework. The source code and example scripts are publicly available at \url{https://github.com/hassanyazdanian/adjoint-bayesian-semiconductor.git}.

% \textcolor{brown}{Do we need really to make the codes public? We have still a lot of ideas to continue in this direction... We shall say codes are available upon request (?) If you two are voting for making the code online, I am happy and will follow you :) }

% ---------------------------------------------------------------
\subsection{Experimental Setup}
\label{subsec:setup}
% ---------------------------------------------------------------

All experiments are conducted on the unit square domain $\Omega=(0,1)^2$, with boundary decomposition $\partial\Omega = \partial\Omega_1 \cup \partial\Omega_2 \cup \partial\Omega_3$ as defined in Section~\ref{ss:doping}. The domain is discretized on a uniform finite element mesh with $N=48$ elements per spatial direction, %using continuous piecewise-linear (CG1) finite element basis functions, 
using first order Lagrange basis functions, yielding $(48+1)^2 = 2401$ degrees of freedom for the state.

The ground-truth doping field $c^\dagger(\boldsymbol{x})$ is piecewise constant, with $p$-type plateau $c_{\mathrm{p}} = -2$ and $n$-type plateau $c_{\mathrm{n}} = 1$. Two representative configurations are considered: Example~1 employs a straight planar interface as a baseline geometry (Fig. \ref{fig:ex1_noise5}a), while Example~2 considers a curved interface motivated by geometrically structured or radial pn-junctions (Fig. \ref{fig:ex2_noise5}b). The latter induces a more nonlinear forward response and a more challenging posterior geometry.

Synthetic measurements are generated by solving the forward problem \eqref{eq:semi-PDE-strong} for the ground-truth doping field $c^\dagger$ to obtain $\boldsymbol y^{\mathrm{true}} = \mathcal G^h(c^\dagger)$, and then adding Gaussian noise,
\[
\boldsymbol y
=
\boldsymbol y^{\mathrm{true}}
+
\boldsymbol\varepsilon,
\qquad
\boldsymbol\varepsilon \sim \mathcal N(0,\sigma_\mathrm{noise}^2 I).
\]
The noise standard deviation $\sigma_\mathrm{noise}$ is set proportional to the root-mean-square of the demeaned noise-free signal, $\sigma_\mathrm{noise} = \rho\|\boldsymbol y^{\mathrm{true}}-\overline{y}^{\mathrm{true}}\boldsymbol 1\|_2/\sqrt{N_\mathrm{meas}}$, where $N_\mathrm{meas}$ is the number of measurements and $\overline{y}^{\mathrm{true}}$ is the sample mean of the noise-free data.

\subsection{Inference Scenarios}

We consider two inference settings. In the first setting, the plateau values \(c_{\mathrm p}\) and \(c_{\mathrm n}\) are assumed known, and only the latent Gaussian field controlling the interface geometry is inferred. Using the KL reconstruction map defined in \eqref{eq:KL-reconstruction-map}, the unknown is the coefficient vector \(\boldsymbol z\in\mathbb R^{N_{\mathrm{KL}}}\), and the negative log-posterior is \begin{equation*}
\mathcal J(\boldsymbol z) = 
\frac{1}{2\sigma^2_{\mathrm{noise}}} \left\| \mathcal G^h\big(F_2(\boldsymbol z)\big) - \boldsymbol y \right\|_2^2 + \frac12\|\boldsymbol z\|_2^2 . 
\end{equation*} 
Here, \(F_2(\boldsymbol z)\) is understood as
\(F_2\big(\mathcal R_N(\boldsymbol z),c_{\mathrm p},c_{\mathrm n}\big)\),
as in \eqref{eq:pf_map_finite}, for fixed \(c_{\mathrm p}\) and \(c_{\mathrm n}\). The operator \(\mathcal G^h\) denotes the discretized forward map from the doping profile to the measured boundary flux. Since the Dirichlet boundary values are fixed, the lifting function is independent of the inferred parameters.
% We consider two inference settings. In the first setting, the plateau values \(c_{\mathrm p}\) and \(c_{\mathrm n}\) are assumed known, and only the latent Gaussian field controlling the interface geometry is inferred. In the truncated KL expansion \eqref{eq:truncated_KL} \bbkc{notation issue sneaks here too}, the unknown is the coefficient vector \(\boldsymbol x\in\mathbb R^{N_{\mathrm{KL}}}\), and the negative log-posterior is
% \begin{equation*}
%     \mathcal J(\boldsymbol x)
%     =
%     \frac{1}{2\sigma^2_{\mathrm{noise}}}
%     \left\|
%     \mathcal G^h\big(F_2(\boldsymbol x)\big)
%     -
%     \boldsymbol y
%     \right\|_2^2
%     +
%     \frac12\|\boldsymbol x\|_2^2 .
%     \label{eq:posterior-known-plateau}
% \end{equation*}
% Here \(\mathcal G^h\) denotes the discretized forward map from the KL coefficients to the measured boundary flux. Since the Dirichlet boundary values are fixed, the lifting function is independent of the inferred parameters.
%, and the simplified gradient expression \eqref{eq:NLL-gradient-no-boundary} applies.

In the second setting, the plateau values are inferred jointly with the interface geometry. The parameter vector is \(\boldsymbol\theta=(\boldsymbol z,c_{\mathrm p},c_{\mathrm n})\), and the negative log-posterior takes the form
\begin{equation*}
    \mathcal J(\boldsymbol\theta)
    =
    \frac{1}{2\sigma^2_{\mathrm{noise}}}
    \left\|
    \mathcal G^h\big(F_2(\boldsymbol\theta)\big)
    -
    \boldsymbol y
    \right\|_2^2
    +
    \Psi(\boldsymbol\theta),
\end{equation*}
where \(\Psi(\boldsymbol\theta)\) contains the Gaussian contribution
\(\frac12\|\boldsymbol z\|_2^2\) and the prior contributions of the plateau parameters. The plateau values are assigned independent uniform priors on prescribed intervals,
\[
c_{\mathrm p} \sim \mathcal U(c_{\mathrm p}^{\min},c_{\mathrm p}^{\max}),
\qquad
c_{\mathrm n} \sim \mathcal U(c_{\mathrm n}^{\min},c_{\mathrm n}^{\max}).
\]
For sampling, we introduce unconstrained variables \(r_{\mathrm p}\) and
\(r_{\mathrm n}\), which are mapped to the physical intervals by
\[
c_{\mathrm p}
=
c_{\mathrm p}^{\min}
+
(c_{\mathrm p}^{\max}-c_{\mathrm p}^{\min})S(r_{\mathrm p}),
\qquad
S(r)=\frac{1}{1+\exp(-r)},
\]
and analogously for \(c_{\mathrm n}\). The corresponding Jacobian correction is included in \(\Psi\), so that the induced priors on \(c_{\mathrm p}\) and \(c_{\mathrm n}\) remain uniform. In this setting, the parameters affect both the interior doping field and the Dirichlet boundary conditions.%, requiring the full adjoint gradient formulation \eqref{eq:NLL-gradient-entry}.

% ---------------------------------------------------------------
\subsection{Interface Reconstruction with Known Plateau Values}
\label{subsec:scenario_A}
% ---------------------------------------------------------------
We first examine the known-plateau setting for the two junction geometries shown in Figs.~\ref{fig:ex1_noise5}(a) and~\ref{fig:ex2_noise5}(a). The covariance operator in \eqref{eq:matern-cov} is used with \(\tau=\ell^{-2}\) and \(\alpha=\nu+1\), where \(\ell=0.15\) and \(\nu=3\), corresponding to \(\tau=0.15^{-2}\) and \(\alpha=4\). In the numerical implementation, the eigenvalue problem associated with \eqref{eq:matern-cov} is solved with homogeneous Neumann boundary conditions. The truncated KL expansion is constructed using a variance-retention criterion, retaining \(99.8\%\) of the prior variance, which results in \(N_{\mathrm{KL}}=30\) coefficients. The sigmoid steepness parameter is set to \(k=20\), producing sharp but everywhere differentiable transitions in the doping field. Figure~\ref{fig:prior_samples} shows representative prior realizations.

\begin{figure}
    \centering
    \includegraphics[width=1.0\linewidth]{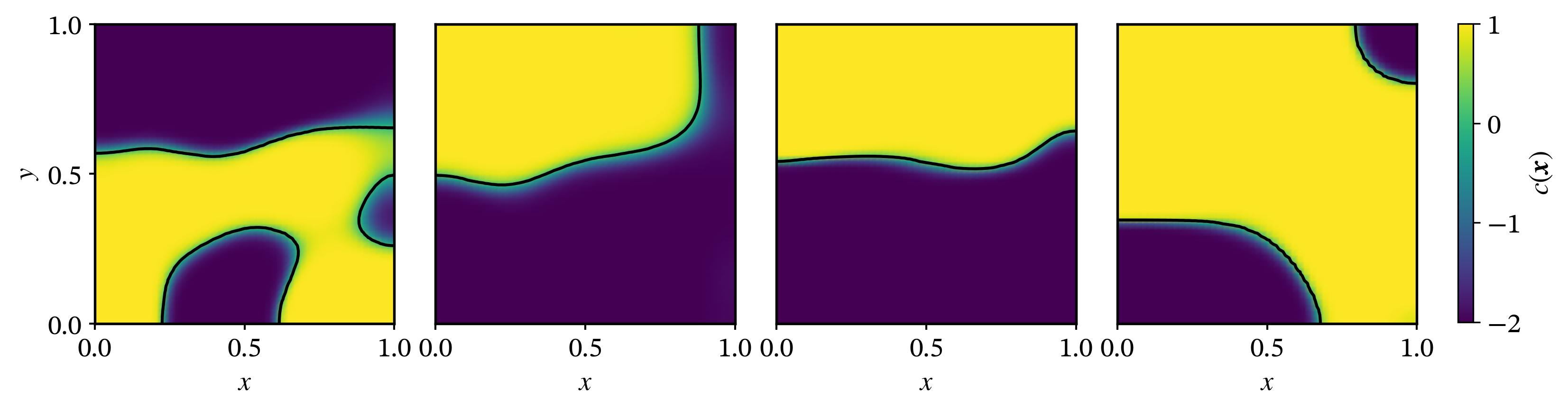}
    \caption{Representative realizations from the pushforward Whittle--Mat\'ern prior for the doping field \(c(\boldsymbol{x})\). Samples are generated on a \(48\times 48\) finite-element mesh using a truncated KL expansion with \(N_{\mathrm{KL}}=30\). The covariance parameters are chosen as \(\tau=\ell^{-2}\) and \(\alpha=\nu+1\), with \(\ell=0.15\) and \(\nu=3\), and the sigmoid steepness parameter is set to \(k=20\). The solid black contour indicates the zero-level set \(c(\boldsymbol{x})=0\), corresponding to a candidate pn-junction interface.}
    \label{fig:prior_samples}
\end{figure}

To generate synthetic data, we consider the noise level parameter $\rho = 0.05$. We compare MAP estimation, NUTS, and preconditioned Crank--Nicolson (pCN). For the pCN sampler, a single chain is run for $300{,}000$ iterations, with the first $30{,}000$ samples discarded as burn-in. For NUTS, the warmup phase consists of $500$ iterations, followed by $1000$ posterior samples.

\subsubsection{Example 1: Straight Junction}
Fig.~\ref{fig:ex1_noise5}(d) shows the MAP estimate. In our experiments, the MAP estimate alone does not provide a reliable characterization of the posterior distribution, as it exhibits strong sensitivity to initialization: different starting points may lead to qualitatively different interface geometries, indicating the presence of multiple local minima. %This nonconvex behavior arises from the nonlinear forward model together with the limited boundary measurements.

To quantify reconstruction quality, we compute the relative \(L^2\)-error,
\begin{equation}
    \mathrm{RE}
    =
    \frac{\|\bar{c} - c^\dagger\|_{L^2}}{\|c^\dagger\|_{L^2}},
\end{equation}
where the \(L^2\)-norm is taken over the computational domain. This gives \(\mathrm{RE}=0.15\). We also evaluate the structural similarity index (SSIM)~\cite{wang2004image}, which measures the agreement of spatial structures, local contrasts, and interface sharpness on a scale from \(0\) (no similarity) to \(1\) (perfect agreement). The reconstruction attains \(\mathrm{SSIM}=0.84\). Together, these metrics indicate the limited accuracy of a single point summary in this setting.

\begin{figure}[htbp]
\centering

% ---------------- Row 1: ----------------
\begin{subfigure}[t]{0.305\textwidth}
    \centering
    \includegraphics[width=\linewidth]{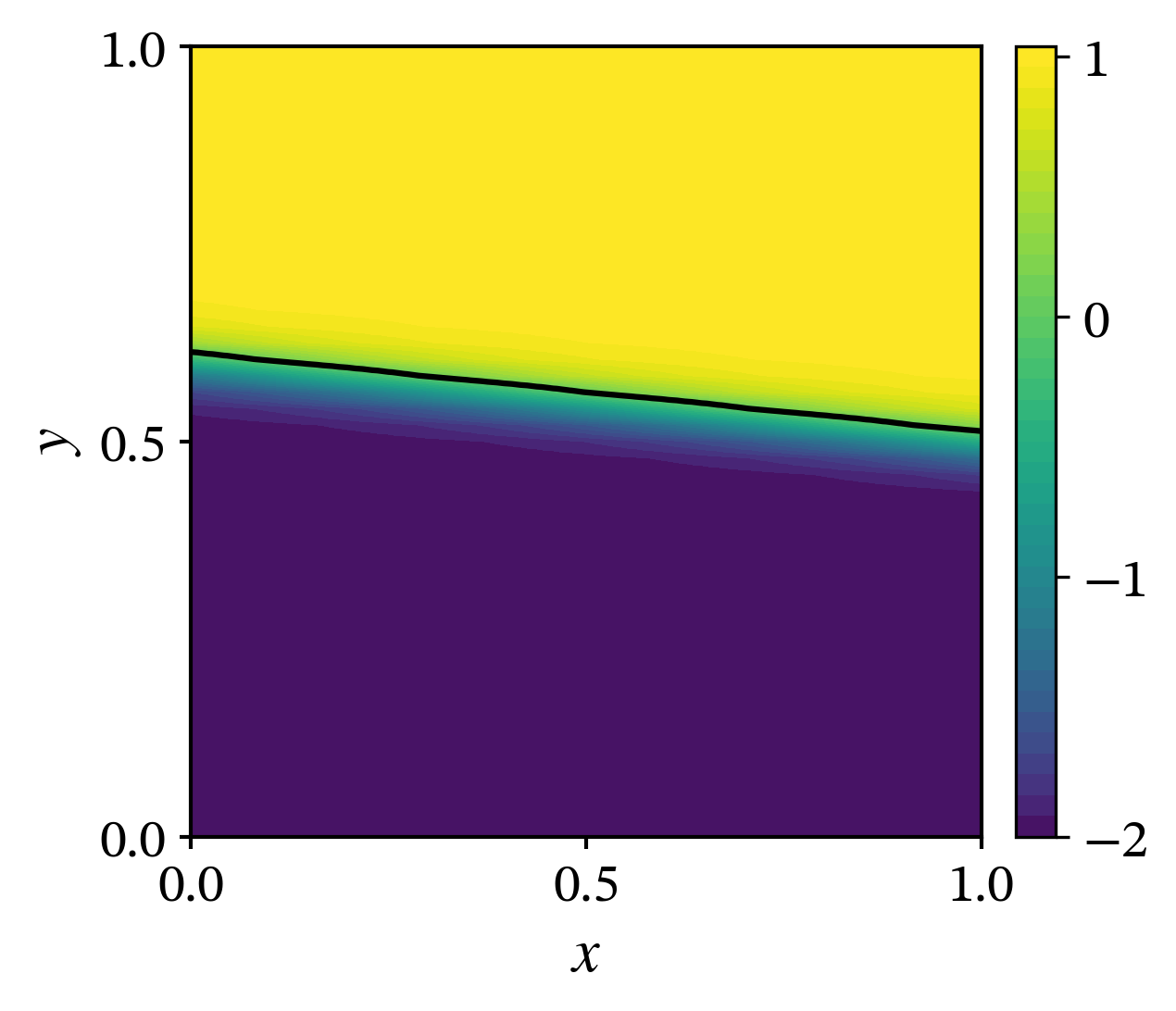}
    \caption{True doping profile}
\end{subfigure}
\hfill
\begin{subfigure}[t]{0.34\textwidth}
    \centering
    \includegraphics[width=\linewidth]{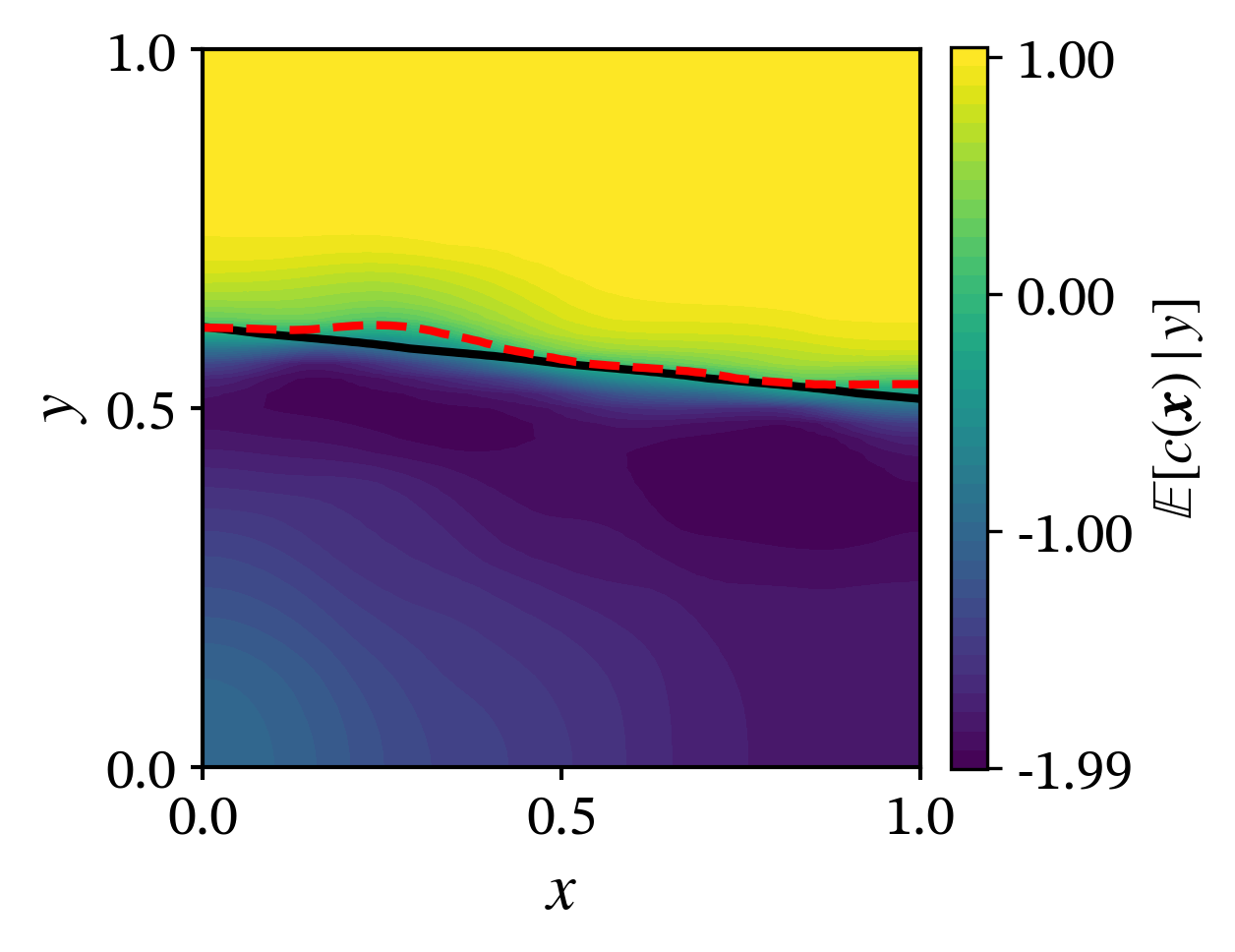}
    \caption{Posterior mean (pCN)}
\end{subfigure}
\hfill
\begin{subfigure}[t]{0.34\textwidth}
    \centering
    \includegraphics[width=\linewidth]{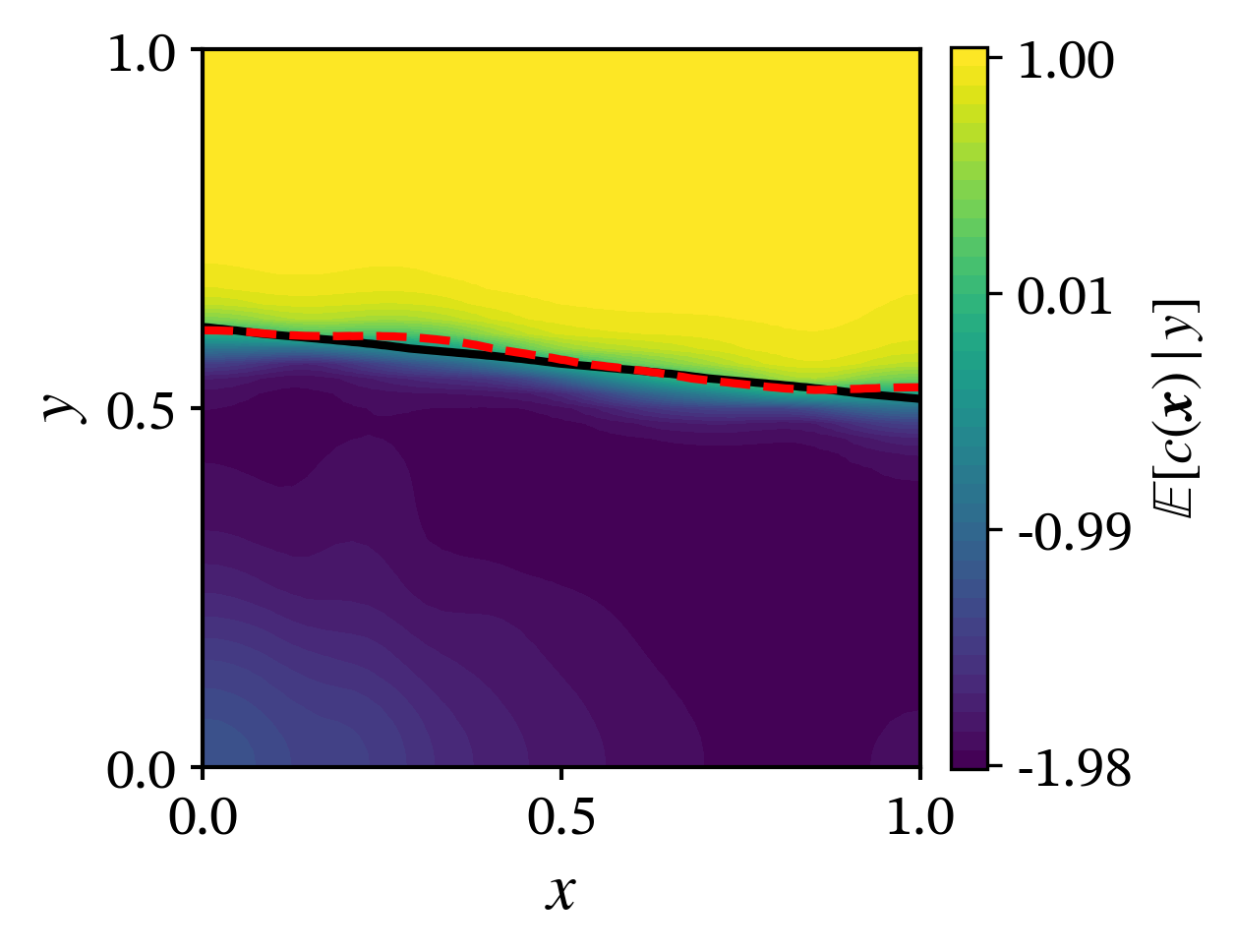}
    \caption{Posterior mean (NUTS)}
\end{subfigure}

\vspace{0.3cm}

% ---------------- Row 2: ----------------
\begin{subfigure}[t]{0.328\textwidth}
    \centering
    \includegraphics[width=\linewidth]{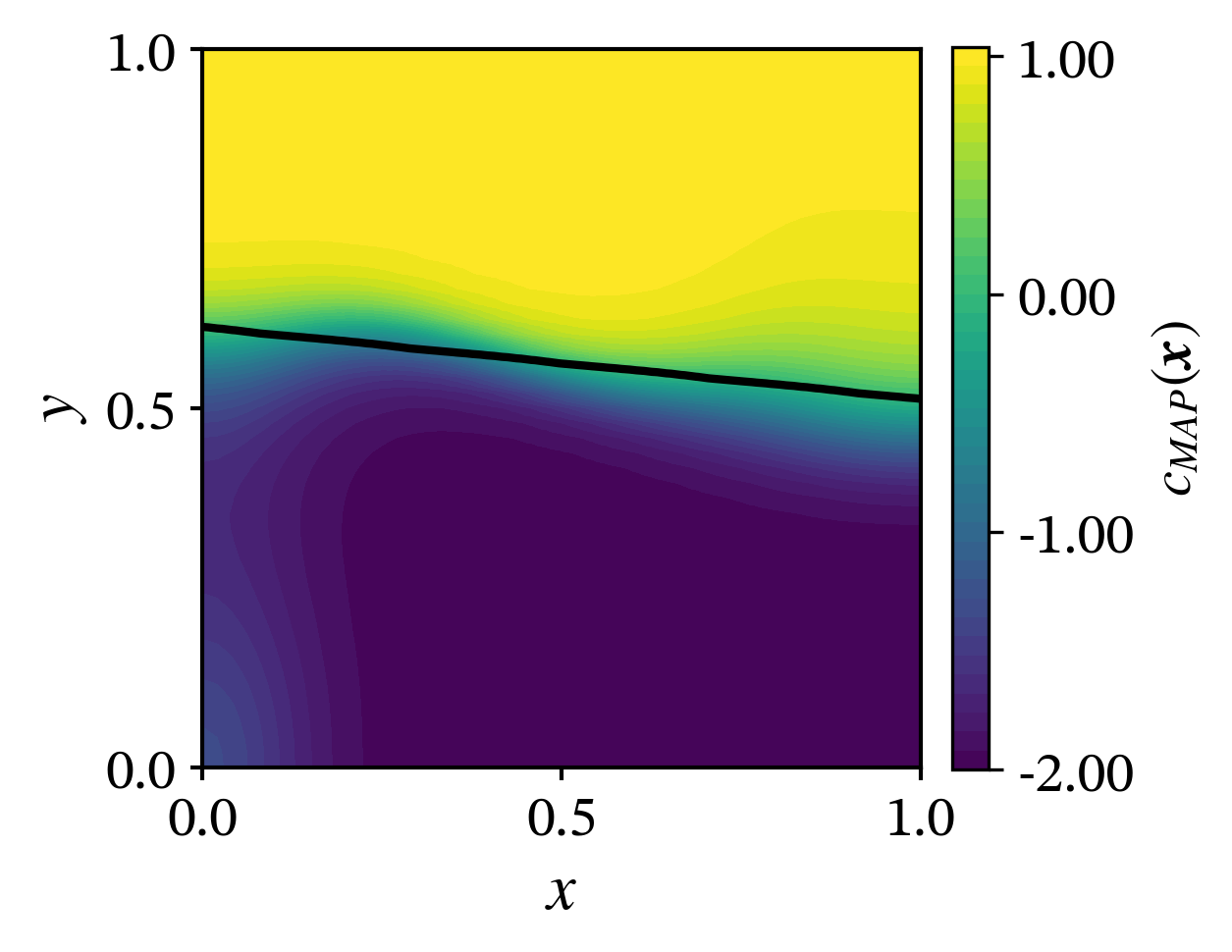}
    \caption{MAP estimate.}
\end{subfigure}
% \begin{subfigure}[t]{0.31\textwidth}
%     \centering
%     \vspace{0pt}
%     \rule{0pt}{0.1\textwidth}
% \end{subfigure}
\hfill
\begin{subfigure}[t]{0.328\textwidth}
    \centering
    \includegraphics[width=\linewidth]{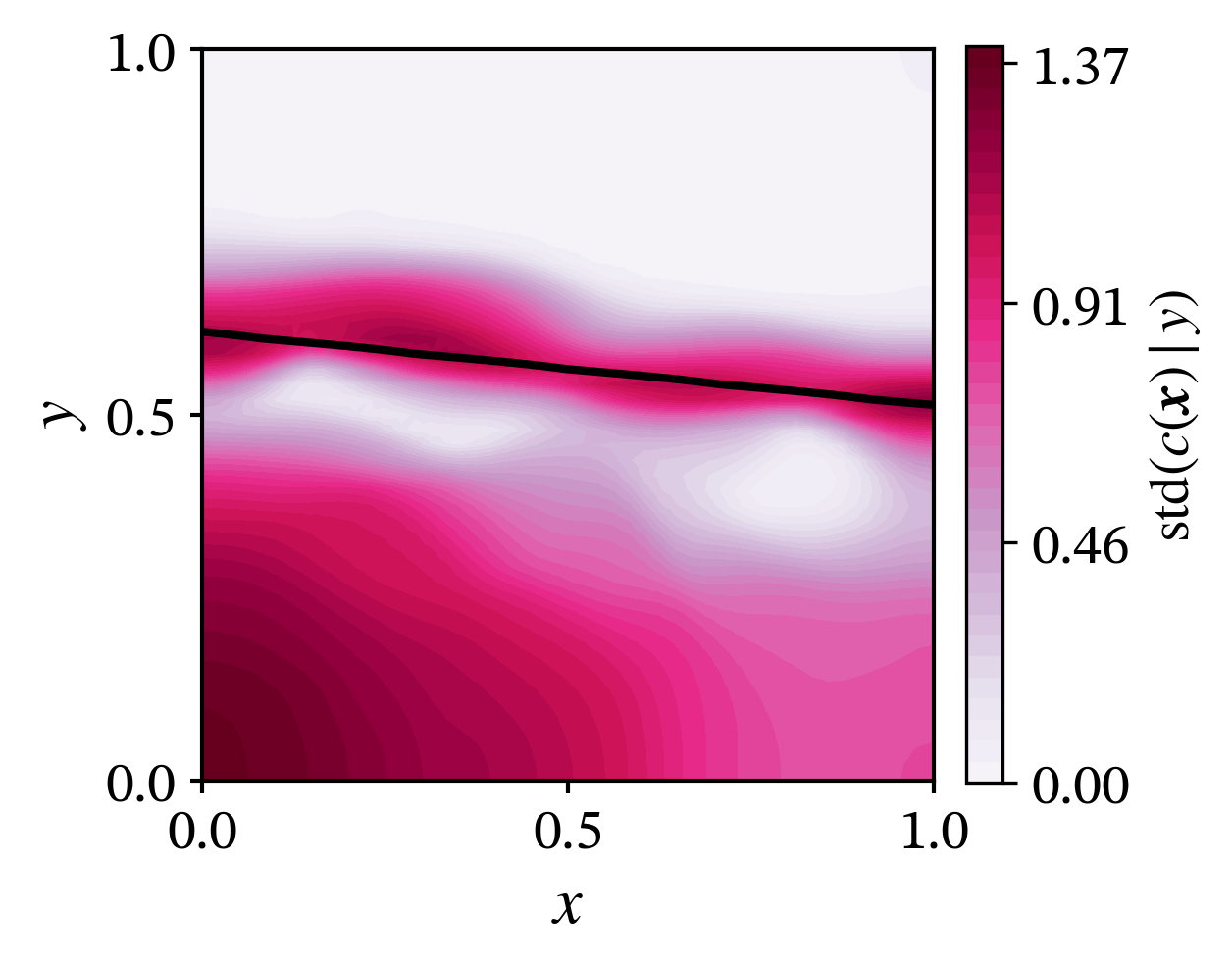}
    \caption{Posterior std (pCN)}
\end{subfigure}
\hfill
\begin{subfigure}[t]{0.328\textwidth}
    \centering
    \includegraphics[width=\linewidth]{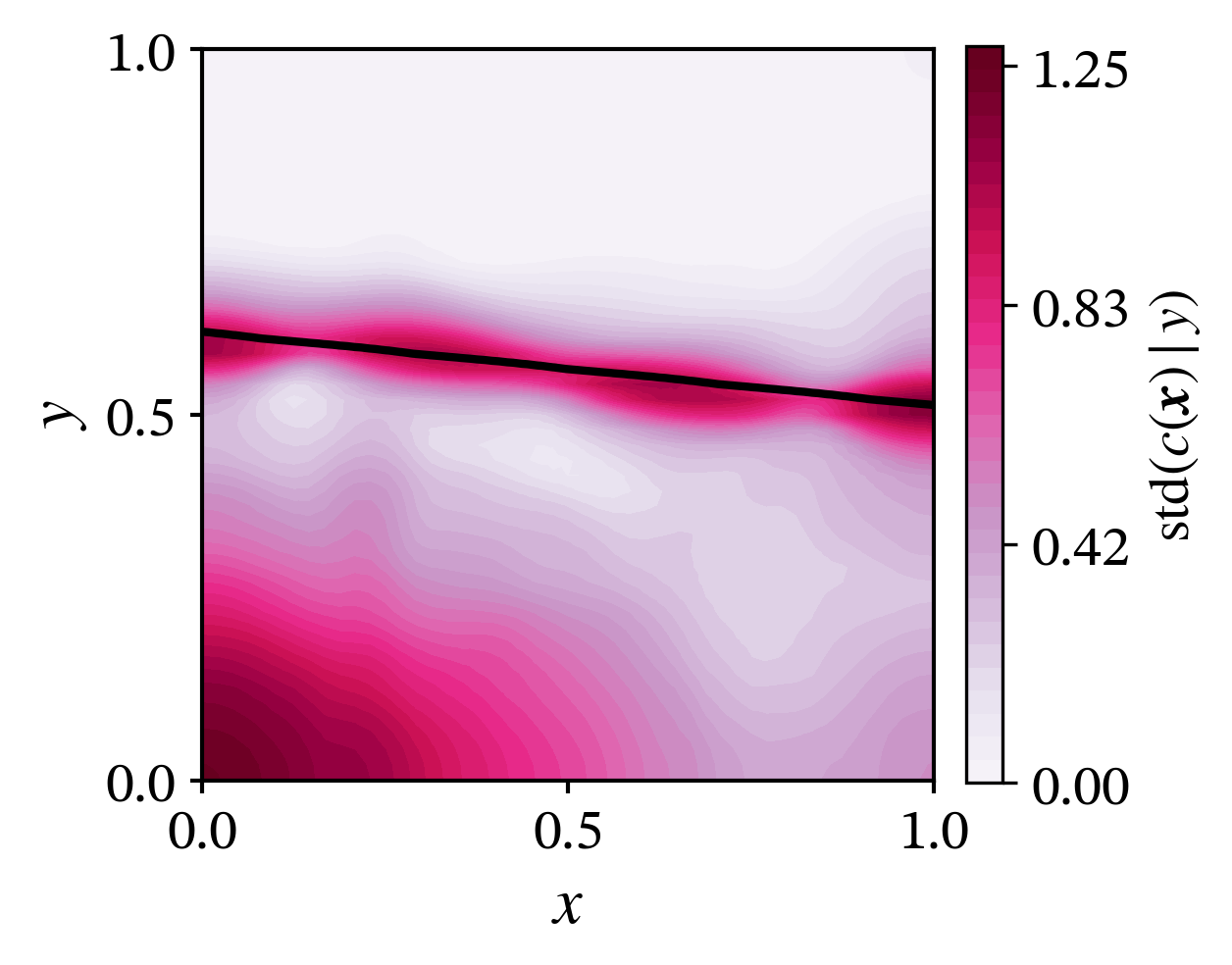}
    \caption{Posterior std (NUTS)}
\end{subfigure}

\caption{Reconstruction results for the known-plateau setting with a straight interface (Example~1). (a) True doping profile, (b) posterior mean obtained using pCN, (c) posterior mean obtained using NUTS, (d) MAP estimate, (e) posterior standard deviation (std) obtained using pCN, and (f) posterior standard deviation obtained using NUTS. The black solid line indicates the true pn-junction interface, while the red dashed line indicates the inferred interface.}
\label{fig:ex1_noise5}
\end{figure}

Figs.~\ref{fig:ex1_noise5}(b) and~\ref{fig:ex1_noise5}(e) show the posterior mean and standard deviation (std) obtained using pCN. The posterior mean provides a reasonable reconstruction, with \(\mathrm{RE}=0.20\) and \(\mathrm{SSIM}=0.72\). As expected, the uncertainty is higher around the interface and in regions farther from the measurement boundary. The total runtime is approximately \(1.33\) hours. However, autocorrelation analysis reveals severely degraded mixing across the inferred modes. The Markov chain exhibits extremely slow exploration of the posterior distribution, with a mean effective sample size (ESS) \cite{mcbook} of \(291\) and normalized efficiency \(\mathrm{ESS}/N_s \approx 10^{-3}\), where \(N_s\) denotes the number of post-burn-in samples. The trace plots in Fig.~\ref{fig:Trace_Scen_1}(a) further show that the chain evolves through small, highly correlated moves and remains confined to local regions of the parameter space. These diagnostics indicate that, for the present posterior and tuning, pCN provides statistically inefficient exploration despite producing a visually reasonable posterior mean.

Figs.~\ref{fig:ex1_noise5}(c) and~\ref{fig:ex1_noise5}(f) show the posterior mean and standard deviation obtained using NUTS. The posterior mean accurately recovers the planar interface, while the posterior standard deviation is sharply localized in a narrow band around the junction and increases toward the lower part of the domain, consistent with the reduced sensitivity of the measurements to regions farther from the boundary. This spatially structured uncertainty indicates that the data strongly constrain the interface location while leaving regions farther from the measurements less informed. The posterior mean achieves \(\mathrm{RE}=0.11\) and \(\mathrm{SSIM}=0.89\). The total runtime is approximately \(4.50\) hours. The mean ESS is \(729\), corresponding to \(\mathrm{ESS}/N_s \approx 0.73\). This efficiency is more than two orders of magnitude higher than that of pCN. The trace plots in Fig.~\ref{fig:Trace_Scen_1}(b) show well-mixed chains without visible stagnation, indicating efficient exploration of the posterior distribution.

The main numerical observations are summarized in Table~\ref{tab:ex1_sampler_comparison}. Taken together, these results show that the posterior mean and standard deviation fields obtained by pCN and NUTS are visually similar in Example~1, but the sampling diagnostics reveal a substantial difference in statistical reliability. NUTS provides accurate interface reconstruction together with efficient posterior exploration, as indicated by its substantially higher normalized efficiency, \(\mathrm{ESS}/N_s\), and well-mixed trace plots. In contrast, although pCN produces similar-looking posterior summaries, its poor mixing and extremely low effective sample size indicate less effective exploration of the posterior distribution. Consequently, the uncertainty estimates obtained from pCN should be interpreted with caution and are less statistically reliable than those obtained with NUTS. Furthermore, the results for Example~1 show that the MAP estimate is highly sensitive to initialization in our implementation. For this reason, MAP estimates are not considered further in Example~2.

\begin{figure}[!t]
\centering

\begin{subfigure}[t]{0.48\textwidth}
    \centering
    \includegraphics[width=\linewidth]{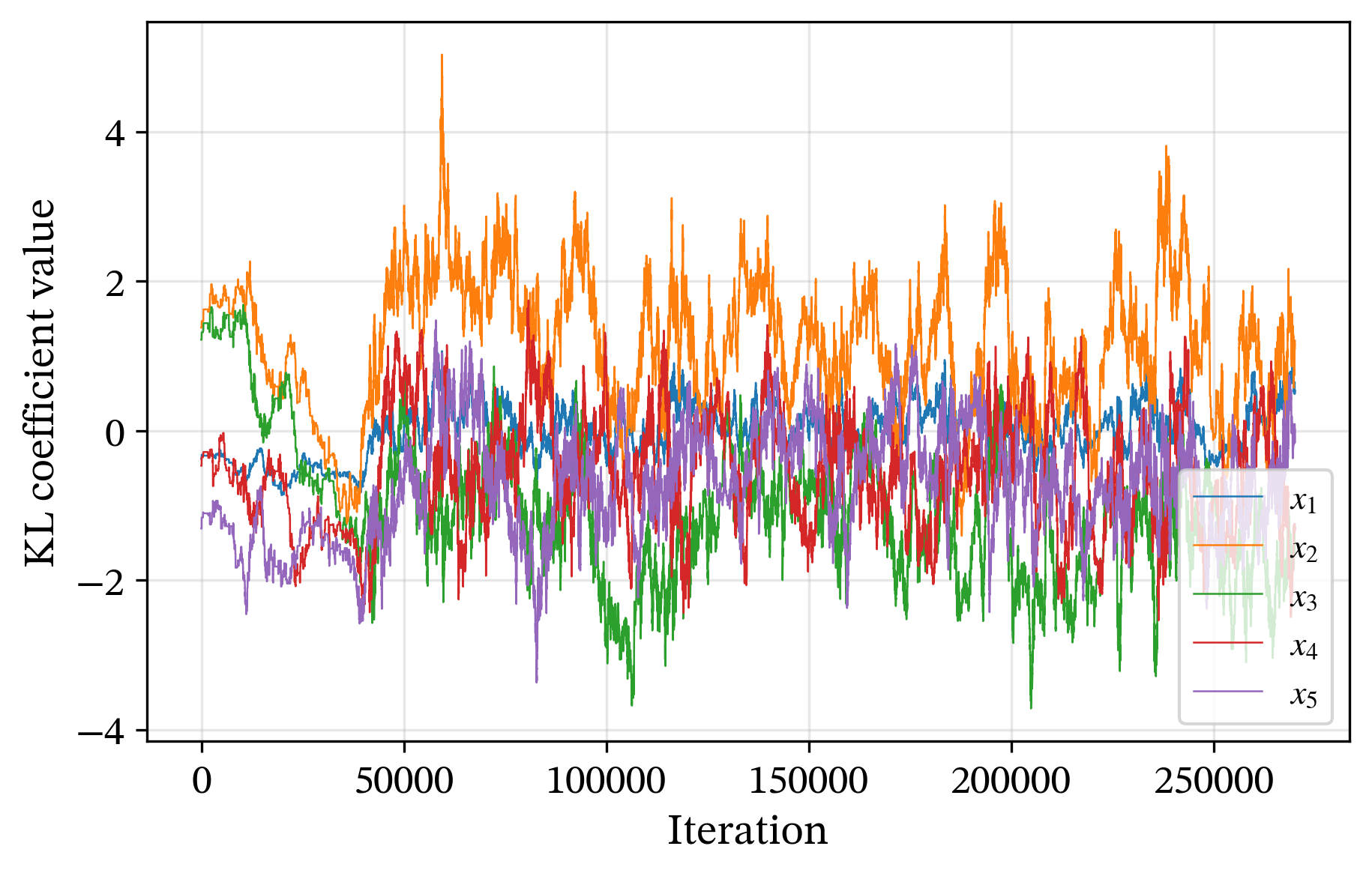}
    \caption{Example~1 (pCN)}
\end{subfigure}
\hfill
\begin{subfigure}[t]{0.48\textwidth}
    \centering
    \includegraphics[width=\linewidth]{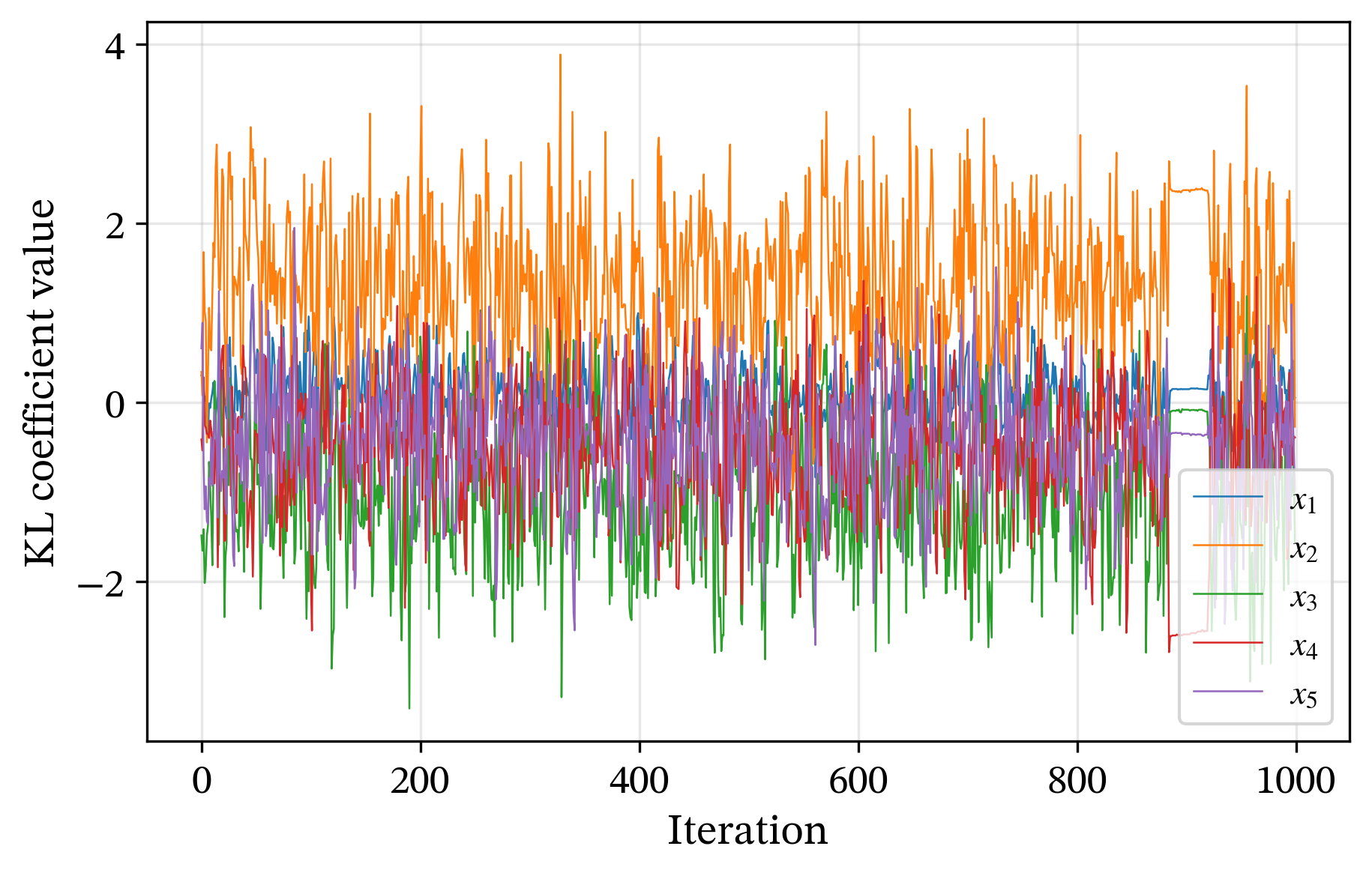}
    \caption{Example~1 (NUTS)}
\end{subfigure}

\vspace{1em}
\begin{subfigure}[t]{0.48\textwidth}
    \centering
    \includegraphics[width=\linewidth]{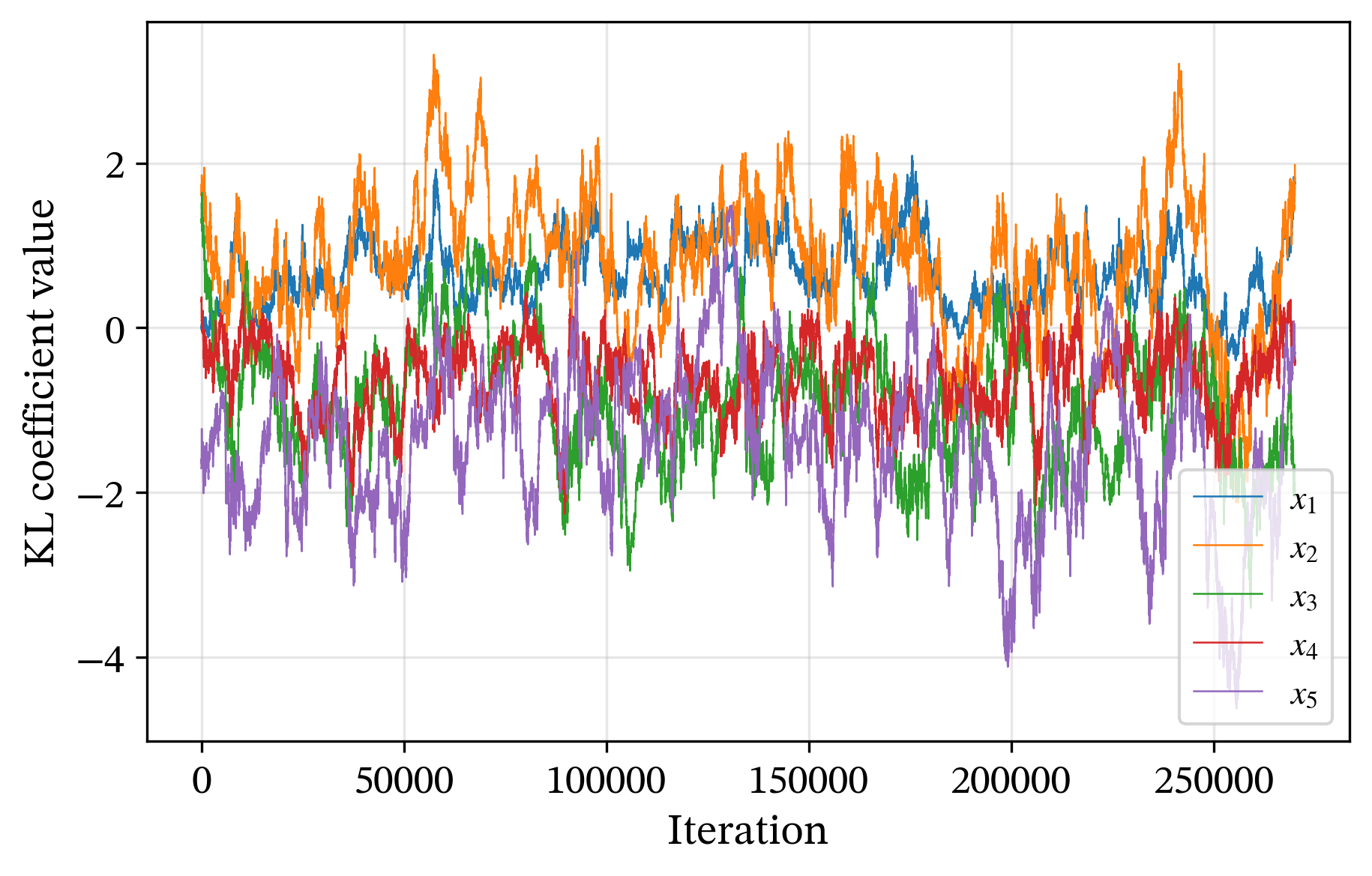}
    \caption{Example~2 (pCN)}
\end{subfigure}
\hfill
\begin{subfigure}[t]{0.48\textwidth}
    \centering
    \includegraphics[width=\linewidth]{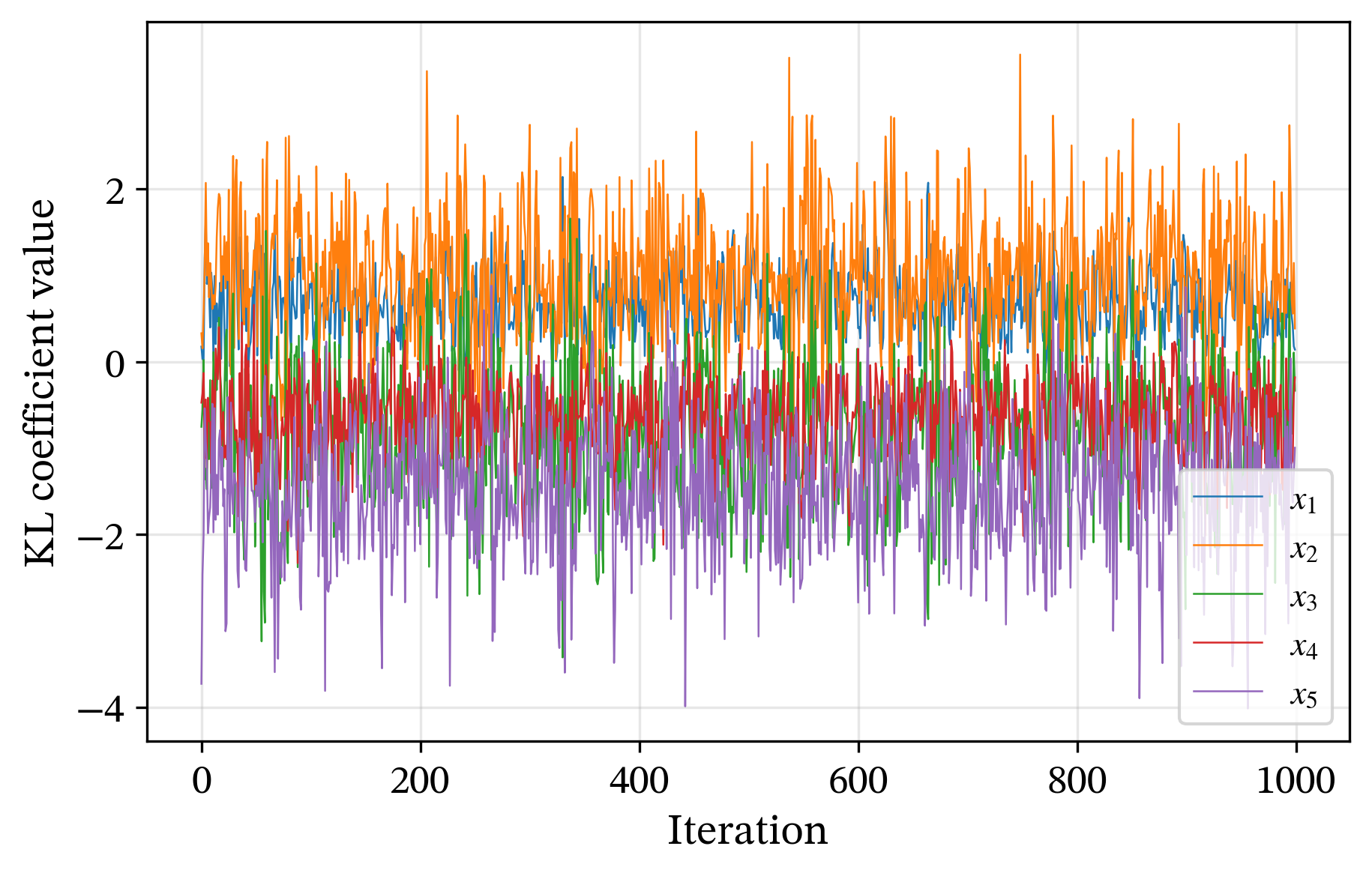}
    \caption{Example~2 (NUTS)}
\end{subfigure}

\caption{Trace plots of five representative KL coefficients for the reconstruction problem with known plateau values. First row: Example~1 using (a) pCN and (b) NUTS. Second row: Example~2 using (c) pCN and (d) NUTS. In both examples, the NUTS chains show faster mixing and better exploration of the posterior distribution, while the pCN chains exhibit strong autocorrelation and slow mixing, with long periods of persistence in localized regions of the parameter space.}
\label{fig:Trace_Scen_1}
\end{figure}

\begin{table}[t]
\centering
\caption{Comparison of pCN and NUTS for Example~1 with known plateaus.}
\label{tab:ex1_sampler_comparison}
\begin{tabular}{lcccccc}
\hline
Method & RE & SSIM & Runtime (h) & Mean ESS & \(\mathrm{ESS}/N_s\) & Mixing behaviour \\
\hline
pCN  & 0.20 & 0.72 & 1.33 & 291 & \(\approx 10^{-3}\) & highly correlated \\
NUTS & 0.11 & 0.89 & 4.50  & 729 & \(\approx 0.73\)     & well mixed \\
\hline
\end{tabular}
\end{table}

\subsubsection{Example 2: Curved Junction}

Compared to the planar case, the curved junction represents a geometrically more complex doping interface, closer to shapes that may arise in practical diode fabrication. Since the measurements are restricted to boundary fluxes, resolving spatial variations in the junction location is expected to be more challenging.

\begin{figure}[!htbp]
\centering
% ---------------- Row 1: ----------------
\begin{subfigure}[t]{0.305\textwidth}
    \centering
    \includegraphics[width=\linewidth]{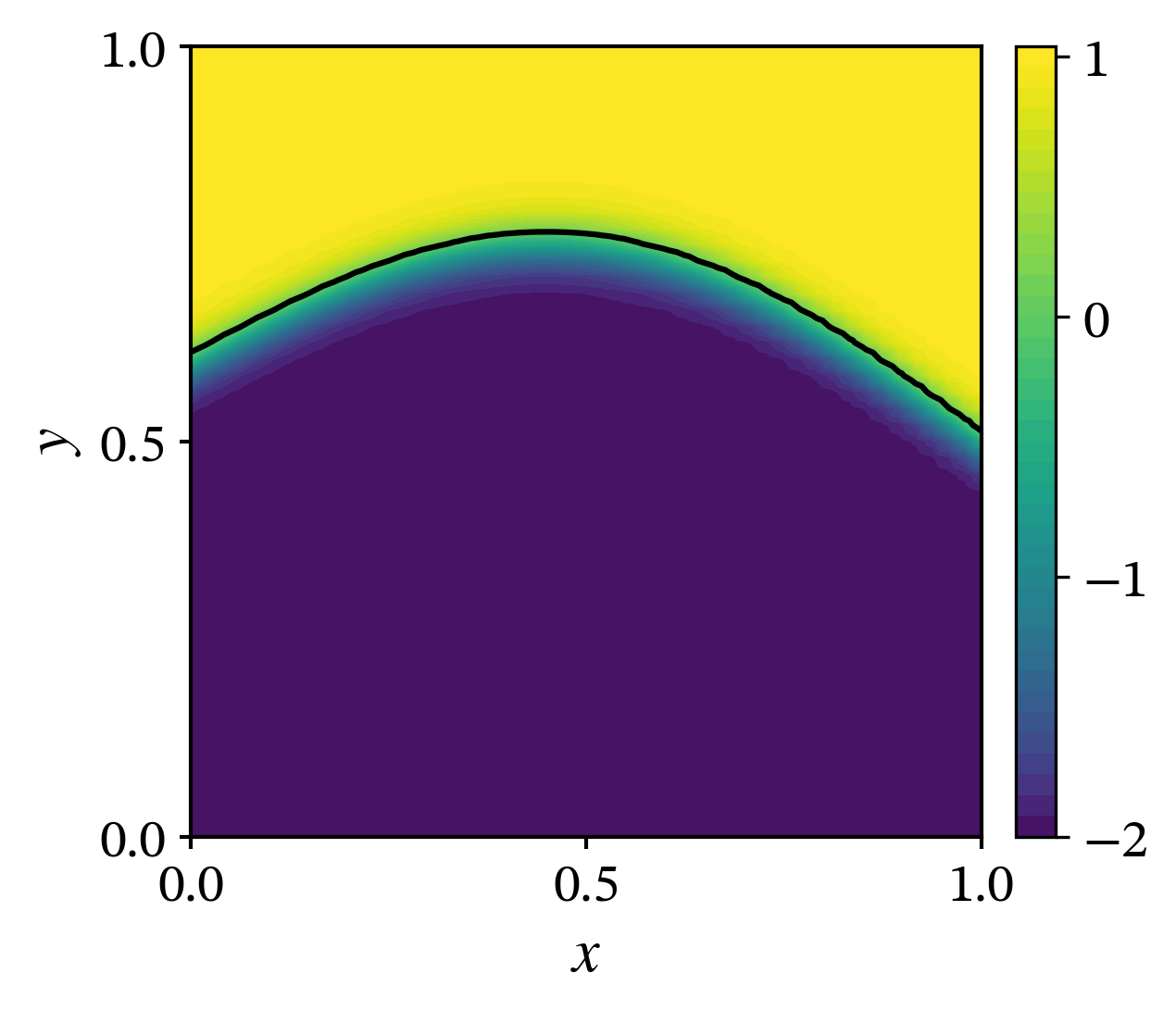}
    \caption{True doping profile}
\end{subfigure}
\hfill
\begin{subfigure}[t]{0.34\textwidth}
    \centering
    \includegraphics[width=\linewidth]{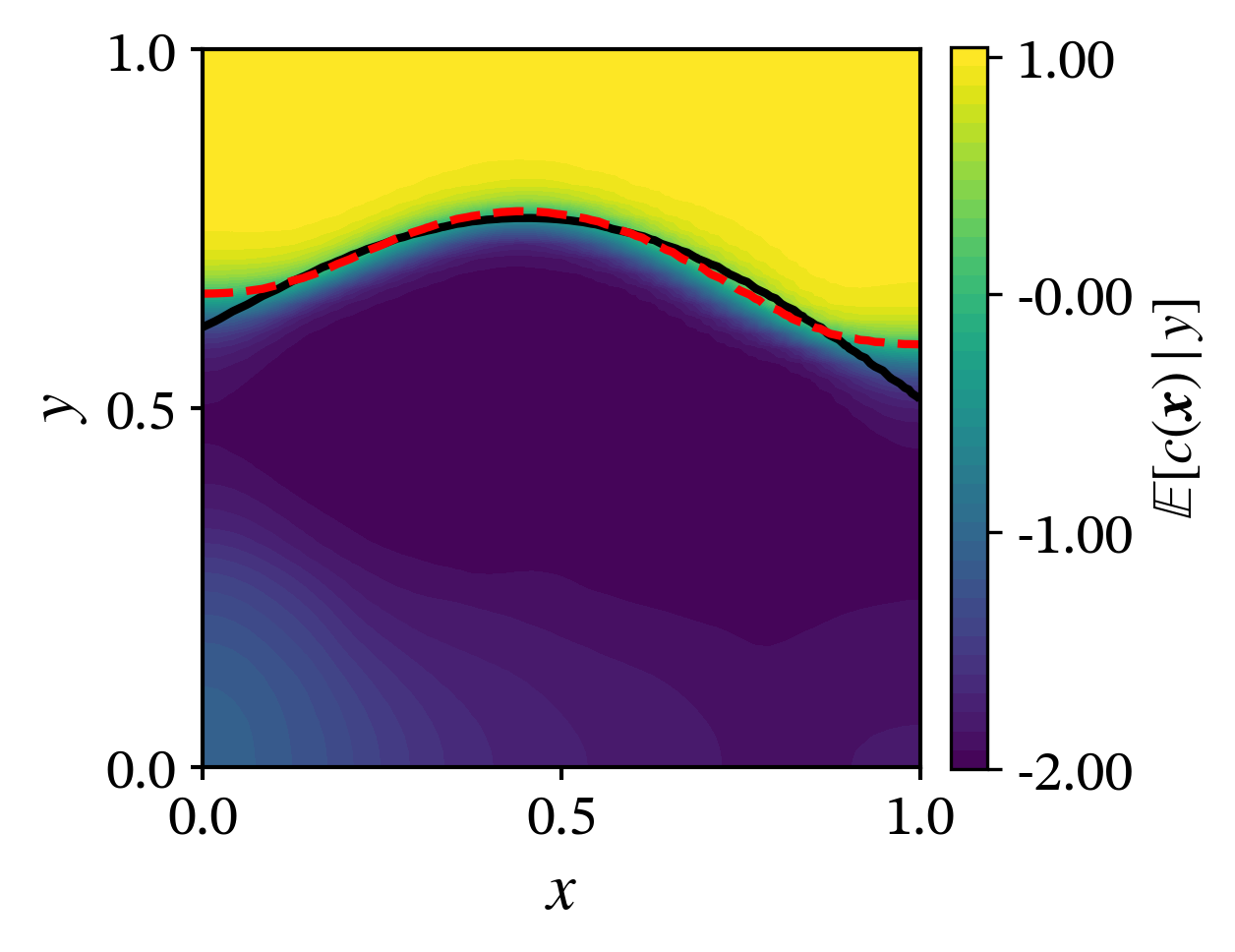}
    \caption{Posterior mean (pCN)}
\end{subfigure}
\hfill
\begin{subfigure}[t]{0.34\textwidth}
    \centering
    \includegraphics[width=\linewidth]{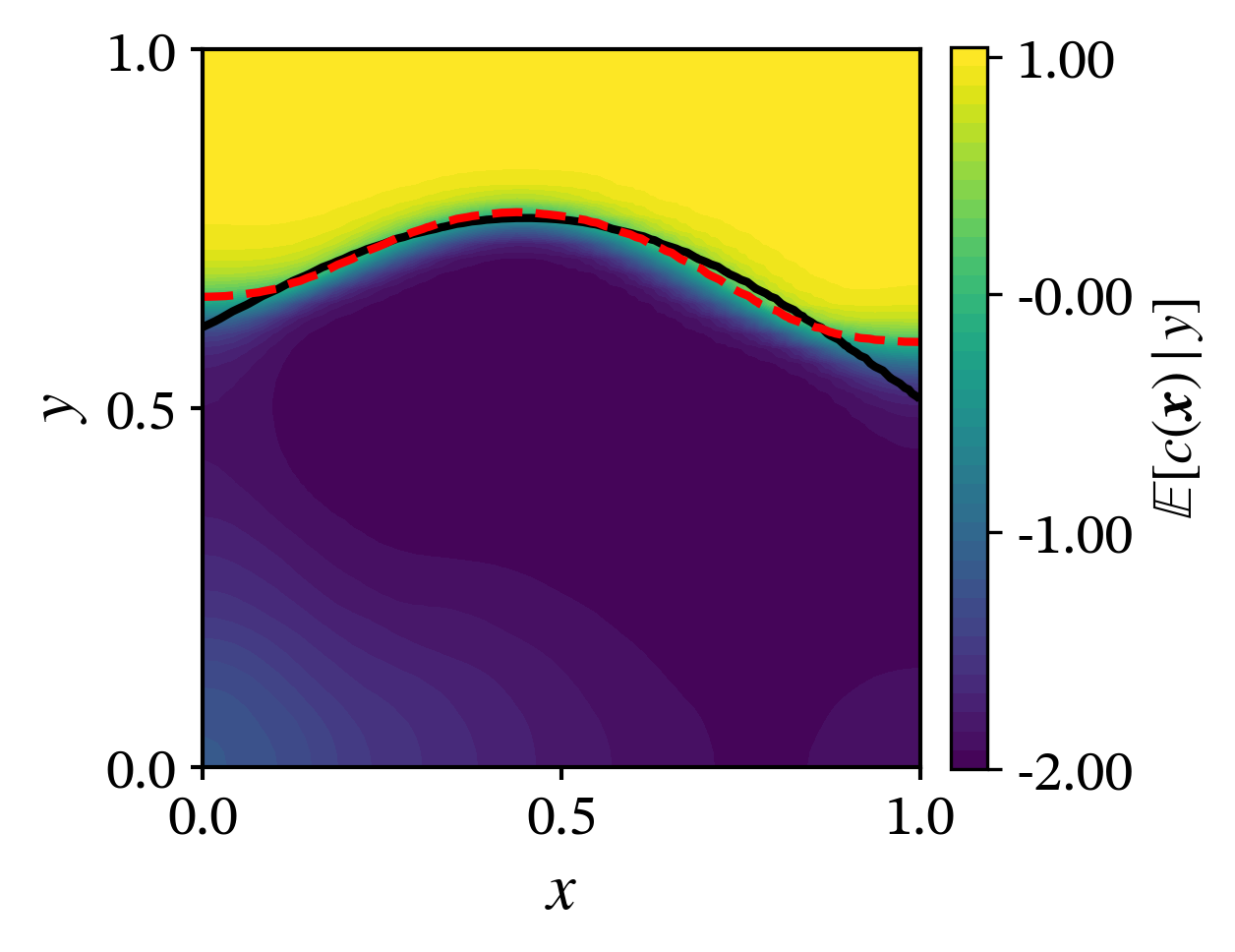}
    \caption{Posterior mean (NUTS)}
\end{subfigure}

\vspace{0.3cm}

% ---------------- Row 2: ----------------
\begin{subfigure}[t]{0.305\textwidth}
    \centering
    \vspace{0pt}
    \rule{0pt}{0.1\textwidth}
\end{subfigure}
\hfill
\begin{subfigure}[t]{0.34\textwidth}
    \centering
    \includegraphics[width=\linewidth]{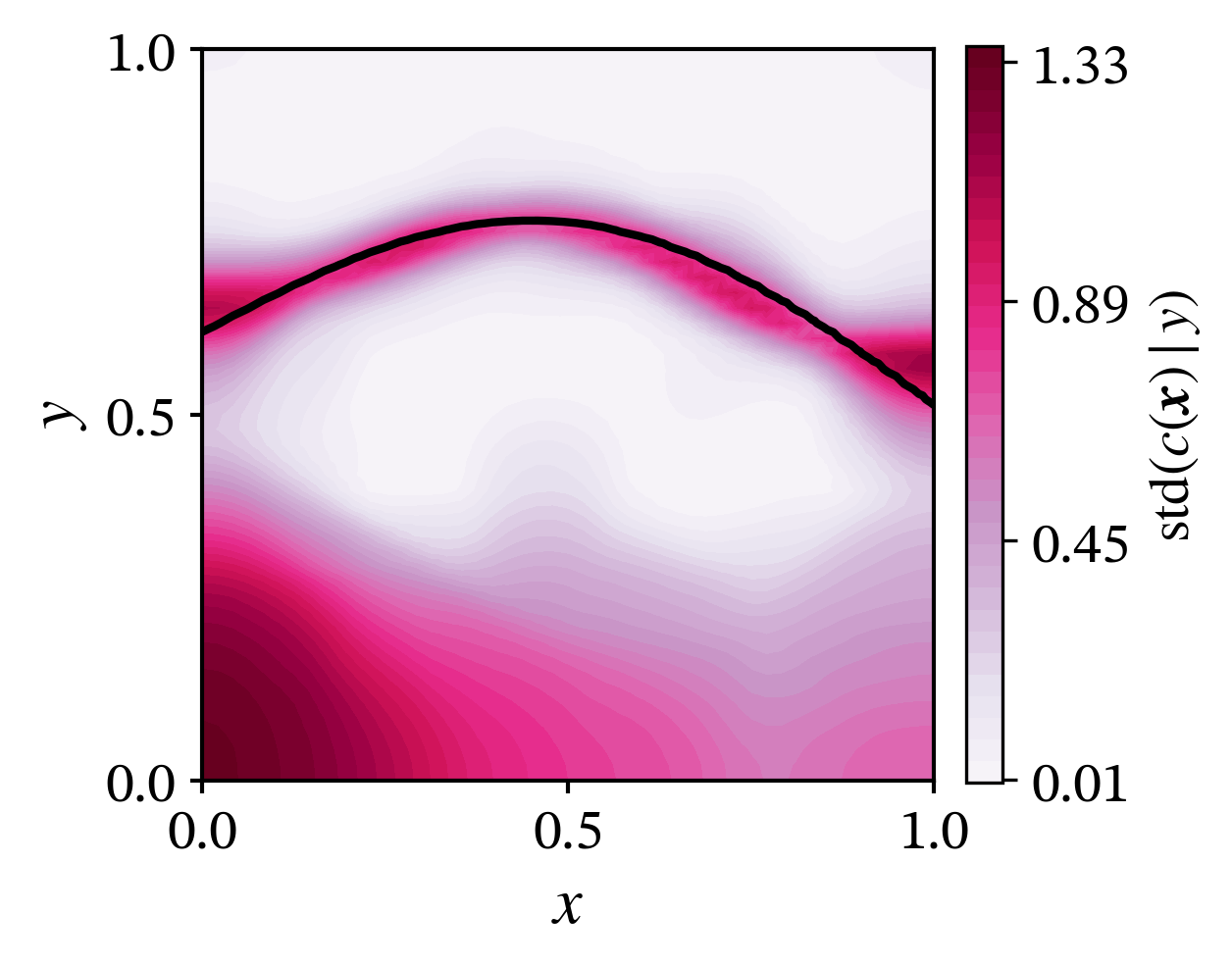}
    \caption{Posterior std (pCN)}
\end{subfigure}
\hfill
\begin{subfigure}[t]{0.34\textwidth}
    \centering
    \includegraphics[width=\linewidth]{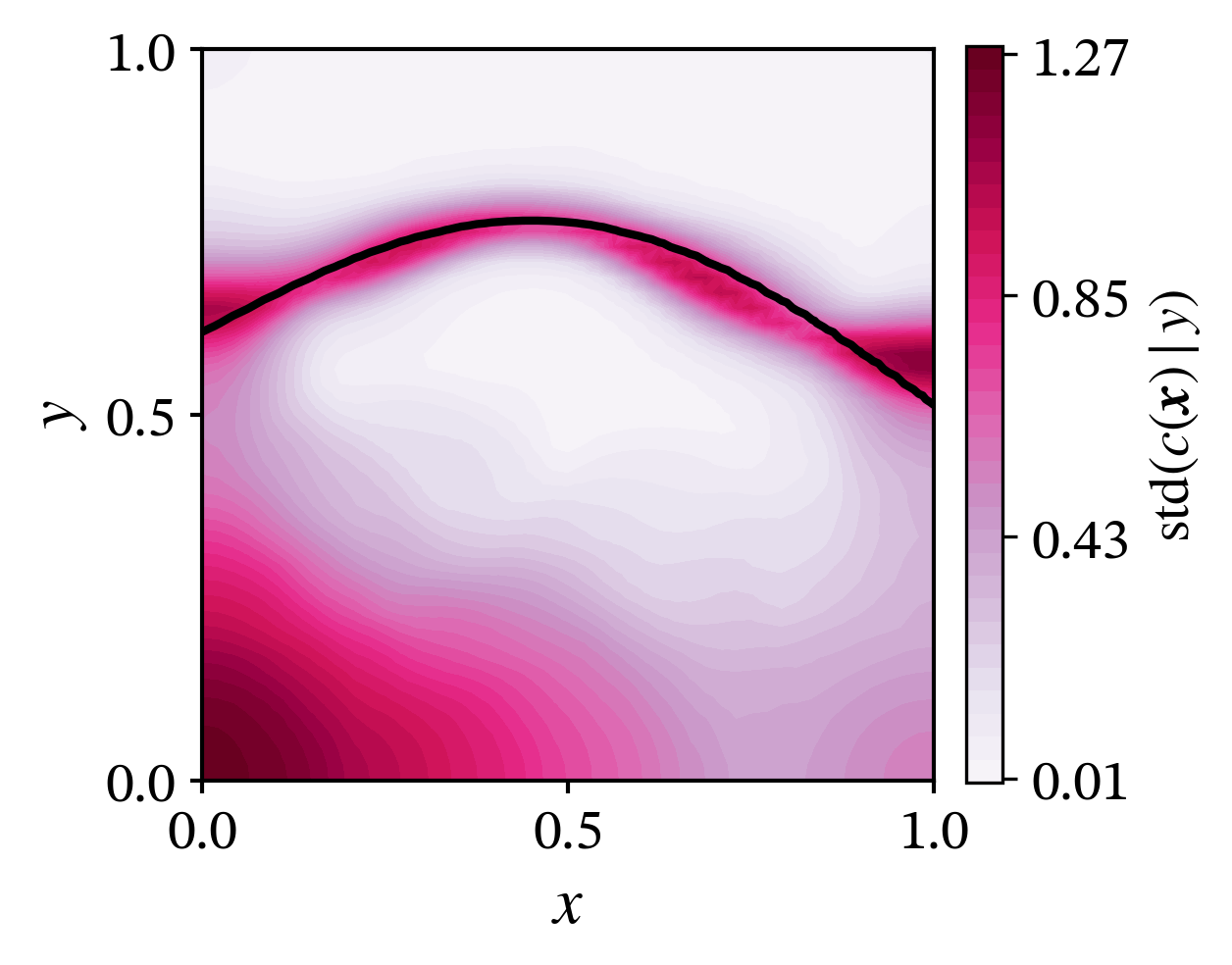}
    \caption{Posterior std (NUTS)}
\end{subfigure}
\caption{Reconstruction results for the known-plateau setting with a curved interface (Example~2). (a) True doping profile, (b) posterior mean obtained using pCN, (c) posterior mean obtained using NUTS, (d) posterior standard deviation (std) obtained using pCN, and (e) posterior standard deviation obtained using NUTS. The black solid line indicates the true pn-junction interface, while the red dashed line indicates the inferred interface.}
\label{fig:ex2_noise5}
\end{figure}

Figs.~\ref{fig:ex2_noise5}(b) and \ref{fig:ex2_noise5}(d) show the posterior mean and standard deviation obtained using pCN. The posterior mean recovers the curved interface, with only mild smoothing effects near the ends of the curve. The posterior standard deviation exhibits a clear spatial structure, with uncertainty concentrated around the interface region and in regions farther from the measurement boundary. This reflects both ambiguity in the junction location and reduced sensitivity of the measurements to distant parts of the domain. The posterior mean yields \(\mathrm{RE}=0.15\) and \(\mathrm{SSIM}=0.81\). The total runtime is approximately \(1.33\) hours. However, the sampler exhibits poor mixing, with mean \(\mathrm{ESS}\approx 288\) and normalized efficiency \(\mathrm{ESS}/N_s \approx 10^{-3}\). The trace plots in Fig.~\ref{fig:Trace_Scen_1}(c) show highly correlated chains confined to local regions of the parameter space. These diagnostics indicate that, although pCN produces visually reasonable posterior summaries, its uncertainty estimates should be interpreted with caution.

Figs.~\ref{fig:ex2_noise5}(c) and \ref{fig:ex2_noise5}(e) show the posterior mean and standard deviation obtained using NUTS. Similar to pCN, the posterior mean recovers the curved interface, with only mild smoothing effects near the ends of the curve. The posterior standard deviation also exhibits a clear spatial structure, with uncertainty concentrated around the interface region and in regions farther from the measurement boundary. Quantitatively, the reconstruction achieves $\mathrm{RE}=0.14$ and $\mathrm{SSIM}=0.84$, representing a slight improvement over pCN. The total runtime is approximately $5.49$ hours. The mean ESS is $978$, corresponding to excellent $\mathrm{ESS}/N_s \approx 0.98$. This represents a substantial increase in sampling efficiency compared to pCN, even in this more challenging setting. The trace plots in Fig.~\ref{fig:Trace_Scen_1}(d) show well-mixed chains without visible stagnation, indicating efficient posterior exploration.

The main numerical observations for Example~2 are summarized in Table~\ref{tab:ex2_sampler_comparison}. Based on the results of the known-plateau setting, we therefore do not consider pCN in the joint reconstruction experiments. Its poor mixing and extremely low effective sample size indicate that, for the present posterior and tuning, pCN is not sufficiently reliable for posterior exploration in this problem.

\begin{table}[t]
\centering
\caption{Comparison of pCN and NUTS for Example~2 with known plateaus.}
\label{tab:ex2_sampler_comparison}
\begin{tabular}{lcccccc}
\hline
Method & RE & SSIM & Runtime (h) & Mean ESS & \(\mathrm{ESS}/N_s\) & Mixing behaviour \\
\hline
pCN  & 0.15 & 0.81 & 1.33 & 288 & \(\approx 10^{-3}\) & highly correlated \\
NUTS & 0.14 & 0.84 & 5.49 & 978 & \(\approx 0.98\) & well mixed \\
\hline
\end{tabular}
\end{table}

%--------------------------------------------------------------
\subsection{Joint Reconstruction}
\label{subsec:joint_reconstruction}
% ---------------------------------------------------------------
We now examine the joint reconstruction of both the interface geometry and the plateau values for the straight junction geometry shown in Fig.~\ref{fig:ex1_noise5}(a). This setting is more challenging than the known-plateau case, since the plateau parameters affect both the interior doping field and the Dirichlet boundary conditions, thereby introducing additional nonlinearity and posterior coupling. The prior parameters are chosen as in Section~\ref{subsec:scenario_A}. Synthetic data are generated with a relative noise level \(\rho=0.05\). Posterior sampling is performed using NUTS with \(500\) warmup iterations followed by \(1000\) posterior samples. The plateau values are assigned independent uniform priors on prescribed intervals,
$c_{\mathrm p} \sim \mathcal U(-2.5,-1.5)$ and $c_{\mathrm n} \sim \mathcal U(0.5,1.5)$.

Figs.~\ref{fig:posterior_field}(a)--(b) show the posterior mean and standard deviation of the reconstructed doping field. The posterior mean recovers the overall structure of the true field, but the inferred interface, shown by the red dashed line, is shifted downward relative to the true interface, shown by the black solid line. The posterior standard deviation is concentrated around the inferred interface, indicating that the dominant uncertainty is associated with the interface location. Away from the interface and closer to the measurement boundary, the uncertainty remains low, reflecting better constrained regions. The reconstruction achieves \(\mathrm{RE}=0.40\) and \(\mathrm{SSIM}=0.56\), illustrating the increased difficulty of the joint inference problem.

Figs.~\ref{fig:posterior_field}(c)--(d) show the marginal posterior distributions of \(c_{\mathrm p}\) and \(c_{\mathrm n}\). Both distributions exhibit skewness and long tails toward the admissible bounds. In particular, \(c_{\mathrm p}\) shifts toward more negative values, with posterior mean and standard deviation \(-2.15 \pm 0.25\), while \(c_{\mathrm n}\) shifts toward larger values, with posterior mean and standard deviation \(1.17 \pm 0.25\). These results indicate that the data do not strongly constrain the plateau magnitudes independently.

The downward shift of the inferred interface appears to be linked to bias in the inferred plateau values. Increasing \(c_{\mathrm n}\) or decreasing \(c_{\mathrm p}\) increases the contrast in the doping profile and affects the boundary flux. To remain consistent with the observed data, the model can partially compensate for these changes by shifting the interface downward. This suggests a non-identifiability mechanism between the interface position and the plateau values. To quantify this effect, we define the interface depth \(d\) as the vertical coordinate of the zero-level set at \(x=0.5\). The posterior samples yield
\[
    \mathrm{corr}(c_{\mathrm p},d)=0.43,
    \qquad
    \mathrm{corr}(c_{\mathrm n},d)=-0.52.
\]
These moderate correlations support the presence of a compensation mechanism between the plateau values and the inferred interface location: variations in the plateau magnitudes can be partially offset by shifts in the interface depth, leading to a family of near-equivalent explanations of the observed boundary flux data.

The main numerical observations for the joint reconstruction experiment are summarized in Table~\ref{tab:joint_reconstruction_summary}. The total runtime is approximately \(10\) hours. The mean effective sample size for the KL coefficients \(\boldsymbol{z}\) is \(862\), corresponding to \(\mathrm{ESS}/N_s=0.86\). The effective sample sizes for the plateau parameters are \(493\) for \(c_{\mathrm p}\) and \(368\) for \(c_{\mathrm n}\). These diagnostics indicate that the sampler remains effective in exploring the posterior in this joint reconstruction setting. The trace plots in Fig.~\ref{fig:joint_trace} show well-mixed chains, indicating that the discrepancy between the posterior mean and the ground truth is not primarily due to sampling inefficiency, but rather to structural non-identifiability. This may also be associated with a more complex posterior structure, possibly including multiple modes, as suggested by the sensitivity of the MAP estimates to initialization in the previous section. Overall, the boundary flux data constrain a combination of plateau values and interface geometry, leading to a family of near-equivalent solutions.

\begin{figure}[!t]
\centering
\begin{subfigure}[t]{0.49\textwidth}
    \centering
    \includegraphics[width=\linewidth]{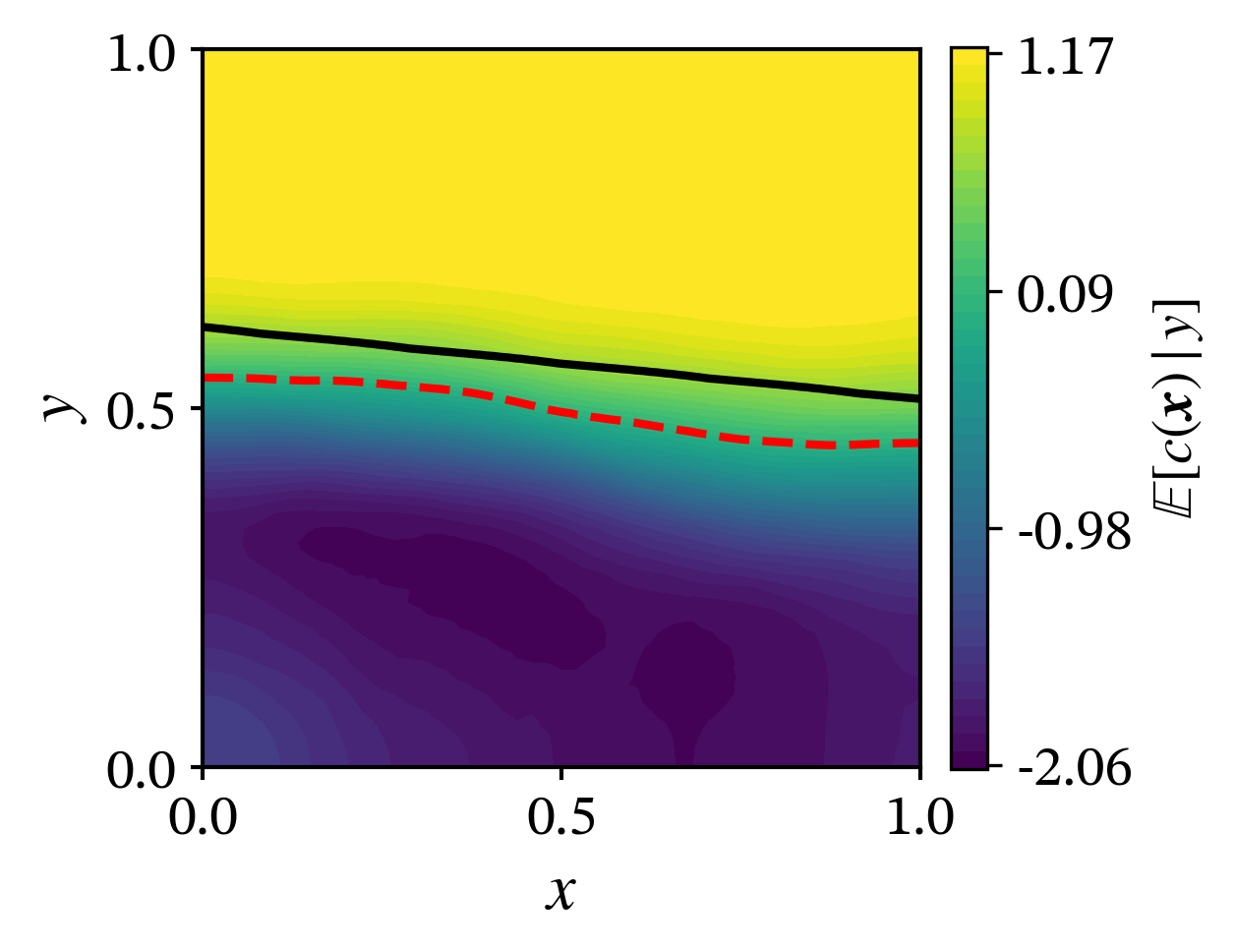}
    \caption{Posterior mean}
\end{subfigure}
\hfill
\begin{subfigure}[t]{0.49\textwidth}
    \centering
    \includegraphics[width=\linewidth]{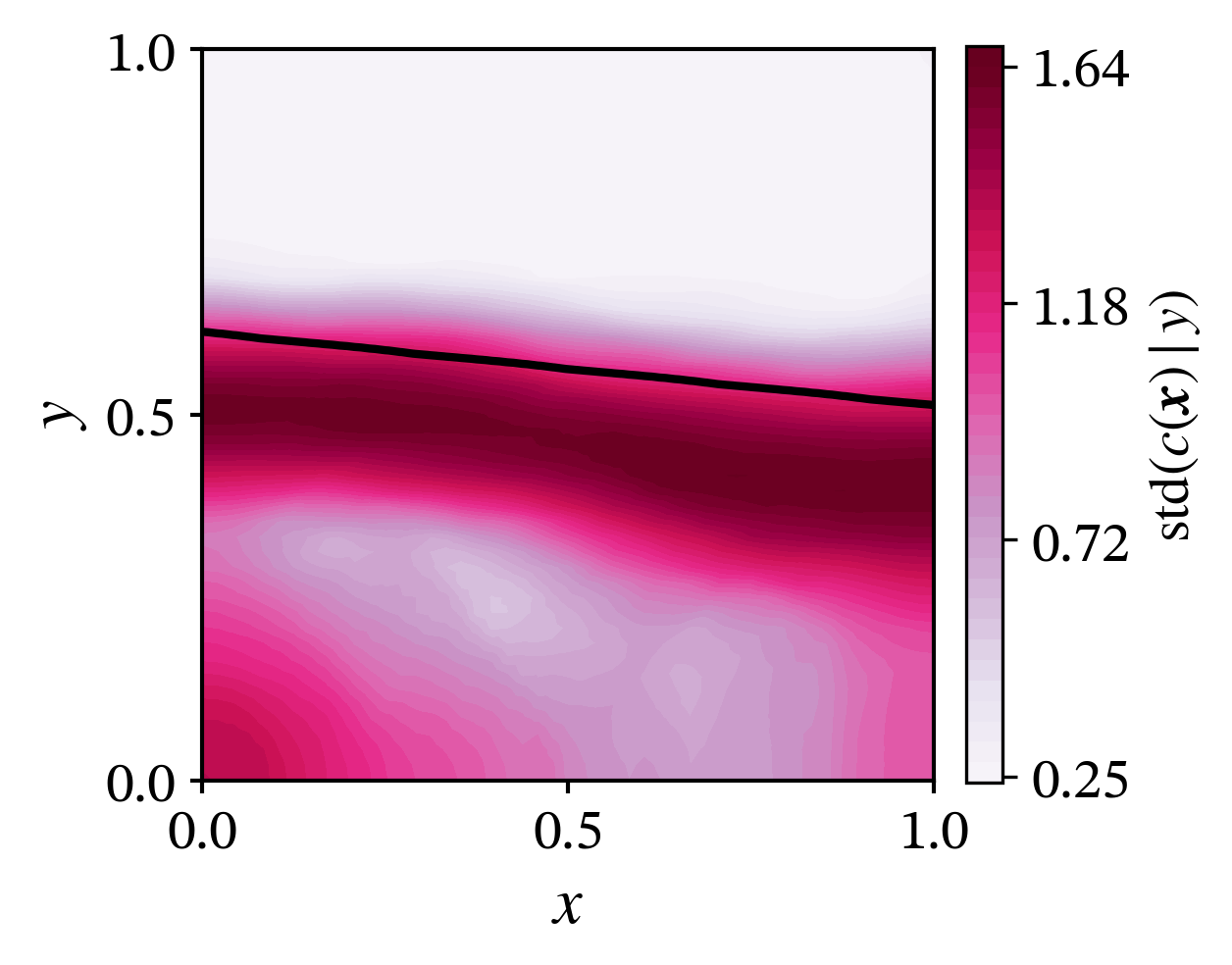}
    \caption{Posterior std}
\end{subfigure}

\vspace{5pt}

\begin{subfigure}[t]{0.45\textwidth}
    \centering
    \includegraphics[width=\linewidth]{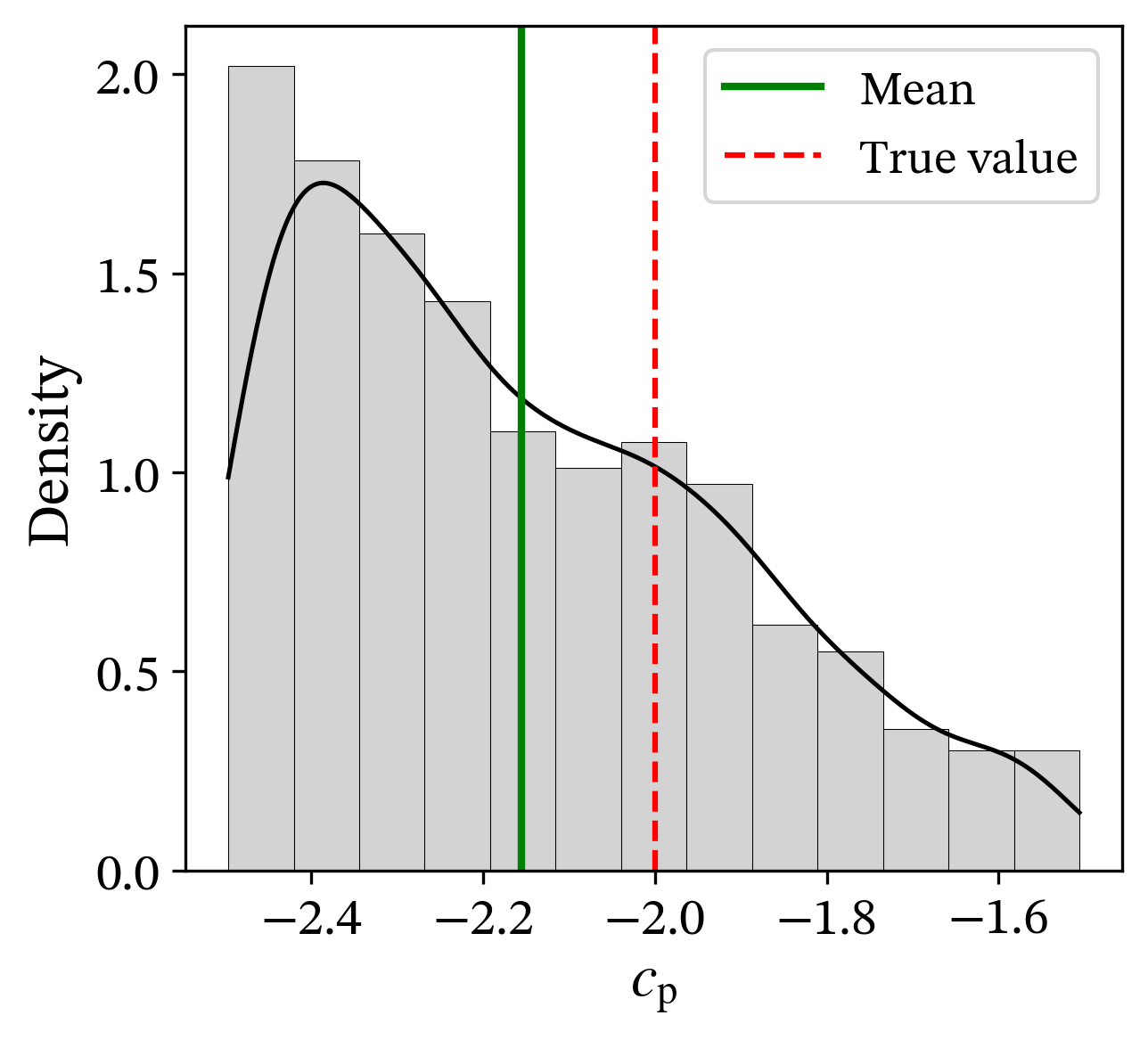}
    \caption{Marginal posterior of $c_{\mathrm{p}}$}
\end{subfigure}
\hfill
\begin{subfigure}[t]{0.45\textwidth}
    \centering
    \includegraphics[width=\linewidth]{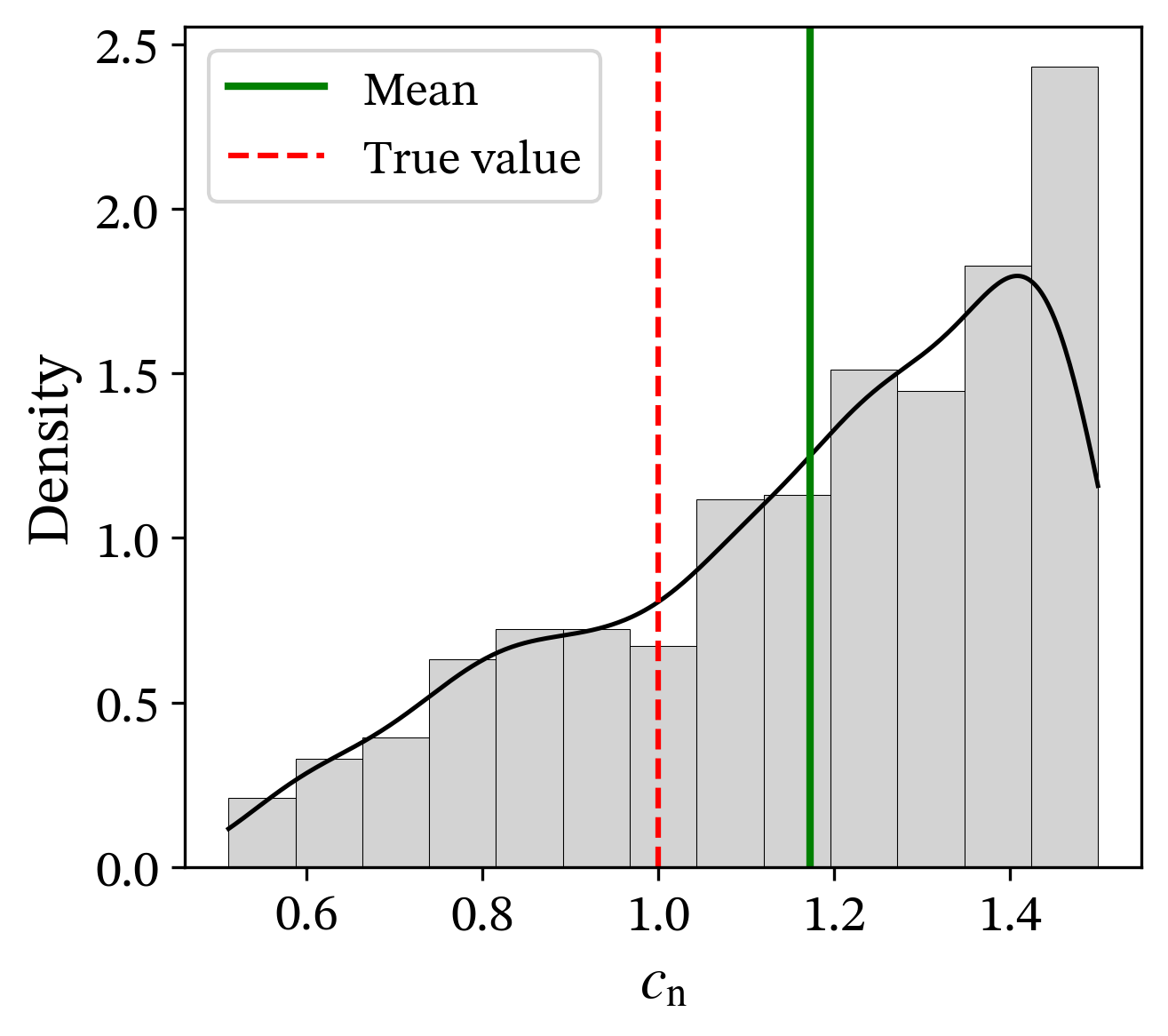}
    \caption{Marginal posterior of $c_{\mathrm{n}}$}
\end{subfigure}

\caption{Joint reconstruction results. (a) Posterior mean and (b) posterior standard deviation of the reconstructed doping field. The black solid line indicates the true pn-junction interface, while the red dashed line indicates the inferred interface. (c)--(d) Marginal posterior distributions of the plateau values \(c_{\mathrm p}\) and \(c_{\mathrm n}\), respectively.}
\label{fig:posterior_field}
\end{figure}

\begin{figure}[!htbp]
\centering
\includegraphics[width=0.65\linewidth]{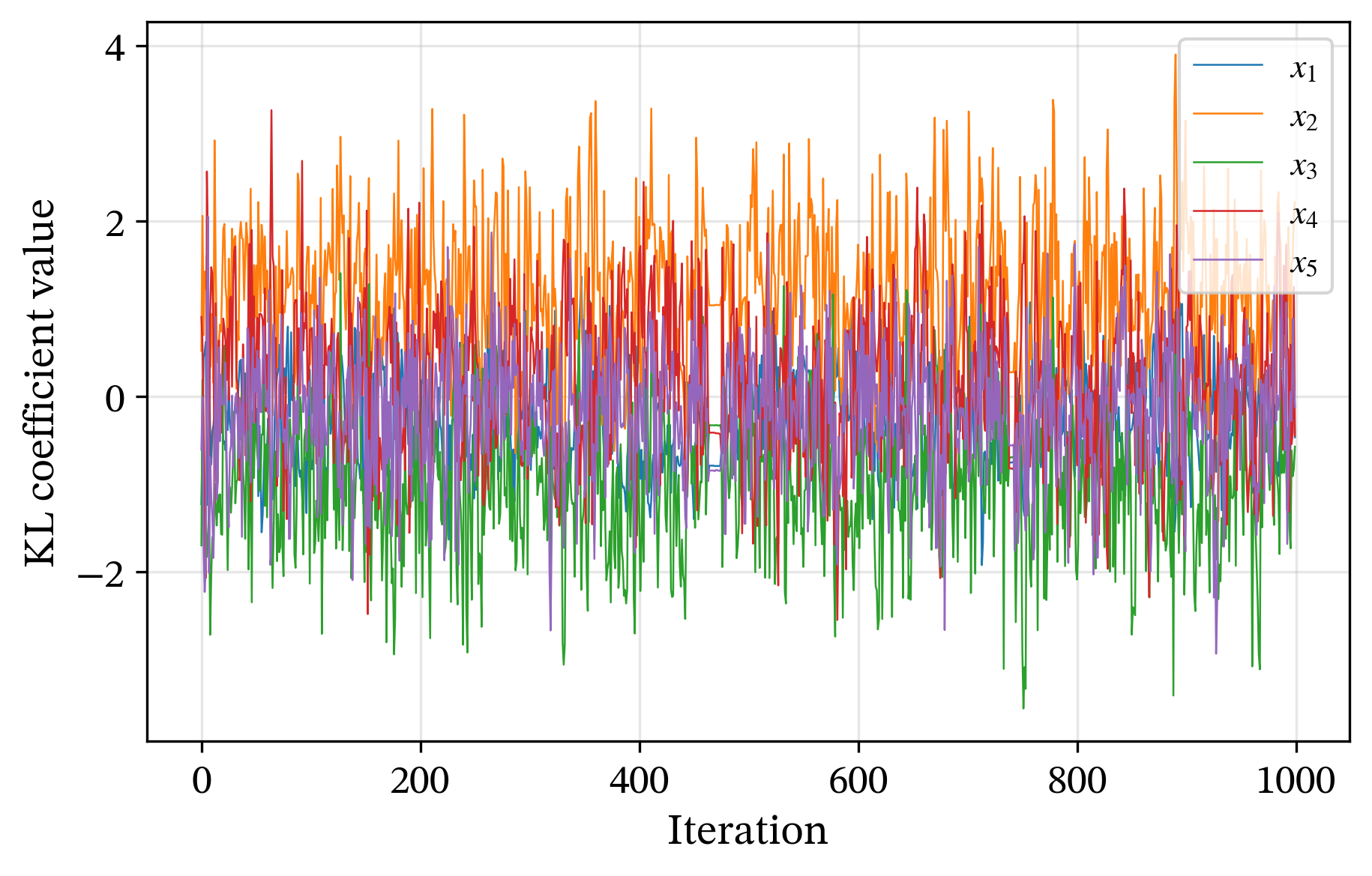}
\caption{Trace plots for representative KL coefficients in the joint reconstruction experiment. The chains exhibit stable mixing without visible stagnation.}
\label{fig:joint_trace}
\end{figure}

\begin{table}[t]
\centering
\caption{Summary of NUTS results for the joint reconstruction setting with \(N_s=1000\) posterior samples.}
\label{tab:joint_reconstruction_summary}
\begin{tabular}{lcccccc}
\hline
Method & RE & SSIM & Runtime (h) & Mean ESS for \(\boldsymbol z\) & ESS \(c_{\mathrm p}\) & ESS \(c_{\mathrm n}\) \\
\hline
NUTS & 0.40 & 0.56 & 10.0 & 862 & 493 & 368 \\
\hline
\end{tabular}
\end{table}

%%%%%%%%%%%%%%%%%%%%%%%%%%%%%%%%%%%%%%%%%%%%%%%%%%%%%%%%%%%%%%%%%%%%%%%%
\section{Conclusions}\label{s:conclusions}
%%%%%%%%%%%%%%%%%%%%%%%%%%%%%%%%%%%%%%%%%%%%%%%%%%%%%%%%%%%%%%%%%%%%%%%%
We developed a Bayesian framework for reconstructing piecewise-constant doping profiles and unknown pn-junction interfaces in semiconductor devices from boundary flux measurements. The approach combines a differentiable sigmoid pushforward prior with adjoint-based gradient computation, enabling NUTS sampling for a nonlinear PDE-constrained inverse problem. The prior construction provides a natural way to represent nearly piecewise-constant doping profiles with sharp interface structure while retaining the differentiability required for gradient-based inference. Well-posedness of the Bayesian formulation is established by proving Lipschitz continuity of the forward map and Hellinger stability of the posterior.

Numerical experiments show that the proposed prior recovers sharp straight and curved interfaces in the known-plateau setting. The posterior standard deviation is concentrated primarily near the reconstructed interface and in regions less constrained by the boundary observations. NUTS yields well-mixed chains with normalized effective sample sizes one to two orders of magnitude larger than pCN. This efficiency gain reflects the compatibility between the differentiable prior, the adjoint-based gradients, and the posterior geometry of this problem, and should not be interpreted as a general property of gradient-based MCMC.

In the joint reconstruction setting, posterior correlations between the plateau parameters and interface depth reveal a structural non-identifiability: changes in doping contrast can be partially compensated by shifts in interface location. This shows that point estimates alone may be misleading and underscores the value of full posterior analysis, including uncertainty fields, parameter correlations, and sampling diagnostics.

The model considered here is a reduced equilibrium approximation and does not include drift-diffusion transport, recombination, generation, or externally driven currents. Therefore, the boundary fluxes should be understood as mathematical observations for a nonlinear elliptic inverse source problem rather than directly measurable diode currents. Extending the framework to stationary or time-dependent drift-diffusion systems is a natural next step.

More broadly, the combination of geometry-aware pushforward priors, adjoint-based gradients, and gradient-informed sampling is applicable to a range of PDE-constrained inverse problems with interface-dominated unknowns, provided that suitable differentiable parameterizations and adjoint formulations are available. In this sense, the semiconductor-inspired problem studied here can serve as a benchmark for Bayesian geometric inversion and uncertainty quantification in nonlinear PDE-based imaging.

%%%%%%%%%%%%%%%%%%%%%%%%%%%%%%%%%%%%%%%%%%%%%%%%%%%%%%%%%%%%%%%%
\appendix
\numberwithin{equation}{section}
\section{Well-Posedness of the Bayesian Reconstruction}\label{s:inverse problem}

In this appendix, we investigate analytical properties of the forward map associated with the nonlinear Poisson--Boltzmann semiconductor model \eqref{eq:semi-PDE-strong}. Our first goal is to study the dependence of the weak solution on the doping profile, which is a source term in the PDE model. For fixed Dirichlet data, we prove that the solution operator is Lipschitz continuous from $L^2(\Omega)$ to $H^1(\Omega)$. For this result and for the unique solvability of the PDE model, we need the following assumption:\\

\noindent {\bf Assumption~1.}\label{assump 1}
 In the semiconductor PDE model~\eqref{eq:semi-PDE-strong},  let $\Omega\subset \R^d, d\leq 3$ be a bounded Lipschitz domain. Assume that 
 \begin{equation}    \epsilon \in L^\infty(\Omega),\qquad 0<\epsilon_0\leq\epsilon(x)\leq\epsilon_1<\infty, \qquad {\text{a.e.}} \,\,{\text{in}}\,\, \Omega,
\end{equation}
and that the doping profile satisfies $c\in L^2(\Omega)$. Furthermore, assume that the Dirichlet data $f_1$ and $f_2$ admit a lifting $g\in H^1(\Omega)$, i.e., 
\begin{equation}\label{eq:lifting}    g|_{\partial\Omega_2}=f_1, \qquad g|_{\partial\Omega_3}=f_2.
\end{equation}

Under Assumption~1 and fixed admissible Dirichlet data, Eq.~\eqref{eq:semi-PDE-strong} defines a well-posed forward problem for semiconductors in the sense that for every $c\in L^2(\Omega)$, the equation admits a unique weak solution $u(c)\in H^1(\Omega)$.

%%%%%%%%%%%%%%%%%%%%%%%%%%%%%%%%%%%%%%%%%%%%%%%%%%%%%%%%%%%%%%%%%%%%%%%%
% {The Well-posedness of the Bayesian  Reconstruction Problem}\label{ss:well-posedness}
\begin{theorem}\label{thm:well-posedness}
    Let Assumption~1 hold. Suppose $c_1, c_2 \in L^2(\Omega)$, and let $u_1$ and $u_2$ denote the corresponding weak solutions to the PDE model~\eqref{eq:semi-PDE-strong}. Then, there exists constant $\eta:=\eta(\epsilon_0,k_P)$%$\eta:=\eta(\epsilon_0^{-1},k_p)$ 
    such that
\begin{equation}\label{eq:Solution-Lipschitz}
     \|u_1 - u_2\|_{H^1(\Omega)}\leq \eta \|c_1 - c_2\|_{L^2(\Omega)},  
    \end{equation}
    where $k_P$ is the Poincaré constant.
%    where $\eta$ depends only on $\epsilon_0$ and the Poincare constant $k_P$.
\end{theorem}
\begin{proof}
    Let $c_1, c_2\in L^2(\Omega)$ and $u_i=u(c_i),\,i=1,2$, denote the corresponding weak
solutions of \eqref{eq:semi-PDE-strong} with the same Dirichlet boundary data together with 
    homogeneous Neumann boundary conditions. We follow \eqref{sec:FEM} for weak solutions and define the solution space $V$ and the subspace $V_0$ for the PDE model.
    Setting $W:=u_1-u_2 \in V_0$, subtracting the two PDEs corresponding to the parameter-solution pairs $(c_1,u_1)$ and $(c_2, u_2)$, % yields
    % \begin{equation*}
    %     -\nabla\cdot (\epsilon\nabla W) + 2\delta^2\big(\sinh(u_1)-\sinh(u_2)\big) = c_1 - c_2,
    % \end{equation*}
   testing the model %equation~\eqref{eq:semi-PDE-strong}  
    with $W\in V_0$, and integrating by parts yields       \begin{equation}\label{eq:Variational Formulation}
        \int_\Omega\epsilon|\nabla W|^2 + 2\delta^2\int_\Omega\big(\sinh(u_1)-\sinh(u_2)\big)W = \int_\Omega (c_1 - c_2)W.
    \end{equation}         %\begin{equation}\label{eq:Variational Formulation}
  %      \int_\Omega\big(\sinh(u_1)-\sinh(u_2)\big)W \geq0.
  %  \end{equation} 
    Using the uniform ellipticity $\epsilon_0\leq \epsilon$, we obtain $\epsilon_0\|\nabla W\|_{L^2}^2 \leq \epsilon\int_\Omega |\nabla W|^2$.  According to the strong monotonicity of the nonlinearity of the PDE model, i.e., $\sinh(u_i)$, we have %$\big(\sinh(x)-\sinh(y)\big)(x-y)\geq0,\,\forall x,y\in\R$, which leads to 
    $     \int_\Omega\big(\sinh(u_1)-\sinh(u_2)\big)W \geq0$.   This leads to the variational formulation \begin{equation}\label{eq:Variational Formulation2}
      \epsilon_0\|\nabla W\|_{L^2}^2 \leq \int_\Omega (c_1 - c_2)W, 
    \end{equation} 
    dropping the nonnegative nonlinear term in~\eqref{eq:Variational Formulation}.
By applying the Cauchy-Schwarz inequality, we get
    \begin{equation}\label{eq:RHS estimation}
      \int_\Omega (c_1 - c_2)W\leq\|c_1 - c_2\|_{L^2} \| W\|_{L^2}
       \end{equation} 
    % \leq 2\delta^2\| W\|_{L^2}^2  + \frac{1}{2\delta^2}\|c_1 - c_2\|_{L^2}^2. 
 % \end{equation} 
    Then by using the Poincaré inequality and substituting into \eqref{eq:Variational Formulation2}, we obtain   \begin{equation}\label{eq:Variational Formulation3}
      \epsilon_0\|\nabla W\|_{L^2}^2 \leq \|c_1 - c_2\|_{L^2}k_P\| \nabla W\|_{L^2}, 
    \end{equation} 
   leading to $\|\nabla W\|_{L^2} \leq \frac{k_P}{\epsilon_0}\|c_1 - c_2\|_{L^2}$. On the other hand, using the Poincaré inequality and the definition of the $H^1$-norm, we get $\|W\|_{H^1} \leq \sqrt{1+k_P^2}\|\nabla W\|_{L^2}$. Thus, we obtain
     \begin{equation}\label{eq:final estimation}
     \| W\|_{H^1(\Omega)}\leq \eta\|c_1 - c_2\|_{L^2(\Omega)}, 
    \end{equation}
    where $\eta:=\frac{k_P}{\epsilon_0}\sqrt{1+k_P^2}$ depending only on $\epsilon_0$ and the Poincaré constant $k_P$.
    %, which 
    %results in the global Lipschitz continuity of the solution with respect to the parameter.
\end{proof}
The Lipschitz continuity result is a direct consequence of the uniform ellipticity of the diffusion operator and the monotonicity of the nonlinear term in the PDE model. Thus, the forward map is continuous, ensuring the stability of the model with respect to perturbations in the doping profile. This continuity property plays an essential role in the Bayesian inversion framework, as it is used to establish the well-posedness of the posterior measure and its stability with respect to observations. Furthermore, it provides the analytical foundation for the application of function-space MCMC methods and gradient-based sampling algorithms in the infinite-dimensional setting.

The following lemma states the existence of an a priori estimate for the solution of the Poisson-Boltzmann PDE model for semiconductors:

\begin{lemma}\label{lemma:a-priori estimate}
Let $c\in L^2(\Omega)$ and let $u\in H^1(\Omega)$ be the weak solution of
\eqref{eq:semi-PDE-strong}. Under Assumption~1, there exists a constant
$C>0$, independent of $c$, such that
\[
\|u\|_{H^1(\Omega)}
\le
C\Bigl(
\|c\|_{L^2(\Omega)}
+
\|g\|_{H^1(\Omega)}
\Bigr),
\]
where $g\in H^1(\Omega)$ is a lifting of the prescribed Dirichlet data (see~\eqref{eq:lifting}).
\end{lemma}

\begin{proof}
Testing the equation with $w:=u-g$ with $w\in V_0$ and using the monotonicity of the nonlinearity term $\sinh$, as well as the ellipticity, and the Poincaré inequality yields the result.
\end{proof}

\begin{remark}[Concentration-dependent Dirichlet data]
The stability result in Theorem~\ref{thm:well-posedness} is stated for fixed Dirichlet data. If the boundary potentials are instead induced by scalar plateau concentrations, i.e.,
\[
f_p=\operatorname{arcsinh}\left(\frac{c_p}{2\delta^2}\right),
\qquad
f_n=\operatorname{arcsinh}\left(\frac{c_n}{2\delta^2}\right),
\]
then an analogous estimate holds, provided the corresponding liftings depend Lipschitz-continuously on these boundary data. In particular, since \(s\mapsto \operatorname{arcsinh}(s/(2\delta^2))\) is Lipschitz for fixed \(\delta>0\), the forward map remains stable with respect to perturbations in both the bulk doping profile and the scalar plateau concentrations.
\end{remark}

%%%%%%%%%%%%%%%%%%%%%%%%%%%%%
\subsection{Setting} 

For the rest of this section, we assume that $(X, \mathcal{A}, \mu)$ is a measure space, where $X$ is a set, $\mathcal{A}$ a $\sigma$-algebra on $X$, and $\mu$ a measure which is defined on $\mathcal{A}$.
Furthermore, we define the space 
${\text{meas}}(X, \nu) := \{\mu\,\, {\text{measure on}}\,\, X : \mu \ll \nu\}$,
where $\mu \ll \nu$ means that $\mu$ is absolutely continuous with respect to $\nu$. 

Suppose that the prior measure is given by $\mu_0$, and data $y$  by~\eqref{eq:forward-model}, which is corrupted by Gaussian noise $\varepsilon\sim \mathcal{N}(0,\Sigma_\varepsilon)$. Further assume that $X$ and $Y$ are respectively the parameter and data spaces, and the data negative-log-likelihood function $J\colon X\times Y\to \R$ is given by 
\begin{equation}\label{eq:neg-log-like}
   J(c;y):=\frac{1}{2}\|y-\mathcal{G}(c)\|^2_{\Sigma_\varepsilon},
\end{equation}
where $\|\cdot\|_{\Sigma_\varepsilon}:=\|\Sigma_\varepsilon^{-1/2} \cdot\|_2$. Then, using Bayes' theorem, the posterior measure $\mu^y$ is the conditional distribution of the parameter $c$, given the data $y$, and it is defined as $\mu^y(d c)\propto \exp{(-J(c;y))}\mu_0(d c)$. Since, in infinite-dimensional spaces, the prior measure \(\mu_0\) does not admit a density with respect to a Lebesgue measure, the posterior is defined through the Radon--Nikodym derivative with respect to \(\mu_0\):
\begin{equation}\label{eq:posterior}
  \frac{d \mu^y}{d \mu_0}(c)=\frac{1}{z(y)}\exp{(-J(c;y))},  
\end{equation}
where $z(y):=\int_X \exp(-J(c;y)) d\mu_0(c)$.

We recall (from Section~\ref{ss:doping}) the forward operator 
$\mathcal G := \mathcal O \circ \mathcal W
$,
which maps the doping function $c$ to the boundary measurements.
\begin{theorem}\label{posterior-wellposedness}
 Under Assumption~1, let the solution operator $\mathcal{W}\colon c(\boldsymbol{x})\mapsto u(\boldsymbol{x})$ satisfy the Lipschitz estimate~\eqref{eq:Solution-Lipschitz}, and the observation operator $\mathcal O$ be bounded and linear. Then, 
 \begin{enumerate}
     \item the semiconductor forward operator $\mathcal{G}=\mathcal{O}\circ \mathcal W$
     %(c)=\mathcal{O}(u(c))$
     is Lipschitz continuous, and the associated Bayesian doping inverse problem, described in Section~\ref{ss:doping}, admits a well-defined posterior measure given by \eqref{eq:posterior}, which is absolutely continuous with respect to the prior measure $\mu_0$, i.e., $\mu^y\ll\mu_0$, and

     \item the posterior is stable with respect to the data in the Hellinger metric, i.e, for every $r > 0$, there exists $C(r) > 0$ such that for all data $y_1, y_2 \in Y$ with
$\max\{\|y_1\|_Y , \|y_2\|_Y \} < r$, it holds 
\begin{equation}\label{eq:Hellinger stability}
  d_H(\mu^{y_1},\mu^{y_2}) \leq  C(r) \|y_1-y_2\|_Y. 
 \end{equation}
 \end{enumerate}

\end{theorem}

\begin{proof}
  \begin{enumerate}
     \item Using Theorem~\ref{thm:well-posedness} and due to the fact that   the observation map $\mathcal{O}\colon H^1(\Omega)\to \R^{N_{\text{meas}}}$ is bounded linear, we obtain
     \begin{equation}
       \big|\mathcal{G}(c_1) - \mathcal{G}(c_2)\big|\leq C\|u(c_1) - u(c_2)\|_{H^1}\leq C\|c_1 - c_2\|_{L^2}, 
     \end{equation}
i.e., the forward operator is Lipschitz continuous. Furthermore, we need to prove the continuity of the misfit function. To this end, we assume the negative-log-likelihood function $J$, defined in equation~\eqref{eq:neg-log-like} and, for any fixed data $y\in Y=\R^{N_{\text{meas}}}$, we prove the stability of this function with respect to the parameter $c$; 
assuming $c_i \in X$ and using the identity \eqref{eq:neg-log-like}, we obtain
\begin{equation*}
J(c_1;y)-J(c_2;y)
=
\frac12 \langle \Sigma_\varepsilon^{-1}(2y-\mathcal G(c_1)-\mathcal G(c_2)),
\mathcal G(c_1)-\mathcal G(c_2)\rangle.
\end{equation*}
Applying Cauchy--Schwarz inequality and boundedness of $\Sigma_\varepsilon^{-1}$ yields
\begin{equation*}
|J(c_1;y)-J(c_2;y)|
\le C \|2y-\mathcal G(c_1)-\mathcal G(c_2)\|
\|\mathcal G(c_1)-\mathcal G(c_2)\|
\end{equation*}
and thus
$\|\mathcal G(c)\|\le C(1+\|c\|)$, we conclude
\begin{equation}\label{eq:contin. of NLL}
|J(c_1;y)-J(c_2;y)|
\le C(1+\|c_1\|_{L^2}+\|c_2\|_{L^2})\|c_1-c_2\|_{L^2}.
\end{equation}
In particular, $J(\cdot;y)$ is locally Lipschitz and measurable, and $\exp\big(- J(c;y)\big)$ is integrable with respect to the Gaussian prior measure $\mu_0$. Thus, the normalization constant $z(y):=\int_X \exp(-J(c;y)) d\mu_0(c)$ satisfies $0< z(y)<\infty$, and the posterior~\eqref{eq:posterior} is well-defined.

\item For the second part of the proof, we prove the Hellinger stability of posterior measure $\mu^y$; first, we note that, using a similar proof of \eqref{eq:contin. of NLL}, we obtain the continuity of the negative-log-likelihood function with respect to the data, i.e., 
for every $r > 0$, there exists $C=C(r) > 0$ such that for all data $y_1, y_2 \in Y$ and all $c\in X$ with
$\max\{\|c\|_X, \|y_1\|_Y , \|y_2\|_Y \} < r$, it holds 
\begin{equation}\label{eq:continuity of NNL in data}
  \big|J(c;y_1) - J(c;y_2)\big| \leq  C (1+\|c\|_X) \|y_1-y_2\|_Y. 
 \end{equation}
To show equation~\eqref{eq:Hellinger stability}, we define $z_i:=z(y_i)$ and $J_i:=J(c;y_i)$ for $i=1,2$, as well as
\[
f_i(c):= \frac{1}{\sqrt{z_i}}\exp\left(-\frac12J(c;y_i)\right),
\qquad i=1,2.
\]
We introduce the Hellinger distance $d_H$ between the two posterior measures $\mu^{y_1}$ and $\mu^{y_2}$ corresponding to data $y_1$ and $y_2$ by
\begin{align}\label{eq:Hellinger posterior stability}
 d_H^2 (\mu^{y_1}, \mu^{y_2}) &:= \frac12 \int_X \Big(\sqrt{\frac{d\mu^{y_1}}{d\mu_0}} - \sqrt{\frac{d\mu^{y_2}}{d\mu_0}}  \Big)^2 d\mu_0 = \frac12\int_X \Big|f_1 - f_2\Big|^2 d\mu_0. 
\end{align}
where
 \begin{align}\label{eq:Hellinger posterior stability2}
 \Big|f_1 - f_2\Big| \leq \Big| \frac{1}{\sqrt{z_1}} - \frac{1}{\sqrt{z_2}}\Big|
 \exp{(-\frac12J_1)} + 
  \frac{1}{\sqrt{z_2}}  \Big|\exp{(-\frac12J_1)}- \exp{(-\frac12J_2)}\Big| =: I_1+ I_2. 
 \end{align}
Recall $z(y):=\int_X \exp{(-J(c;y))} d\mu_0(c)$. The mean-value theorem gives 
 \begin{equation}
     \Big|e^{-J_1}- e^{-J_2}\Big| \leq C (1+\|c\|_X) \|y_1-y_2\|_Y
 \end{equation}
and integrating with respect to $\mu_0$ yields
\begin{equation}\label{eq:z1-z2}
     \Big|z_1-z_2\Big| \leq C \|y_1-y_2\|_Y \int_X(1+\|c\|_X) d\mu_0(c)\leq C \|y_1-y_2\|_Y,
 \end{equation}
as a Gaussian prior has finite moments.
Therefore, using \ref{eq:z1-z2}, we find a bound for $I_1$ as following
 \begin{equation}
     \Big|\frac{1}{\sqrt{z_1}}-\frac{1}{\sqrt{z_2}}\Big| e^{-\frac12J_1}\leq C \Big|z_1-z_2\Big| \leq C\|y_1-y_2\|_Y ,
 \end{equation}
 and using the mean-value theorem and the estimate \ref{eq:continuity of NNL in data}, a bound for $I_2$ as following:
  \begin{equation}
    \frac{1}{\sqrt{z_2}}\Big|e^{-\frac12J_1}- e^{-\frac12J_2}\Big|\leq C\Big|J_1-J_2\Big| \leq C (1+\|c\|_X) \|y_1-y_2\|_Y.
 \end{equation}
This leads to 
 \begin{equation}
   \Big|f_1-f_2\Big|\leq  I_1+I_2\leq   C (1+\|c\|_X) \|y_1-y_2\|_Y
 \end{equation}
 and thus
 \begin{align}\label{eq:Hellinger posterior stability4}
 d_H^2 (\mu^{y_1}, \mu^{y_2}) \leq  \frac12\int_X \Big|f_1 - f_2\Big|^2 d\mu_0 \leq C \|y_1-y_2\|^2_Y \int_X(1+\|c\|^2_X) d\mu_0(c)
 % \frac12 \int_X (I_1+I_2)^2 d\mu_0 &\leq C \|y_1 - y_2\|_Y^2 \int_X \|c\|_X^2 d\mu_0(c)\\ \nonumber
 \end{align}
 and
 $$d_H (\mu^{y_1}, \mu^{y_2})\leq C \|y_1 - y_2\|_Y,$$
which completes the proof.
\end{enumerate}
\end{proof}

 % The above Theorem shows that the Bayesian posterior measure for the diode model
 % The Bayesian inverse problem, even with a nonlinear level-set (sigmoid) parameterization, remains well-posed. 
 % This ensures robustness of the Bayesian formulation with respect to measurement noise and provides a rigorous foundation for subsequent numerical sampling methods.

 The above theorem shows that the Bayesian inverse problem for the diode model remains well-posed under the proposed nonlinear sigmoid pushforward parameterization of the doping profile. This establishes stability of the posterior with respect to perturbations in the measurement data and provides a rigorous foundation for the numerical sampling methods used in this work.

\section*{Declarations}

\noindent\textbf{Competing interests}
The authors declare that they have no competing interests.

\noindent\textbf{Data availability}
The data used in the numerical experiments are synthetically generated as described in the manuscript.

\noindent\textbf{Code availability}
The source code and example scripts are available at \url{https://github.com/hassanyazdanian/adjoint-bayesian-semiconductor}.

\noindent\textbf{Author contributions}
Hassan Yazdanian, Leila Taghizadeh, and Babak Maboudi Afkham contributed to the study conception and design. Hassan Yazdanian developed the numerical implementation and performed the experiments. Leila Taghizadeh and Hassan Yazdanian prepared the first draft of the manuscript. Leila Taghizadeh and Babak Maboudi Afkham contributed to the mathematical formulation, interpretation of the results, and manuscript revision. All authors read and approved the final manuscript.

%%%%%%%%%%%%%%%%%%%%%%%%%%%%%%%%%%%%%%%%%%%%%%%%%%%%%%%%%%%%%%%%%%%%%%%
\bibliographystyle{unsrt}
\bibliography{References}
%%%%%%%%%%%%%%%%%%%%%%%%%%%%%%%%%%%%%%%%%%%%%%%%%%%%%%%%%%%%%%%%%%%%%%%%
%%%%%%%%%%%%%%%%%%%%%%%%%%%%%%%%%%%%%%%%%%%%%%%%%%%%%%%%%%%%%%%%%%%%%%%%
\end{document}